\documentclass[11pt]{article}
\usepackage{amsmath,amssymb,mathrsfs,amsthm,amsfonts}
\usepackage{graphics, setspace}
\usepackage{times}
\usepackage{enumerate}
\usepackage{epsfig}
\usepackage{graphicx}
\usepackage[section]{placeins}

\newlength{\defbaselineskip}
\newcommand{\setlinespacing}[1]%
           {\setlength{\baselineskip}{#1 \defbaselineskip}}
           \allowdisplaybreaks

\newtheorem{lemma}{Lemma}[section]
\newtheorem{theorem}{Theorem}[section]
\newtheorem{definition}{Definition}[section]
\newtheorem{coro}{Corollary}[section]

\newtheorem{remark}{Remark}[section]

\newtheorem{conjecture}{Conjecture}[section]
\newtheorem{assumption}{Assumption}

\newcommand{\RR}{{\mathbb R}}
\newcommand{\ZZ}{{\mathbb Z}}

\newcommand{\D}{{\mathcal D}}
\newcommand{\C}{{\mathcal C}}

\newcommand{\sP}{{\cal P}}

\newcommand{\sS}{{\cal S}}
\newcommand{\sT}{{\cal T}}

\newcommand{\sM}{{\cal M}}

\newcommand{\sO}{{\cal O}}

\newcommand{\ra}{\rightarrow}
\newcommand{\Ra}{\Rightarrow}
\newcommand{\ua}{\uparrow}
\newcommand{\da}{\downarrow}
\newcommand{\deq}{\stackrel{\rm d}{=}}
\newcommand{\beql}[1]{\begin{equation}\label{#1}}
\newcommand{\eqn}[1]{(\ref{#1})}
\newcommand{\beq}{\begin{displaymath}}
\newcommand{\eeqno}{\end{displaymath}}

\newcommand{\af}{\alpha}

\newcommand{\lm}{\lambda}
\newcommand{\ep}{\epsilon}

\newcommand{\qandq}{\quad\mbox{and}\quad}

\newcommand{\qforq}{\quad\mbox{for }}
\newcommand{\qorq}{\quad\mbox{or }}
\newcommand{\qasq}{\quad\mbox{as ~}}

\newcommand{\qinq}{\quad\mbox{in ~}}
\newcommand{\qforallq}{\quad\mbox{for all }}

\newcommand{\bes}{\begin{equation*}}
\newcommand{\ees}{\end{equation*}}
\newcommand{\bequ}{\begin{equation}}
\newcommand{\eeq}{\end{equation}}
\newcommand{\bi}{\begin{itemize}}
\newcommand{\ei}{\end{itemize}}
\newcommand{\bsplit}{\begin{split}}
\newcommand{\esplit}{\end{split}}
\newcommand{\bea}{\begin{eqnarray}}
\newcommand{\eea}{\end{eqnarray}}
\newcommand{\beas}{\begin{eqnarray*}}
\newcommand{\eeas}{\end{eqnarray*}}
\newcommand{\btab}{\begin{tabular}}
\newcommand{\etab}{\end{tabular}}

\newcommand{\barq}{\bar{Q}}
\newcommand{\barz}{\bar{Z}}
\newcommand{\barm}{\bar{M}}

\newcommand{\barx}{\bar{X}}

\newcommand{\1}{\textbf{1}}

\def\lm{\lambda}
\def\tinf{\rightarrow\infty}

\def\SS{\mathbb{S}}

\def\I{\mathcal{I}}

\def\T{\mathcal{T}}

\def\mA{\mathcal{A}}

\title{{A Switching Fluid Limit of a Stochastic Network \\
Under a State-Space-Collapse Inducing Control with Chattering}}
\author{Ohad Perry \qandq Ward Whitt}

\begin{document}

\maketitle

\begin{abstract}
\singlespacing
Routing mechanisms for stochastic networks are often designed to
produce {\em state space collapse} (SSC) in a heavy-traffic limit, i.e.,
to confine the limiting process to a lower-dimensional subset of its full state space.
In a fluid limit, a control producing asymptotic SSC corresponds to an ideal {\em sliding mode} control
that forces the fluid trajectories to a lower-dimensional {\em sliding manifold}.
Within deterministic dynamical systems theory, it is well known that sliding-mode controls can cause the system
to chatter back and forth along the sliding manifold due to delays in activation of the control.
For the prelimit stochastic system, chattering implies fluid-scaled fluctuations that are larger than typical stochastic fluctuations.

In this paper we show that chattering can occur in the fluid limit of a controlled stochastic network
when inappropriate control parameters are used.
The model has two large service pools operating under the
{\em fixed-queue-ratio with activation and release thresholds} (FQR-ART) overload control which we proposed in a recent paper.
The FQR-ART control is designed to produce asymptotic SSC by automatically activating sharing
(sending some customers from one class to the other service pool)
once an overload occurs. 
We have previously shown that this control can be effective, even if the service rates are less for the other shared customers,
if the control parameters are chosen properly.
We now show that, if the control parameters are not chosen properly, then
 delays in activating and releasing the control can cause chattering
with large oscillations in the fluid limit.
In turn, these fluid-scaled fluctuations lead to severe congestion,
even when the arrival rates are smaller than the potential total service rate in the system,
a phenomenon referred to as {\em congestion collapse}.
We show that the fluid limit can be a bi-stable switching system possessing
a unique nontrivial periodic equilibrium, in addition to a unique stationary point.
\end{abstract}

\setcounter{section}{0}    

\section{Introduction} \label{secIntro}

\paragraph{State Space Collapse, Sliding Motion and Chattering.} 

Asymptotic state space collapse (SSC) in heavy-traffic limits is often a key step in
developing effective (e.g., asymptotically optimal) controls for multidimensional stochastic networks;
e.g., \cite{bramsonSSC,GW-Convex,GW-QIR,PW13,Reiman84,SSS04,Stolyar04,WilliamsSSC}.
(Related ideas date back to \cite{W71}, but the systems there are uncontrolled.)
As the term suggests, SSC means that the
limit process is of a lower dimension than the prelimit process.
More precisely, if SSC holds, then the limit process ``collapses'' (i.e., is confined) to a
lower dimensional subset of its full state space.
It is significant that SSC is often not only a mathematical tool that is employed to simplify asymptotic analysis, but rather, as in \cite{PW13},
SSC {\em may be a goal} of the control.
See also page 136 in \cite{AsmussenBook}.


In the context of a {\em functional weak law of large numbers} (FWLLN) or {\em fluid limit},
asymptotic SSC corresponds to the limiting deterministic fluid process exhibiting a {\em sliding motion}, i.e.,
all the fluid trajectories ``slide'' on a lower-dimensional subspace, called a {\em sliding manifold};
see, e.g., \S 14.1 in \cite{Khalil} and \S 1.2.3 in \cite{LiberzonBook}.
In such cases, the fluid limit often has discontinuous dynamics in its full state space; i.e., it is governed by
an {\em ordinary differential equation} (ODE) with a discontinuous right-hand side.
The discontinuous dynamics is often avoided by assuming that the initial condition is asymptotically on the sliding manifold
and restricting attention to the behavior of the limit on that region of the state space.
However, if the initial condition of the fluid limit is not on the sliding manifold, the fluid trajectory must first go through
a transient period before reaching the manifold; see Theorem 3 in \cite{bramsonSSC} and the explanation preceding it.

An effective SSC control must therefore (i) pull the system to the sliding manifold without undue delay
and (ii) 
ensure that the system remains on the sliding manifold thereafter.
For queueing networks, this may require specifying different routing rules
for different regions of the state space - on and off the sliding manifold.
For example, suppose that the state space $\SS$ can be partitioned into three disjoint subsets
$\sM$, $\sM^+$ and $\sM^-$, where
$\sM$ is a sliding manifold, while
$\sM^+$ and $\sM^-$ are ``above'' and ``below'' $\sM$.
A sliding-mode control starting in $\sM^-$ may move upwards toward $\sM$, and
move downwards toward $\sM$ from $\sM^+$.
Ideally, a sliding-mode control that starts in $\sM^-$ will switch immediately once the fluid trajectory hits $\sM$,
aiming to keep that trajectory sliding on $\sM$
after that hitting time.
In reality, however, there may be a delay period until the control switches, so that the trajectory
will cross immediately into $\sM^+$ after hitting $\sM$. Once the control finally switches, the trajectory is in $\sM^+$
and the trajectory reverses its direction towards $\sM$, but
may again cross $\sM$, this time into $\sM^-$, because of delays in switching the control.
This is the {\em chattering} phenomenon in the control literature; see \S 14.1 in \cite{Khalil}.
When this chattering occurs, the sliding manifold $\sM$ becomes a {\em switching manifold}, because the system switches its dynamics
each time it crosses $\sM$.
Figure \ref{figChat} depicts a schematic representation of chattering about a manifold $\sM$, denoted by the dashed line, in the two-dimensional plane.
\begin{figure}[h!]
\begin{center}
\includegraphics[width=6cm]{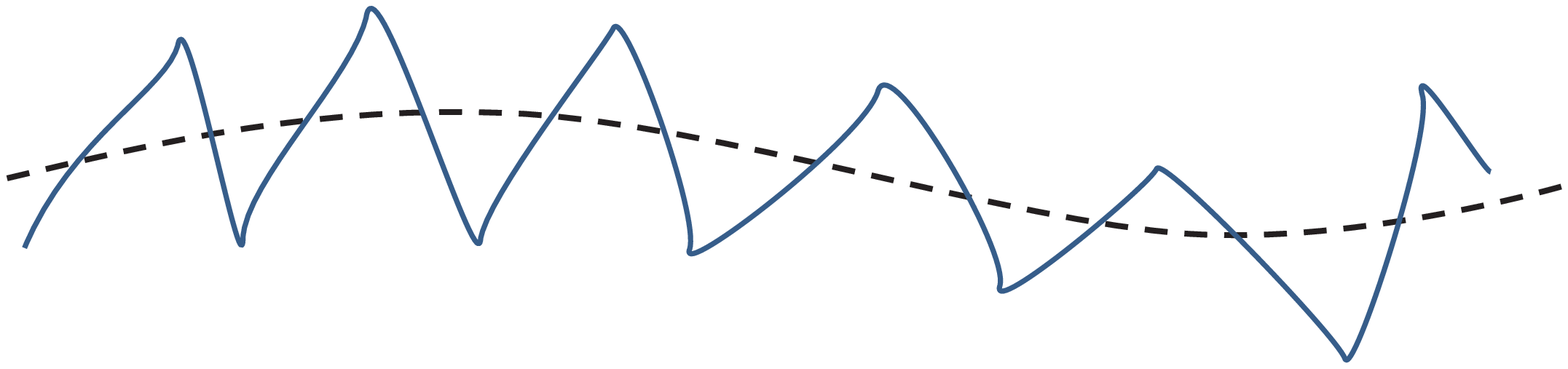}
\caption{Chattering about $\sM$}
\end{center}
\label{figChat}
\end{figure}
\vspace{-.9cm}

\paragraph{The Setting.}
In this paper we illustrate the chattering phenomenon in a queueing network.
Specifically, we consider a deterministic fluid approximation arising in the many-server heavy-traffic limit
for a system with two service pools, each having its own arrival process and designated queue,
that is operating under the {\em fixed-queue-ratio with activation and release thresholds} (FQR-ART) overload control
which we suggested in \cite{PW14}.
Normally, the two pools process work from their designated queues only.
However, when an overload occurs due to an unexpected shift in the arrival rates,
the control automatically identifies which queue should receive help and sharing begins, so that
jobs from the overloaded queue are routed to both service pools,
according to a routing rule that will be specified below.

Since the model was motivated by applications to call centers,
we consider the system to be a call center with two large pools of agents and two associated customers classes,
and refer to customers that are served in the other (not their designated) pool as ``shared customers''.
When sharing is activated, the goal is to maintain the two queues nearly fixed at a pre-specified ratio
that is optimal in a fluid approximation during overload periods; see \cite{PW09}.

We showed that sharing can be effective even if sharing is inefficient, i.e., the shared customers are served at a slower rate.
Since there is the possibility of performance degradation if there is too much sharing,
it is necessary to choose the control parameters appropriately.
The root cause of the chattering discussed here is indeed the combination of excessive inefficient sharing
and poorly chosen control parameters.
To avoid excessive simultaneous sharing of customers
in both directions (``two-way sharing,''see \S 4.1 in \cite{PW09}),
sharing with pool $1$ helping queue $2$ is activated only if
the number of shared customers in pool $2$ is below a certain (small) threshold, and similarly in the other direction.
This latter restriction can cause delays in activating sharing when the direction of overload switches.
Once activated, the control aims to produce asymptotic SSC by confining the queues to
a certain region of the state space in the fluid limit \cite{PW13}.
In the fluid limit, this SSC translates to sliding motion on one of two sliding manifolds,
each associated with one direction of sharing.
We elaborate in \S \ref{secModel} below.

Here, we carefully examine the bad behavior that can occur when the control parameters are not chosen appropriately.
In those cases,
delays in activating the control can cause so much chattering that the fluid trajectory hits both
sliding manifolds, without remaining in either.
As a consequence,
the chattering
is more complicated than in the example above.
Here the chattering manifests itself in periodic oscillations.
The oscillatory behavior leads to inefficient utilization of the service capacity, thus creating severe overloads,
even though the arrival rates we consider are smaller than the potential service capacity.
Subcritical queueing networks that become overloaded due to exercising a bad control
are said to experience {\em congestion collapse}, as in
\cite{SW11}; see \S 1.2 in \cite{PW14}.

Chattering in sliding-mode controls is a well-known phenomenon in deterministic control theory.
Indeed, chattering is considered to be the natural ``state of affairs'',
whereas perfect sliding motion is considered ``ideal'' and typically unrealistic; e.g., \S 14.1 in \cite{Khalil}.
Accordingly, even though we focus on a single system that operates under a specific control, our results
have broader relevance. In particular, similar phenomena should be expected to occur with other SSC-inducing controls
when there are deviations from ideal modeling assumptions, such as stationarity, or ``convenient'' initial conditions and control settings.

\paragraph{Switching Dynamical Systems.}

The chattering found in the fluid model implies that the ODE governing the
evolution of the fluid trajectories switches whenever the control is activated or released.
Therefore, the appropriate fluid model $x := \{x(t) : t \ge 0\}$ for
the stochastic system is a {\em switching dynamical system}
$\dot{x} = f_{\sigma(x)}(x)$, where $\sigma(x)$ achieves a finite set of values,
$f_i$ is a continuous function for each value $i$ of $\sigma$,
but the function $f_{\sigma}$ is discontinuous \cite{LiberzonBook}.
As the notation suggests, the switching epochs are state dependent (depending only on the value of the solution $x$),
so that the ODE is autonomous (time-homogeneous).

The framework of switching systems in general, and of systems with sliding motion in particular,
is outside the classical ODE and dynamical-systems theory,
because the right-hand side function $f_{\sigma}$ is not continuous, and so it is not locally Lipshcitz.
Hence, the conditions of the Picard-Lindel\"{o}f theorem, ensuring the existence of a unique solution to the ODE, are not satisfied.
In general, the existence of a unique solution to a switching system with no sliding motion
can only hold in the Carath\'{e}odory sense, namely, such a
solution is an absolutely-continuous function that satisfies the ODE almost everywhere; see \cite{LiberzonBook}.
A solution with a sliding motion is generally considered to hold in the Filippov sense \cite{Filippov},
but we have shown in \cite{PW11b, PW13} how to prove that a unique solution exists for our system via a stochastic {\em averaging principle}
when the fluid limit slides on its target sliding manifold (i.e., the control achieves the desired asymptotic SSC).
Since we do not consider SSC in this paper, we do not review the Filippov theory, nor the averaging principle method.
The theory of the former is found in \cite{Filippov}, and the latter in \cite{kurtzAP} and \cite{PW11b,PW13}.

\paragraph{Analytical Contributions.}
In addition to exposing the chattering behavior discussed above,
our current work has important analytical contributions.
We emphasize at the outset that the derivation of the fluid model (which will also be shown to be the FWLLN in \S \ref{secFWLLN}) is standard,
and the analytical contributions lie in the nontrivial qualitative analysis of the fluid model.
Specifically, we provide sufficient conditions for chattering to lead to endless oscillations, and prove the existence of a periodic equilibrium.
Furthermore, we provide a simple algorithm to efficiently analyze the system
for any given initial condition.

It is known that even seemingly simple switching systems can experience chaotic-like behavior, e.g.,
have infinitely-many periodic equilibria that are dense in the state space,
and exhibit high sensitivity to perturbations of the initial condition (popularly known as ``the butterfly effect'');
see, e.g., \cite{ChaseChaos, Erramilli-chaos}.
Such systems are clearly unamenable to long-run analysis.
Even fluid models of {\em uncontrolled} systems can have uncountably-many periodic equilibria \cite{LiuWh11}.
However, numerical experiments suggest that our system has at most one periodic equilibrium, and that it is bi-stable,
i.e., any fluid trajectory can have long-run behavior of only two kinds:
either it converges to the periodic equilibrium, or else it converges to the unique stationary point (which is therefore asymptotically stable).


To conduct a more rigorous study of the (bi)stability properties of the fluid model, we create an
approximation to the fluid system.
(Note that ``stability'' here does not refer to the prelimit queueing system
which is always stable due to assumed abandonment.)
For that approximating dynamical system we show that all oscillating solutions must converge to the unique periodic equilibrium
(of the approximating system), while
all other solutions converge to the unique stationary point, which is the same as that of the fluid limit.
In particular, the approximating system is bistable.
We conjecture that the same is true for the fluid limit; see Conjecture \ref{conj} below.
This conjecture is supported by numerical experiments in \S \ref{secNumeric}.

To summarize our analytical contribution,
we develop and analyze two layers of approximations, one being the fluid limit, which approximates the stochastic system,
and the other being an approximating dynamical system which serves as a simplified approximation to the fluid limit,
whose qualitative behavior is easier to characterize.

\paragraph{Implications of the Fluid Analysis to the Stochastic System.}

A straightforward implication of our result that the fluid limit may oscillate indefinitely is that the prelimit
stochastic systems can experience congestion collapse.
Moreover, the fluid limit may oscillate, even though the stochastic system in the pre-limit is an ergodic {\em continuous-time Markov chain} (CTMC)
and is therefore necessarily aperiodic with a unique equilibrium (stationary) distribution.
Since the CTMC converges to its unique stationary distribution also for initial conditions that are associated with
oscillatory fluid limits, one concludes that the convergence rate of the CTMC to stationarity must be prohibitively slow.
We elaborate in \S \ref{secImplications} of the appendix.

Our fluid analysis also has indirect implications to the stochastic system.
Specifically, stochastic noise, which is not captured by the fluid approximation,
may eventually push the system into the oscillatory behavior, even if the system is unambiguously initialized in the attraction region
of the stationary point. This suggests that stochastic fluctuations can lead to {\em fluid-scaled fluctuations}.
In addition, oscillations can occur in the stochastic system even if its fluid limit does not possess a periodic equilibrium, and never oscillates.
Therefore, studying the relatively simple fluid model is important for gaining insight into
the dynamics of the stochastic system. See the examples in \S \ref{secExampleWithAbd} below.

\paragraph{Organization.}
The rest of the paper is organized as follows.
We describe the stochastic model and the control in \S \ref{secModel}.
In \S \ref{secFluidModel} we explain how to construct a direct fluid model to approximate the system's dynamics.
The switching fluid model is derived in \S \ref{secOscFluid}.
Qualitative analysis, including relevant equilibrium and stability notions for dynamical systems, are rigorously defined
and analyzed in \S \ref{secQualitative}.
In \S \ref{secEquilibria} we show that the fluid model can oscillate indefinitely and when it does we
show there exists a periodic equilibrium.
The approximating dynamical system to the fluid model is developed in \S \ref{secApproxDS} and
is shown to be bi-stable.
Numerical examples and simulation experiments are provided in \S \ref{secNumeric}.
We conclude in \S \ref{secConclude}.

Additional material appears in an appendix.
We develop important bounds on the fluid processes and the switching times in \S \ref{secBoundsOSC}
which are employed 
to prove Theorem \ref{thEndless} which states that there are parameters under which the fluid model oscillates indefinitely.
In \S \ref{secGeom} we prove that the solutions to the approximating system converge geometrically fast to their equilibrium.
We prove that the fluid model considered in this paper arises as a bona-fide {\em functional weak law of large numbers} (FWLLN) in \ref{secAsymptotic}.

\section{The Model} \label{secModel}

We start by reviewing the stochastic model which is assumed Markovian, and in particular, it can be described as a CTMC.
In \S \ref{secFluidModel} we quickly develop the deterministic fluid model to the stochastic system, which will be our focus in this paper.
We defer the proof that the fluid model is indeed a rigorous approximation via a FWLLN to the appendix; see \S \ref{secFWLLN}.

The model has two large service pools of many homogeneous agents in a call center, each with
with its own arrival stream and designated queue for waiting customers. We assume that customers have finite patience, and will abandon
if their wait time in queue exceeds their patience.
The two pools are designed to operate independently when both are normally loaded, i.e., to serve their own arrivals only,
but all the agents can help both customer classes.

Sharing of customers (namely, routing customers from one pool to be served in the other pool) may be beneficial if one of the pools is overloaded, even if
sharing makes the second service pool overloaded as well, because abandonment keep the two queues stable.
Indeed, in \cite{PW09} we showed that sharing of customers may be optimal during overload periods in a deterministic fluid approximation, assuming a convex
holding cost is incurred on the two queues. However, as we showed in Proposition 2 in \cite{PW09}, when agents are less efficient in serving the other
class, i.e., agents serve shared customers slower on average than their designated customers,
it is never optimal to share in both directions simultaneously.
Nevertheless, since sharing of customers in either direction takes place sometimes, the routing graph of the system
has the letter X shape, as can be seen in Figure \ref{figX}, and is therefore called the X model in the call-center literature.

\begin{figure}[h!]
\begin{center}
\includegraphics[width=4.5cm]{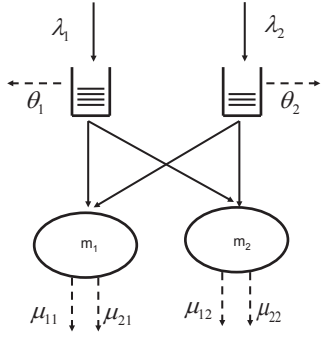}
\caption{The $X$ model}
\end{center}
\label{figX}
\end{figure}

In general, there is a fluid-optimal amount of sharing for any given pair of arrival rates and so, to find how many agents in the helping pool
should be assigned to shared customers requires knowing the exact arrival rates during the overload period.
A simplification is achieved by observing that the exact amount of sharing does not need to be determined at the outset, since it can be achieved,
at least approximately, if the two fluid queues are kept at a fixed ratio during overload periods. We again refer to \cite{PW09}.
There is a different optimal ratio for each direction of sharing, and the direction of sharing depends on which pool is overloaded.

The above reasoning lead us to design the {\em fixed queue ratio with thresholds} (FQR-T) overload control, which
(i) is activated automatically once the queue ratio exceeds a certain ``activation threshold'' (so that the system is considered overloaded);
(ii) aims to maintain the two queues at a pre-specified fixed ratio (in the many-server asymptotics);
(iii) class-$i$ customers are routed to pool $j$ only if there are no class $j$ customers in pool $i$, $i \ne j$.

In time-varying settings, the direction of overload may switch, so that the direction of sharing must switch as well.
If the {\em one-way sharing rule} in Condition (iii) above is forced, then substantial delays in switching the direction of sharing
may occur. We therefore modified FQR-T in \cite{PW14} by introducing {\em release thresholds} for the service process.
Specifically, in the modified {\em fixed queue ratio with activation and release thresholds} (FQR-ART) control
the one-way sharing rule is relaxed as follows: class-$1$ customers can be routed
to pool $2$, provided that the number of class-$2$ customers in pool $1$ is smaller than a release threshold
$\tau_{2,1} > 0$, and similarly in the other direction.
We elaborate in \S \ref{secControl} below.

\paragraph{Cyclic Routing Graph.}
An important characteristic of the X model is that its (undirected) routing graph is cyclic.
In particular, it is the most basic cyclic {\em parallel server system} (PSS). The X model is therefore easier to study than other cyclic PSS's
but at the same time serves as a representative to problems that are associated with its cyclic structure.
Indeed, in \cite{PW09} we showed that the QIR control from \cite{GW-QIR} can produce severe congestion collapse
if applied to the X model when the service rates of shared customers are slower than those of designated ones.
This congestion collapse cannot occur in PSS's having a tree graph; see Theorem 3.1 in \cite{GW-QIR}.
The oscillatory behavior analyzed here is also due to cyclic structure of the system.

\subsection{The FQR-ART Control} \label{secControl}

We will start by developing a deterministic fluid approximation for the stochastic system directly,
but to fully describe the control we must consider that fluid model from an asymptotic perspective,
We therefore consider a sequence of X models indexed by superscript $n$, where system $n$ has $m^n_i$ agents in pool $i$
and arrival rate $\lm^n_i$ of class-$i$ customers, $i = 1,2$. We assume that the arrival rates and number of agents in each pool
grow proportionally to $n$ as $n\tinf$, putting us in the many-server heavy-traffic framework. See Assumption \ref{assHT} in \S \ref{secReview}.

The control of each system $n \ge 1$ is based on two {\em activation thresholds}, $k^n_{1,2}$ and $k^n_{2,1}$, two {\em release thresholds},
$\tau^n_{1,2}$ and $\tau^n_{2,1}$, and two ratio parameters $r_{1,2}$ and $r_{2,1}$.
These ratios, which are independent of $n$, are chosen to be optimal in a fluid model
of an overloaded system (here we will consider underloaded systems), as was mentioned above.

Let $Q^n_i(t)$ denote the number of class-$i$ customers waiting in their designated queue at time $t$,
and let $Z^n_{i,j}(t)$ denote the number of class-$i$ customers being served in pool $j$ at time $t$.
The FQR-ART is an overload control, namely, it is designed to be activated and start customer sharing automatically when an overload occurs.
To define overloads, we consider the difference processes
\bequ \label{Dprocesses}
D^n_{1,2}(t) \equiv Q^n_1(t) - r_{1,2}Q^n_2(t) - k^n_{1,2} \qandq D^n_{2,1} \equiv r_{2,1}Q^n_2(t) - Q^n_1(t) - k^n_{2,1}, \quad t \ge 0.
\eeq
As long as $D^n_{1,2} < 0$ and $D^n_{2,1} < 0$, the system is considered normally loaded. Once one of these difference process hits $0$,
which corresponds to the ratio between the two queues hitting one of the activation thresholds, the system is deemed overloaded,
and sharing begins, provided that there is only a small number of shared customers in the overloaded pool.
By ``small number'' we mean that the number of shared customers in the overloaded pool is no larger than its associated release threshold.
For example, if $D^n_{1,2}(t) \ge 0$, then class $1$ is judged to be overloaded (because then $Q^n_1(t) - r_{1,2} Q^n_2(t) \ge k^n_{1,2}$)
and it is desirable to send class-1 customers to be served in pool $2$. However, sharing is allowed only if $Z^n_{2,1}(t) \le \tau^n_{2,1}$.
Similar rules apply to overloads in the other direction.

Once sharing is activated, say with class $1$ receiving help from pool $2$, the routing rule is as follows:
Any agent, from either pool, that becomes available at any time $t$, will take his next customer
from class $1$ if $D^n_{1,2}(t) > 0$, and will take his next customer from his designated queue otherwise.
Observe that this means that agents from pool $1$ will only take customers from their own queue, but some class $1$ customers will be
routed to pool $2$.
The routing mechanism when class $2$ is overloaded is similar, with $D^n_{2,1}$ replacing $D^n_{1,2}$,
and the labels of the thresholds switched.

\paragraph{Spare Capacity in One Pool.}
With release thresholds the possibility of having congestion collapse due to too much simultaneous sharing is avoided.
However, when one pool has significant idleness (due to low arrival rate)
while the other pool is severely overloaded, it may be beneficial to consider that idleness as ``spare capacity'' in the system,
and exploit it to help the congested queue.
That motivates an exception to the sharing rule when one pool has idleness while the other has a queue that is larger than its corresponding activation threshold.
For example, if pool $2$ has idleness (and no queue, necessarily) and $Q^n_1(t) \ge k^n_{1,2}$,
then a newly available agent in pool $2$ at such a time $t$ will take a customer from
queue $1$, regardless of the value of $Z^n_{2,1}(t)$, i.e., even if $Z^n_{2,1}(t) > \tau^n_{2,1}$.
However, to avoid sharing in pool $2$ beyond its spare capacity, there is strict priority to class-$2$
customers in the sense that a type-$2$ agent will always give strict priority to its own customer class whenever $Z^n_{2,1}(t) > \tau^n_{2,1}$.
The exact same routing rule is used if, at a time $t$ in which a type $1$ agent becomes available, pool $1$ has idleness and $Q^n_2(t) \ge k^n_{2,1}$.

This control is a version of the control in ec21 in the electronic companion (EC) of \cite{PW09}.
It is not hard to show that when pool $i$ has spare capacity, $i=1,2$, its own queue
will remain null in the {\em fluid model (and fluid limit)} when this control is employed, due to the strict priority it receives from its agents.

\subsection{A Deterministic Fluid Model} \label{secFluidModel}

If the arrival processes are independent Poisson processes, and all service times and times to abandon
are independent exponential random variables, then the six-dimensional process
\bequ \label{Xn}
X^n(t) = (Q^n_i(t), Z^n_{i,j}(t) ; i,j = 1,2), \quad t \ge 0,
\eeq
is a CTMC. Our goal is to develop and then analyze a fluid approximation for this CTMC, based on asymptotic considerations
(which will be made rigorous in \S \ref{secFWLLN}).

When sharing is active, the control aims to keep the two queues at the corresponding fluid-optimal ratio, either $r_{1,2}$ or $r_{2,1}$,
depending on the direction of sharing.
Minor modifications to the statement and proof of Corollary 4.1 in \cite{PW13} show that, if the system is overloaded and there is no
sharing initially, then the control achieves asymptotic SSC in the fluid limit
(or under any scaling of the appropriate process in \eqref{Dprocesses} that is larger than $\log n$).
More general assumptions were considered in \cite{PW14}. 
The mathematical support for the asymptotic SSC was a direct consequence of the aforementioned stochastic averaging principle.

The oscillatory performance and its resulting congestion collapse we analyze here does not involve the averaging principle, because there is no SSC.
Indeed, unlike the fluid models in \cite{PW13} and \cite{PW14}, the fluid model we develop here has an explicit solution.
The challenges are associated with proving that oscillations (and congested collapse) can be self-sustained and in studying the long-run
behavior of the fluid model.

It is significant that the fluid approximation for $X^n$ is obtained as the FWLLN for $\barx^n \equiv X^n/n$, see \S \ref{secFWLLN}. However, we start
by deriving the fluid model directly. (We refer to the fluid model as fluid approximation or limit, depending on the context, as the terms are
equivalent in our case.)
For each of the six stochastic processes comprising $X^n$ in \eqref{Xn} there is a fluid counterpart, namely a deterministic and
almost-everywhere differentiable function. We let $x \equiv \{x(t) : t \ge 0\}$ denote the fluid approximation of $X^n$, where
\bes
x(t) = (q_1(t), q_2(t), z_{1,1}(t), z_{1,2}(t), z_{2,1}(t), z_{2,2}(t)), \quad t \ge 0,
\ees
and call a time $t$ ``regular'' if $x(t)$ is differentiable at $t$.
In our case, any compact interval will have at most a finite number of points that are not regular.

To derive the fluid equations, we simply replace the instantaneous rates of the stochastic processes at each time $t$ with instantaneous
rates of change of the derivatives of their fluid counterparts, e.g., the instantaneous
rate of abandonment from queue $1$ at time $t$ in system $n$ is $\theta_1 Q_1^n(t)$, which becomes $\theta_1 q_1(t)$ in the fluid model.
Similarly, the instantaneous rate of departure from service in pool $j$ at time $t$ is $\mu_{j,j} Z^n_{j,j}(t) + \mu_{i,j} Z^n_{i,j}(t)$
in system $n$ is replaced with the instantaneous processing rate $\mu_{j,j} z_{j,j}(t) + \mu_{i,j} z_{i,j}(t)$ in the fluid model.
Combining all these instantaneous rates gives the derivative of $x(t)$ at a regular time $t$.

For example, if both queues are smaller than the activation thresholds at a time $t$,
then any newly-available agent in pool $1$ will take his next customer from queue $1$ in the stochastic system.
Similar reasonings applied to $q_2$ give that, if $q_1(t) < k_{1,2}$ and $q_2(t) < k_{2,1}$, and $t$ is regular, then
\bequ \label{qODE}
\bsplit
\dot{q}_1(t) & = \lm_1 - \theta_1 q_1(t) - \mu_{1,1}z_{1,1}(t) - \mu_{2,1} z_{2,1}(t), \\
\dot{q}_2(t) & = \lm_2 - \theta_2 q_2(t) - \mu_{2,2}z_{2,2}(t) - \mu_{1,2} z_{1,2}(t).
\end{split}
\eeq
We derive the full set of differential equations for the fluid model during overload periods (due to congestion collapse) in \S \ref{secFluidEq} below.

The purpose of FQR-ART is to produce SSC in the fluid limit by sending customers from one queue
to both pools according to the routing rules described above during overload periods. If the control is successful in achieving SSC,
the six-dimensional fluid model is confined to one of the {\em sliding manifolds}
\bes
\bsplit
\SS_{1,2} \equiv \{x \in \SS : q_1 - r_{1,2} q_2 = k_{1,2},\: z_{1,1} + z_{2,1} = m_1,\: z_{1,2} + z_{2,2} = m_2\}, \\
\SS_{2,1} \equiv \{x \in \SS : r_{2,1}q_2 - q_1 = k_{2,1},\: z_{1,1} + z_{2,1} = m_1,\: z_{1,2} + z_{2,2} = m_2\},
\end{split}
\ees
where $\SS = \RR^2_+ \times [0, m_1] \times [0, m_2]$ is the domain of $x$.

The behavior of the fluid limit when sliding on one of these manifolds can be thought of as an infinitely-fast chattering with infinitely-small
fluctuations of the queues about the corresponding activation threshold.
This view can be justified rigorously via the aforementioned stochastic averaging principle;
see \S 4 in \cite{PW11b} and Theorem 4.1 in \cite{PW13}.

Observe that the fluid model is essentially a {\em three-dimensional process} on either one of these sliding manifolds, because
knowing $x_3 \equiv (q_1, z_{1,2}, z_{2,1})$ for example, is sufficient to determine the value of the remaining three processes.
Here, however, we are interested in bad oscillatory behavior when the fluid model overshoots past
the sliding manifold due to delay in activating the control, where a
delay is caused if $z_{j,i}(t_0) > \tau_{j,i}$, at the time $t_0$ in which $\SS_{i,j}$ is hit.
If no SSC occurs, we must consider all six components of the fluid model and, as will become clear below, four different switching epochs
for each cycle. We can obtain considerable simplification by considering a symmetric model.
Symmetry reduces the amount of notation and, as will become clear later, allows us to focus attention on two switching times in each cycle
instead of four.

\paragraph{A Symmetric Model.}
In order to expose the bad behavior that can result from poorly chosen controls,
we consider a special case that is easier to analyze than the general model.
In particular, we consider systems with the following parameters
\bequ \label{symModel}
\bsplit
\mu_{1,1} & = \mu_{2,2} = 1, \quad \mu_{1,2} = \mu_{2,1} = \mu < 1, \quad \lambda_1  =  \lambda_2 = \lambda < 1, \quad \theta_1 =  \theta_2 = \theta > 0, \\
m_1 & = m_2 = 1, \quad r_{1,2}  =  r_{2,1} = 1, \quad \tau_{1,2} = \tau_{2,1} = \tau > 0 \qandq  k_{1,2} = k_{2,1} = \kappa. 
\end{split}
\eeq
Observe that time is measured in terms of $\mu_{1,1}$ and $\mu_{2,2}$ (which are normalized to be equal to $1$).
In this model there are $5$ parameters instead of $16$ in the general case.
There is the triple of model parameters $(\lambda, \mu, \theta)$ and the pair of control parameters $(\kappa, \tau)$.
Note that each of the pools is underloaded if there is no sharing that slows its potential service capacity,
because $\lm_i < \mu_{i,i}m_i = 1$, $i = 1,2$.

In this model, there is sharing with all class-$1$ fluid sent to pool $2$ if $q_2 (t) > q_1 (t) + \kappa$ and $z_{1,2} (t) \le \tau$;
there is sharing with all class-$2$ fluid sent to pool $1$ if $q_1 (t) > q_2 (t) + \kappa$ and $z_{2,1} (t) \le \tau$;
there is complex sharing, associated with sliding motion and described by the averaging principle if
if $q_2 (t) = q_1 (t) + \kappa$ and $z_{1,2} (t) \le \tau$ or if $q_1 (t) = q_2 (t) + \kappa$ and $z_{2,1} (t) \le \tau$;
there is possibly sharing according to the spare capacity control described above if $q_1(t) \ge \kappa$ and $z_{1,2}(t) + z_{2,2}(t) < m$
or $q_2(t) \ge \kappa$ and $z_{2,1}(t) + z_{1,1}(t) < m$.
otherwise there is no sharing actively taking place.

We have assumed in \eqref{symModel} that $\lm < \mu$, so that either pool is underloaded if it serves its own class only (because $\mu_{i,i}m_i = 1$, $i = 1,2$).
It will be convenient to assume that $\lm \le 1 - \tau$. In that case, if class $i$ fluid is sent to pool $j$ at time $t$, $i \ne j$, then $z_{j,i}(t) \le \tau$
and the instantaneous service rate in pool $i$ is
\bes
\mu z_{j,i}(t) + z_{i,i}(t) = \mu z_{j,i}(t) + (1-z_{j,i}(t)) \ge \mu z_{i,j}(t) + 1 - \tau \ge \mu z_{i,j}(t) + \lm \ge \lm,
\ees
implying that the instantaneous total service rate in pool $i$ is larger than the arrival rate to that pool so that $q_i$ is decreasing; see also \eqref{qODE}.
In addition, to achieve explicit solutions to the ODE's we develop, we will assume that $\theta < \mu$.
We summarize in the following assumption.
\begin{assumption} \label{assRates}
The model parameters satisfy \eqref{symModel}. Furthermore,
$\lm \le 1 - \tau$ ~ and ~ $\theta < \mu.$
\end{assumption}
Assumption \ref{assRates} is not necessary for chattering and oscillations to occur, and is taken in order to somewhat simplify the analysis.


Since the activation thresholds $\kappa$ are strictly positive in the fluid model, there is no ambiguity about the translation
of the FQR-ART control to the fluid model when there is no SSC. It is then entirely determined by the processes
\bequ \label{dProcesses}
d_{1,2}(t) = q_1(t) - q_2(t) - \kappa \qandq d_{2,1}(t) = q_2(t) - q_1(t) - \kappa, \quad t\ge 0,
\eeq
which are simply the fluid counterparts of \eqref{Dprocesses}.
Due to the assumed symmetry, the state space of the fluid model is $\RR_+^2 \times [0, 1]^4$
and the sliding manifold are defined via
\bequ \label{manifolds}
\bsplit
\SS_{1,2} & \equiv \{x \in \SS : d_{1,2} = 0,\: z_{1,1} + z_{2,1} = 1,\; z_{1,2} + z_{2,2} = 1\} \\
\SS_{2,1} & \equiv \{x \in \SS : d_{2,1} = 0,\: z_{1,1} + z_{2,1} = 1,\; z_{1,2} + z_{2,2} = 1\}.
\end{split}
\eeq
For $i,j = 1,2$, $i \ne j$, we define
$$\SS_{i,j}^- \equiv \{x \in \SS : d_{i,j} < 0\} \qandq \SS^+_{i,j} \equiv \{x \in \SS_{i,j} : d_{i,j} > 0\}.$$

If $x(t) \in \SS_{i,j}$ for all $t$ over some interval $I$, then $x$ is said to slide on the sliding manifold $\SS_{i,j}$.
Chattering corresponds to the fluid trajectory hitting and immediately crossing a sliding manifold, e.g., when it is
moving from $\SS_{i,j}^-$ to $\SS_{i,j}^+$ (necessarily via $\SS_{i,j}$) without sliding on $\SS_{i,j}$, and back from $\SS_{i,j}^+$ to $\SS_{i,j}^-$.
It will be clear that chattering about one sliding manifold is not sustainable unless the fluid trajectory makes it all the way to the second
manifold. When both manifolds are hit, we say that the fluid oscillates.
Since we will consider initial conditions in $\SS_{2,1}^+$, a full cycle is considered to end when the fluid trajectory first enters $\SS^+_{2,1}$
after hitting $\SS_{1,2}$.
When chattering or oscillations occur, the sliding manifolds in \eqref{manifolds} become {\em switching surfaces}, because the
dynamics of the fluid model switches when it hits either of these subspaces.

The sliding manifolds in \eqref{manifolds} should not be confused with the {\em invariant manifolds} in \cite{bramsonSSC} which are defined
to be the fixed points of the fluid limit.

\paragraph{The State Space.}
It is easily seen from \eqref{qODE} that
$\dot{q}_i (t) \le \lambda - \theta q_i (t),$
and that this inequality holds for all $t \ge 0$ regardless of the routing.
It follows from the comparison principle for ODE's, e.g., Lemma 3.4 in \cite{Khalil}, that for all $t > 0$,
\bes
q_i(t) \le \max\{q_i(0), \lambda/\theta\}, \quad i = 1,2,
\ees
and that, if $q_i(0) > \lm/\mu$, then $q_i$ must be strictly decreasing as long as $q_i(t) > \lm/\theta$.
Furthermore, $q_i$ can never cross $\lm/\mu$ from below, i.e., if $q_i(s) < \lm/\theta$, then $q_i(t) < \lm/\theta$ for all $t > s \ge 0$.
We can therefore assume without any loss of generality that $q_i(0) < \lm/\theta$ so that the state space of the symmetric model
is the compact and convex subset $\SS \subset \RR_6$, where
\bequ \label{space}
\SS \equiv [0,\lm/\theta]^2 \times [0, 1]^4.
\eeq

\section{The Switching Fluid Model} \label{secOscFluid}

Consider a system that has just recovered from an overload, in which class $1$ was receiving help from pool $2$.
Suppose that $\lm_1$, which was greater than $\mu_{1,1}m_1 = 1$ during the preceding overload period, dropped to the value $\lm < 1$ in \eqref{symModel}
Since sharing was taking place with pool $2$ helping, we necessarily had $z_{2,1} < \tau$ and $q_1 - q_2 = \kappa > 0$ ($x$ sliding on $\SS_{1,2}$)
during the overload period.


Assuming that $z_{1,2}$ was larger than $\tau$ during the preceding overload period,
we designate by $0$ the first time that $z_{1,2}$ hits $\tau$, so that sharing can begin with pool $1$ helping queue $2$ if $d_{2,1}(0) > \kappa$.
Formally, for $\SS$ in \eqref{space},
\begin{assumption}{$($initial condition$)$}\label{assInit1}
\bes
x(0) \in \SS, \quad q_1 (0) > 0, \quad d_{2,1}(0) > 0 \: (\mbox{ i.e., } q_2 (0) > q_1 (0) + \kappa), \quad z_{1,2} (0) = \tau \qandq 0 \le z_{2,1} (0) < \tau.
\ees
\end{assumption}

To describe the oscillatory behavior of the fluid model, we define the hitting times
\bequ \label{T1T2}
\bsplit
T_1 & \equiv \inf{\{t \ge 0: d_{2,1} (t) \le \kappa\}} \\
T_2 & \equiv \inf\{t \ge 0 : z_{2,1}(\Sigma_1 + t) \le \tau\}, \\
T_3 & \equiv \inf\{t \ge 0 : d_{1,2}(\Sigma_2 +t) \le \kappa\}  \\
T_4 & \equiv \inf\{t \ge 0 : z_{1,2}(\Sigma_3 + t) \le \tau\}, \\
\end{split}
\eeq
where, with $T_0 \equiv \Sigma_0 \equiv 0$,
\bequ \label{switchTimes}
\Sigma_k \equiv \sum_{i=0}^k T_i \qandq \I_i \equiv [\Sigma_{i-1}, \Sigma_i), \quad k = 1,2,3,4.
\eeq

We refer to the times $\Sigma_i$ as {\em switching times}, and to $T_i$ as {\em holding times} (the times between switching).
The length of each interval $\I_i$ is $T_i$, i.e., $|\I_i| \equiv \Sigma_i - \Sigma_{i-1} = T_i$, $1 \le i \le 4$.
We will interchangeably write $T_1$ or $\Sigma_1$, and $T_1+T_2$ or $\Sigma_2$, as convenient.

Clearly $T_1 > 0$ for the initial condition in Assumption \ref{assInit1}, but it is possible that $T_i = 0$ for $i > 1$.
Observe that if at the end of the first cycle $x(\Sigma_4)$ satisfies the same conditions specified for $x(0)$
in Assumption \ref{assInit1}, then $x(\Sigma_4)$ can be taken as a new ``initial condition'' for the fluid model (which is time homogeneous, as will be shown below),
and a new cycle begins.
Furthermore, if both fluid queues are strictly positive on $[0, \Sigma_q)$ and $z_{2,1}(\Sigma_1) > \tau$ in addition to
$d_{1,2}(\Sigma_2) > 0$,
then $x(\Sigma_2)$ can be thought of as a ``mirror image'' of $x(0)$ because we necessarily have $0 < z_{1,2}(\Sigma_2) < \tau$.
In particular $x(\Sigma_2)$ satisfies the conditions in Assumption \ref{assInit1}, {\em but with the labels (subscripts) reversed}.
Similarly, if both queues remain positive throughout $[0, \Sigma_3)$, then $x(\Sigma_3)$ is a ``mirror image'' of $x(\Sigma_1) \equiv x(T_1)$.
This observation greatly simplifies the search for a periodic equilibrium since,
on the trajectory of a periodic equilibrium, it holds that $x_s(\Sigma_2) = x(0)$ and $x_s(\Sigma_3) = x(\Sigma_1)$,
where $x_s := (q_2, q_1, z_{2,2}, z_{2,1}, z_{1,2}, z_{1,1})$ (i.e., $x_s$ has the labels of $x$ reversed).
We can then focus on analyzing a half cycle $[0, \Sigma_2]$ for the symmetric model.

Hence, we consider the fluid model as long as the conditions in Assumption \ref{assInit1} hold in the switching times, either for $x$ or for $x_s$.
It will be seen below that, for any initial condition in $\SS$, $0 \le z_{i,j} \le 1$, $i,j = 1,2$.
However, the equations for $q_1$ and $q_2$ can become negative.
We thus consider the fluid model on $[0, \Sigma_q)$, where
\bequ \label{Sigmaq}
\Sigma_q \equiv \inf\{t > 0 : \min\{q_1(t),q_2(t)\} = 0\}.
\eeq
Since $T_1 > 0$ for any initial condition satisfying Assumption \ref{assInit1}, we necessarily have $\Sigma_1 > T_1 > 0$.
Similarly, if $\Sigma_2 > 0$, then necessarily $T_3 > 0$. It follows that, if $\Sigma_q < \Sigma_4$,
then $\Sigma_q \in \I_2$ or $\Sigma_q \in \I_4$.
On the other hand, if $x(\Sigma_4)$ satisfies the conditions in Assumption \ref{assInit1}, then $\Sigma_q > \Sigma_4$.
We then take $x(\Sigma_4)$ as the initial condition for the second cycle, and start over.
We will provide sufficient conditions for $\Sigma_q$ to be infinite, in which case cycle-end time $\Sigma_4$ is the beginning of a
new full cycle, and the fluid model keeps oscillating indefinitely. Since both queues are strictly positive throughout (despite Assumption \ref{assRates}),
we get congestion collapse that is due to self-sustained oscillations.

\subsection{The Switching Fluid Equations} \label{secFluidEq}

\subsubsection{The Equations on $\I_1$: Both Pools Serve Queue $2$ Only}\label{secInt1}

Recall that over the interval $\I_1 \equiv [0, \Sigma_1)$ sharing takes place with both pools accepting only fluid from queue $2$ and no fluid from queue $1$.
For a given initial condition $x(0)$ satisfying Assumption \ref{assInit1}, and determined by specifying the triple $(q_1(0), q_2(0), z_{2,1}(0))$,
the fluid equations for the service process are therefore
\begin{eqnarray}\label{poolsT1}
\dot{z}_{1,1} (t) & = & -z_{1,1} (t) \mu_{1,1}, \quad \mbox{so that} \quad z_{1,1} (t) = (1 - z_{2,1} (0))e^{-t} \qandq z_{2,1} (t) = 1 - z_{1,1} (t) \nonumber \\
&& \quad \quad \mbox{so that} \quad z_{2,1} (t)  = 1 - (1 - z_{2,1} (0))e^{-t} \nonumber \\
\dot{z}_{1,2} (t) & = & -z_{1,2} (t) \mu_{1,2}, \quad \mbox{so that} \quad  z_{1,2} (t) = \tau e^{-\mu t} \qandq z_{2,2} (t) = 1 - \tau e^{-\mu t},
\end{eqnarray}

and the fluid equations for the queue processes are
\begin{eqnarray}\label{queuesT1}
\dot{q}_{1} (t) & = & \lambda - q_1 (t)\theta, \nonumber \\
\dot{q}_{2} (t) & = & \lambda - q_2 (t)\theta -z_{1,1} (t) \mu_{1,1} - z_{2,1} (t)\mu_{2,1} -z_{1,2} (t)\mu_{1,2} -z_{2,2} (t) \mu_{2,2}  \nonumber \\
& = &\lambda - q_2 (t)\theta -  [(1 - z_{2,1} (0))e^{-t} + 1  -\tau e^{-\mu t} ]
- [1 - (1 - z_{2,1} (0))e^{-t} + \tau e^{-\mu t}]\mu \nonumber  \\
& = &(\lambda - 1 - \mu) - q_2 (t)\theta -  (1-\mu)(1- z_{2,1} (0))e^{-t}+ (1-\mu)\tau e^{-\mu t}.
\end{eqnarray}

For the given initial condition $x(0)$, we can calculate the interval termination time $T_1$ and the fluid performance functions in $\I_1$.
Observe that by first solving for the service processes in \eqref{poolsT1}, the autonomous (time-homogeneous) ODE for the queues becomes
a nonhomogeneous first-order linear ODE. Under the condition $\theta < \mu$ in
Assumption \ref{assRates}, the explicit solution to the ODEs \eqref{queuesT1} over $[0, T_1)$ is

%
\begin{eqnarray}\label{queuesT1exp2}
q_1 (t) & = & q_1 (0)e^{-\theta t} + \left(\frac{\lambda}{\theta}\right)(1 - e^{-\theta t})  \nonumber \\
q_2 (t) & = & q_2 (0)e^{-\theta t} + \left(\frac{\lambda - 1 - \mu}{\theta}\right)(1 - e^{-\theta t})
- \left(\frac{(1-\mu)(1- z_{2,1} (0))}{1 - \theta}\right)(e^{-\theta t}  -e^{-t}) \nonumber \\
&& \quad \quad +  \left(\frac{(1 - \mu)\tau}{\mu - \theta}\right)(e^{-\theta t} -e^{-\mu t}).
\end{eqnarray}

We see that $q_1 (t)$ is strictly increasing in $\SS$ and
necessarily remains strictly positive in the interval $\I_1$.
Given the initial conditions in Assumption \ref{assInit1} and the definition of $\Sigma_1 \equiv T_1$ in
\eqref{T1T2}, this implies that both fluid queue lengths are necessarily strictly positive in the interval $\I_1$, so that $\Sigma_q > T_1$.


%

\subsubsection{The Equations on $\I_2$: No Active Sharing}\label{secInt2}

Given any initial condition $(q_1 (0), q_2 (0), z_{2,1} (0))$, we can calculate $T_1$ and the $6$-tuple $(q_i (T_1), z_{i,j} (T_1)); i, j = 1,2)$.
These provide the initial condition for the second interval $\I_2 \equiv [\Sigma_1, \Sigma_2)$. We assume that
$z_{2,1}(T_1) > \tau$ so that sharing with pool $2$ helping queue $1$ did not begin at time $T_1$ and so $T_2 > 0$.
The fluid equations for the service process for $t \in \I_2$ are
\begin{eqnarray}\label{poolsT2}
\dot{z}_{2,1} (t) & = & -z_{2,1} (t) \mu_{2,1}, \quad \mbox{so that} \quad z_{2,1} (T_1 + t) = [1 - (1 - z_{2,1} (0))e^{-T_1}]e^{-\mu t} \nonumber \\
&& \quad \qandq z_{1,1} (T_1 + t) = 1 - z_{2,1} (T_1 + t) =  1 - [1 - (1 - z_{2,1} (0))e^{-T_1}]e^{-\mu t}      \nonumber \\
\dot{z}_{1,2} (t) & = & -z_{1,2} (t) \mu_{1,2}, \quad \mbox{so that} \quad z_{1,2} (T_1 + t) = \tau  e^{-\mu (T_1 +t)} \nonumber \\
&& \quad \qandq z_{2,2} (T_1 + t) = 1 - z_{1,2} (T_1 + t) = 1 - \tau  e^{-\mu (T_1 +t)}.
\end{eqnarray}

As long as both queues remain positive, since there is no
no new sharing in this second interval $\I_2$, at time $T_1 + t$ for $t \in [0, T_2]$, the queues evolve as follows:
\begin{eqnarray}\label{queuesT2}
\dot{q}_{1} (T_1 + t) & = & \lambda - q_1 (T_1 + t)\theta -z_{1,1} (T_1 + t) \mu_{1,1} -z_{2,1} (T_1 + t)\mu_{1,2} \nonumber \\
&& \quad = -(1 - \lambda) - q_1 (T_1 + t)\theta + (1-\mu)z_{2,1}(T_1)e^{-\mu t} \nonumber \\
\dot{q}_{2} (T_1 + t) & = & \lambda - q_2 (T_1 + t)\theta -z_{2,2} (T_1 + t) \mu_{2,2} - z_{2,1} (T_1 + t)\mu_{2,1} \nonumber \\
&& \quad = -(1 - \lambda) - q_2 (T_1 + t)\theta + (1-\mu)z_{1,2}(T_1)e^{-\mu t}
\end{eqnarray}
under the new initial condition $(q_1 (T_1), q_2 (T_1), z_{1,2} (T_1), z_{2,1} (T_1))$.

Paralleling \eqref{queuesT1exp2}, we can solve these ODE's explicitly: For all $t \in [0, T_2)$
\begin{eqnarray}\label{queuesT2exp}
 q_1 (T_1 + t) & = & q_1 (T_1)e^{-\theta t} + \left(\frac{\lambda -1}{\theta}\right)(1 - e^{-\theta t})
 +  \left(\frac{(1 - \mu) z_{2,1} (T_1)}{\mu - \theta}\right)(e^{-\theta t} - e^{-\mu t}) \nonumber \\
 q_2 (T_1 + t) & = & q_2 (T_1)e^{-\theta t} + \left(\frac{\lambda -1}{\theta}\right)(1 - e^{-\theta t})
 +  \left(\frac{(1 - \mu) z_{1,2} (T_1)}{\mu - \theta}\right)(e^{-\theta t} - e^{-\mu t}),
\end{eqnarray}
provided that $T_1 + t \le \Sigma_q$.

\subsubsection{The Switching Fluid Model}

The equations on $\I_3 \equiv [\Sigma_2, \Sigma_3)$ and $\I_4 \equiv [\Sigma_3, \Sigma_4)$ are derived similarly to the equations for the intervals
$\I_1$ and $\I_2$, assuming $\Sigma_q < \Sigma_4$.
We summarize in the following definition of the direct fluid model.
As was mentioned before, we consider the interval $[0, \Sigma_q)$ and provide sufficient conditions for $\Sigma_q$ to be infinite.
We further prove that oscillations must end at time $\Sigma_q$ when this time is finite.

For two real numbers $a,b$, let $a \wedge b \equiv \min\{a,b\}$. We will later also use the notation $a \vee b$ for the maximum between the two numbers.
\begin{definition}{$($switching symmetric fluid model$)$} \label{DefFluidEq}
For any initial condition $x(0)$ satisfying Assumption \ref{assInit1}, the fluid model for the symmetric system is the solution
$x \equiv \{x(t) : t \in [0, \Sigma_4 \wedge \Sigma_q)\}$ to the autonomous (time invariant) switching ODE
\bequ \label{SwitchODE}
\dot{x} = f_{\sigma(x)}(x), \quad \sigma(x(t)) = 1,2,3,4;
\eeq
where $f_1$ is defined in \eqref{poolsT1}-\eqref{queuesT1}, $f_2$ is defined in \eqref{poolsT2}-\eqref{queuesT2},
$f_3$ satisfies the equations of $f_1$, but with the labels
of the processes reversed, and $f_4$ satisfies and equations of $f_2$, with the labels of the processes reversed.
The {\em switching times} $\Sigma_i$, $1 \le i \le 4$, are determined by the value of the solution $x(t)$ at time $t$
and are defined in \eqref{switchTimes}.
Furthermore, all points $t \in [0, \Sigma_4 \wedge \Sigma_q)$, except for the switching times, are regular.
\end{definition}
We refer to any specific solution to \eqref{SwitchODE} as a fluid solution or a trajectory.
As was mentioned above, if $x(\Sigma_4)$ satisfies Assumption \ref{assInit1}, then it serves as an initial condition for the following cycle, so that
\eqref{SwitchODE} describes the fluid dynamics beyond the first cycle in an obvious way.
In \S \ref{secFWLLN} we will show that the unique solution $x$ to \eqref{SwitchODE} with a given initial condition arises as the
FWLLN of $\barx^n$ in \eqref{Xn} as $n\tinf$ over any compact subinterval of $[0, \Sigma_q)$, and is therefore a fluid limit.


\subsection{The Queue-Difference Process} \label{secQdiffProc}

Let
$$\Delta (t) \equiv q_2 (t) - q_1 (t), \quad t \ge 0.$$
As indicated in \eqref{T1T2}, at time $T_1$ we have $\Delta(T_1) = \kappa$.
If $\dot{\Delta} (T_1) < 0$, then $\Delta (T_1 + t) < 0$ for all $t$ in some interval $[0,\epsilon]$ for $\epsilon > 0$.
In that case, fluid from queue $2$ stops flowing into pool $1$. At some point  $t_0 \in \I_1$ we may have
that $- \Delta(t_0) = \kappa$, in which case sharing should begin with pool $2$ helping queue $1$, unless $z_{2,1}(t_0) > \tau$,
which means that $x$ will cross the sliding manifold $\SS_{1,2}$ into $\SS_{1,2}^+$.
We now study the difference process over $[0, \Sigma_2)$.

In terms of \eqref{queuesT1exp2}, 
\begin{eqnarray}\label{T1T2exp}
 \Delta (t) 
 & = & \Delta (0)e^{-\theta t} -  \frac{1 + \mu}{\theta}(1 -e^{-\theta t}) \nonumber \\
 && \quad - \left(\frac{(1-\mu)(1- z_{2,1} (0))}{1 - \theta}\right)(e^{-\theta t}  -e^{-t})
 +  \left(\frac{(1 - \mu)\tau}{\mu - \theta}\right)(e^{-\theta t} -e^{-\mu t}).
\end{eqnarray}
\begin{lemma}{$($derivative of $\Delta$ over $\I_1$$)$}\label{lmDiff}
The function $\Delta$ in \eqref{T1T2exp} has a negative derivative on $\I_1$ and is therefore strictly decreasing.
In particular, 
\begin{eqnarray}\label{queuesT1a}
\dot{\Delta} (t) & = &  - \theta \Delta (t)  + \Psi (t), \quad t \in \I_1,
\end{eqnarray}
where $\Delta (t) > 0$ and
\begin{eqnarray}\label{queuesT1b}
\Psi (t) & \equiv & - (1 + \mu) -  (1-\mu)(1- z_{2,1} (0))e^{-t}+ (1-\mu)\tau e^{-\mu t} < 0, \quad t \in \I_1,
\end{eqnarray}
so that $\dot{\Delta}(t) < 0$ and
$$-\Psi_U \le \Psi (t) \le -\Psi_L,$$
where
\bequ \label{PsiBds}
0 <  \Psi_L \equiv 2 \mu - (1-\mu)(1 -\tau) < 2  \equiv \Psi_U < \infty, \quad t \in \I_1.
\eeq
\end{lemma}

\begin{proof}
The expression for the derivative (prior to time $T_1$) follows immediately from \eqref{queuesT1}.
The function $\Psi (t)$ in \eqref{queuesT1b} is strictly negative because
\bes
1 + \mu  > 1-\mu > (1-\mu)\tau e^{-\mu t} \qforallq t \ge 0. \qedhere
\ees
\end{proof}

\begin{coro}{$($equation for $T_1)$}\label{corT1}
The time $T_1$ is well defined as the unique solution $t$ to the equation $\Delta (t) = \kappa$.
\end{coro}


We also have an explicit expression for the difference at time $t$ in terms of its value at time $0$.
\begin{lemma}{$($explicit expression as a function of the initial difference$)$}\label{lmDiff2}
The function $\Delta (t)$ can be represented as
\bequ \label{ode5}
\Delta (t) = \Delta (0)e^{-\theta t} + e^{-\theta t} \int_0^t e^{\theta s} \Psi (s) \, ds, \quad t \in \I_1.
\eeq
where $\Psi (t)$ is defined in {\em \eqn{queuesT1b}} and is independent of $\Delta (0)$.
Thus, $\Delta (t)$ is a strictly increasing function of the initial difference $\Delta (0) > 0$.
In addition, $\Psi (s)$ and $\Delta (t)$ are increasing functions of $z_{2,1} (0)$ and $\tau$.
As a consequence, $T_1$ is strictly increasing function of $\Delta (0)$, $z_{2,1} (0)$ and $\tau$.
Moreover,
\bequ \label{DelBd}
\Delta (0) e^{-\theta t} - \Psi_U \left(\frac{1 - e^{-\theta t}}{\theta}\right) \le \Delta (t) \le \Delta (0) e^{-\theta t} - \Psi_L \left(\frac{1 - e^{-\theta t}}{\theta}\right) \qforallq t \in \I_1,
\eeq
for $\Psi_L$ and $\Psi_U$ in {\em \eqn{PsiBds}}.
\end{lemma}

\begin{proof}
Equation \eqn{queuesT1a} is a classic first-order ordinary differential equation, which is known to have the explicit
solution in \eqn{ode5}, where the second term in \eqn{ode5} is independent of $\Delta (0)$.
\end{proof}

%

From \eqref{queuesT2}, we immediately obtain an expression for the derivative of the queue difference,
which we can apply to show that there is no sharing during $\I_2$.

\begin{lemma} \label{lmDiffAfter} 
The derivative of $\Delta$ on $\I_2$ satisfies
\begin{eqnarray}\label{queuesT2a}
\dot{\Delta} (T_1 + t) & = &  - \theta \Delta (T_1 + t)  + A e^{-\mu t}, \quad 0 \le t \le T_2,
\end{eqnarray}
where $\Delta (T_1) = \kappa$ and
\beql{queuesT2b}
A  \equiv (1- \mu)(z_{1,2} (T_1) - z_{2,1} (T_1)) < 0.
\eeq
Hence, $\dot{\Delta} (t) < 0$, so that $d_{2,1}(t) < 0$ ($q_2 (t) < q_1 (t) + \kappa$) for all $t \in \I_2$.
\end{lemma}

From Lemmas \ref{lmDiff} and \ref{lmDiffAfter} we immediately obtain the following corollary.
\begin{coro}{$($monotonicity of $\Delta(t)$ on $[0, \Sigma_2)$$)$} \label{corDeltaMono}
$\dot{\Delta}(t) < 0$ for all $t \in [0, \Sigma_2)$, so that $\Delta$ is strictly decreasing over that interval.
\end{coro}

We can give an explicit expression for the difference process $\Delta (T_1 + t)$, $t \le T_2$ using \eqref{queuesT2exp}. 
\begin{lemma} \label{lmDiffAfter2}
The function $\Delta (t)$ can be expressed as
\begin{eqnarray}\label{queuesT2e}
\Delta (T_1 + t) & = &  \kappa  e^{-\theta t} + \Phi (t), \quad 0 \le t \le T_2,
\end{eqnarray}
where
\begin{eqnarray}\label{queuesT2f}
\Phi (t) & \equiv & A e^{-\theta t} \int_0^t e^{\theta s} e^{-\mu s} \, ds  = A \left(\frac{e^{-\theta t} - e^{-\mu t}}{\mu - \theta}\right) < 0 \qforallq 0 \le t \le T_2,
\end{eqnarray}
with $A < 0$ in {\em \eqn{queuesT2b}}.
In particular, $\Delta (T_1 + t) < \kappa$ for all $t \in \I_2$, so that there is no active sharing in this interval.
\end{lemma}

\begin{proof}  Just as in Lemma \ref{lmDiff}, we apply the explicit solution to the first-order linear ODE to obtain \eqn{queuesT2e} with
\eqn{queuesT2f}.
\end{proof}

\subsection{Conditions for Finiteness of the Switching Times}

From the definition of $T_1$ in \eqref{T1T2} together with \eqref{DelBd}, we immediately get that $T_1 < \infty$.
Given $T_1$, we can apply \eqref{poolsT2} to obtain an equation for $T_2$. If $T_1$ is sufficiently large so that $z_{2,1}(T_1) > \tau$,
then
$$z_{2,1}(\Sigma_2) \equiv z_{2,1} (T_1 + T_2) = z_{2,1}(T_1)e^{-\mu T_2} = [1 - (1 - z_{2,1} (0))e^{-T_1}]e^{-\mu T_2} = z_{1,2} (0) = \tau,$$
where the last equality follows from the definition of $T_2$.
As an immediate consequence of \eqref{T1T2}, we have explicit formulas for $T_2$:
\bequ \label{T2formula}
T_2 = \frac{\log_{e}{(z_{2,1}(T_1)/\tau)}}{\mu} = \frac{\log_{e}{([1 - (1 - z_{2,1} (0))e^{-T_1}]/\tau)}}{\mu}.
\eeq
It is easy to check whether $z_{2,1} (T_1) > \tau$ so that $T_2 > 0$; see \eqref{poolsT1} above.
It suffices to have
$$e^{-T_1} < 1 - \tau \quad \mbox {or, equivalently,} \quad T_1 > -\log_{e}{(1 - \tau)}.$$

Combining \eqref{poolsT2} with \eqref{T2formula} to obtain an expression for $z_{1,2} (\Sigma_2)$
\bequ \label{z12End}
z_{1,2} (\Sigma_2)  =  \tau  e^{-\mu \Sigma_2}.
\eeq
We can apply \eqn{queuesT2exp} to calculate $q_i (\Sigma_2)$ to verify that $q_i (\Sigma_2) > 0$ for $i = 1,2$,
ensuring that $\Sigma_q \ge \Sigma_2$.
If $x(\Sigma_2)$ satisfies the conditions of $x(0)$ in Assumption \ref{assInit1} but with the labels of the processes reversed,
then we can again apply \eqref{poolsT2} (with the labels reversed) to conclude that $T_3 < \infty$.
If $T_3 > 0$, then $T_4$ satisfies a similar equation to \eqref{T2formula}, but with $T_3$ replacing $T_1$ and $z_{1,2}(T_3)$
replacing $z_{2,1}(T_1)$, provided that $z_{1,2}(T_3) > \tau$. 


\section{Qualitative Analysis} \label{secQualitative}

Just as for the stochastic system, it is important to identify the possible equilibrium behavior of the fluid models,
as well as its long-run behavior. We start with formally defining the relevant equilibria for our fluid model
and then stating the main results regarding fluid model.

Recall that the state space of the fluid model is $\SS$ in \eqref{space}.
For the general discussion regarding the long-run behavior of the system, we consider all the possible initial conditions,
and therefore Assumption \ref{assInit1} is not enforced in this section. Specifically, any $\gamma \in \SS$ is allowed to be an initial condition.

\begin{definition}{$($stationary point$)$}
A point $x^* \in \SS$ is stationary for \eqref{SwitchODE} if $x(0) = x^*$ implies that $x(t) = x^*$ for all $t \ge 0$.
\end{definition}

\begin{definition}{$($periodic equilibrium$)$} \label{DefEquilibrium}
A non-constant solution $u^* \equiv \{u^*(t) : t \ge 0\}$ to \eqref{SwitchODE} is a (nontrivial) periodic equilibrium, if there exists a time $T > 0$
such that $u^*(t + T) = u^*(t)$ for all $t \ge 0$. The smallest such $T$ is called the period of $u^*$.
\end{definition}
Note that a solution initialized at a stationary point $x^*$ satisfies $x(t+T) = x(t) = x^*$
for all $t \ge 0$ and all $T > 0$, which is why we require that $u^*$ is not a constant.

\paragraph{Lyapunov Stability of a Stationary Point.}
We will show that for any set of parameters, the fluid model in Definition \ref{DefFluidEq} has a unique stationary point and
that, in some cases, there also exists a unique periodic equilibrium. We will then study the stability properties of the fluid model.
There are three types of stability notions corresponding to stationary points that are relevant for us.

For a stationary point $x^*$, let $\sS_{x^*} \subseteq \SS$ be the stability region of $x^*$, i.e., if $x(0) \in \sS_{x^*}$, then $x(t) \ra x^*$ as $t\tinf$.
Note that, by the definition of $x^*$, $\sS_{x^*}$ is not empty because it contains $x^*$.
\begin{definition}{$($Lyapunov stability$)$} \label{defStableStPt}
A stationary point $x^*$ is said to be
\bi
\item {\bf unstable}, if $\sS_{x^*} = \{x^*\}$;
\item {\bf asymptotically stable}, if $\sS_{x^*}$ contains an open neighborhood of $x^*$;
\item {\bf globally asymptotically stable}, if $\sS_{x^*} = \SS$.
\ei
\end{definition}
We note that for our system with the state space $\SS$ in \eqref{space}, subsets of $\SS \subsetneq \RR_6$ are considered open
in the relative topology induced on $\SS$ by the topology of $\RR_6$. In particular, open subsets can contain points on the boundary
of $\SS$ in $\RR_6$.

\paragraph{Stability of a Periodic Equilibrium.}
When a periodic equilibrium $u^*$ exists, it is possible
for the fluid model to oscillate indefinitely, at least when the initial condition is taken to be on the periodic equilibrium trajectory.
However, we would like to know if the periodic equilibrium is also asymptotically stable in some sense, namely, if there exists a set $\sS_{u^*} \subseteq \SS$
such that, if $x(0) \in \sS_{u^*}$, then $x(t)$ converges to the periodic equilibrium. 
We note that convergence to periodic equilibrium cannot hold in the Lyapunov sense, as in Definition \ref{defStableStPt},
because there would typically be a time shift between the converging solution and the periodic-equilibrium solution.
We therefore say that a solution $x$ converges to a periodic equilibrium $u^*$ if its image ``spirals'' toward
the image of $u^*$ as time increases. (By spiraling we mean that the image of $x$ keeps moving in the direction of $u^*$ and gets closer
to it as time increases; see Lemma \ref{lmAsyCycle} below.)

Consider a switching dynamical system $\dot{x} = f_\sigma(x)$ (not necessarily \eqref{SwitchODE}).
The standard way of proving that a periodic equilibrium $u^*$ (assuming one exists) with period $T$
is stable, is to consider the intersection point $\tilde{u}$ of $u^*$ with a switching surface $\sM$,
and show that any trajectory $x$ that is initialized on $\sM$
sufficiently close to $\tilde{u}$, will reach $\sM$ again after a time that is approximately equal to the period $T$ of $u^*$.
If, in addition, the intersections of $x$ with $\sM$ converge to $\tilde{u}$,
then $u^*$ is asymptotically stable; see, e.g., page 121 in \cite{HDSbook}.

To rigorously define the above asymptotic stability notion, and show that it indeed implies the ``spiraling motion'' of solutions that are initialized
sufficiently close to a periodic equilibrium, we first make a simple observation:
When there are $N > 1$ switching surfaces $\sM_i$, $1 < i \le N$, that are intersected by a {\em stable} periodic equilibrium $u^*$,
the intersections of $x$ with $\sM_i$, as well as the values of $x$ at those intersection points,
will converge to the intersection points of $u^*$ with $\sM_i$ and the values of $u^*$ at these epochs, respectively, for each $i \le N$.
Since this is the case for our system, we define asymptotic stability in term of all four switching surfaces and
the corresponding switching times. To avoid introducing more notation, the definition is given for our system directly.

To that end, let $\sP_{u^*}$ denote the image of a periodic equilibrium $u^*$ having period $T$;
\bes
\sP_{u^*} \equiv \{\gamma \in \SS : \gamma = u^*(t),\; 0 \le t < T\}.
\ees
Since $u^*(0) = u^*(T)$, the set $\sP_{u^*}$ is an {\em invariant} set, namely,
if $y_0 \in \sP_{u^*}$ and $y$ is the unique solution to $\dot{y} = f_\sigma(y)$ in \eqref{SwitchODE}
with initial condition $y(0) = y_0$, then $y(t) \in \sP_{u^*}$ for all $t > 0$.

Let $x$ be a solution to \eqref{SwitchODE} with $x(0) \notin \sP_{u^*}$ and $\Sigma_q = \infty$ (so that $x$ oscillates indefinitely; we will show in Theorem \ref{thEndless}
below that such solutions exist).
Note that if $x$ is an oscillating solution to \eqref{SwitchODE}, then there exists a $t_1 \ge 0$ such that $x(t_1)$ satisfies the conditions in Assumption \ref{assInit1}.
Due to the time-homogeneity of $x$ we can restart the ODE at the first time $t_1 \ge 0$ for which $x(t_1)$ satisfies Assumption \ref{assInit1} by taking $x(0) = x(t_1)$.
Then the solution $\{x(t) : -t_1 \le t < \infty\}$ satisfies Assumption \ref{assInit1} at time $0$.

For $T_i$ and $\Sigma_i$ in \eqref{T1T2} and \eqref{switchTimes}, let $T^{(k)}_i$ and
$\Sigma^{(k)}_i$ be the value of holding time $T_i$ and switching time $\Sigma_i$, respectively, in the $k^{th}$ cycle of $x$, where
\bes
\Sigma^{(1)}_0 \equiv t_1 \quad \mbox{(so that $x(\Sigma_0^{(1)}) \equiv x(0)$ by definition)} \qandq \Sigma_0^{(k+1)} \equiv \Sigma^{(k)}_4, \quad k \ge 1.
\ees
Let $T^*_j$ denote holding time $j$, $1 \le j \le 4$, and $\Sigma^{*(k)}_i$ denote switching time $i$, $0 \le i \le 4$, in the $k^{th}$ cycle of a periodic equilibrium $u^*$,
with $\Sigma^{*(0)}_0 \equiv 0$ and $\Sigma^{*(k+1)}_0 \equiv \Sigma^{*(k)}_4$, $k \ge 1$.
Similarly, for an oscillating solution $x$, let $T^{(k)}_j$, denote holding time $j$, $1 \le j \le 4$, and $\Sigma^{(k)}_i$ denote switching time $i$, $0 \le i \le 4$,
in the $k^{th}$ cycle of $x$, $k \ge 1$, where $\Sigma^{(0)}_0 \equiv 0$ and $\Sigma^{(k+1)}_0 \equiv \Sigma^{(k)}_4$, $k \ge 1$.


\begin{definition}{$($asymptotically stable periodic equilibrium$)$} \label{defStableCycle}
A periodic equilibrium $u^*$ having period $T$ is said to be asymptotically stable if there exists an open subset $\sS_{u^*}$
of\, $\SS$ which contains $\sP_{u^*}$ such that, if $x(0) \in \sS_{u^*}$, then for $1 \le i \le 4$ and any $t > 0$,
\bequ \label{ConvDef}
\lim_{k\tinf} T^{(k)}_i = T^*_i \qandq \lim_{k\tinf} \sup_{0 \le s \le t}\|x(\Sigma^{(k)}_0 + s) - u^*(\Sigma^{*(k)}_0 + s)\| = 0.
\eeq
\end{definition}

\section{Asymptotic Behavior of the Fluid Model} \label{secEquilibria}

In this section we establish results about the asymptotic behavior of the underloaded switching fluid model in \eqn{SwitchODE}.
We show that there always is the underloaded stationary point equilibrium, to which the fluid model converges if it does not oscillate indefinitely.
We show that there exists an overloaded periodic equilibrium if it oscillates indefinitely, and provide sufficient conditions for endless oscillations.
For the discussion of equilibria, we no longer assume initial conditions in Assumption \ref{assInit1}; we allow arbitrary initial conditions in the state space $\SS$.  We
also consider the system after time $\Sigma_q$
in \eqn{Sigmaq}.

\subsection{Existence and Asymptotic Stability of a Unique Stationary Point} \label{secStatPt}

If there is no sharing actively taking place on an interval $[0,T]$, 
then the stochastic system decomposes into two independent $M/M/n+M$ (Erlang-A) queuing systems.
Let $Y^n_i(t) := Q^n_i(t) + Z^n_{i,i}(t)$ denote the total number of customers in each of these systems
and $\bar{Y}^n_i := Y^n_i/n$, $i = 1,2$. Then the fluid  model for $\bar{Y}^n$ in the symmetric case we consider
is the solution of the ODE
\bes
\dot{y}_i = \lm - \mu(1 \wedge y_i) - \theta (y_i - 1)^+, \quad i = 1,2,
\ees
where $a^+ \equiv \max\{a, 0\}$. In this case we have the following elementary, but important, result.

\begin{theorem}\label{thAsymptStableStatPt}
If $q_i(0) \le \kappa$,
then no sharing will ever begin in the fluid model and $x (t) \ra x^*_0$ as $t\tinf$, where
\bequ \label{statPt}
x^*_0 \equiv (q^*_1, q^*_2, z^*_{1,1}, z^*_{1,2}, z^*_{2,1}, z^*_{2,2}) = (0,0,\lm,0,0,\lm).
\eeq
Hence, $x^*_0$ is an asymptotically stable stationary point.
\end{theorem}

\begin{proof}  No sharing will ever occur because
$q_i = (y_i-1)^+$, and if $y_i(t) > 1$, so that the queue is positive, then $y_i(t)$ is decreasing at $t$, $i = 1,2$.
(Recall that $\lm < \mu = 1$.) Hence, even if $d_{i,j}(0) = \kappa$ for $(i,j) = (1,2)$ or $(i,j) = (2,1)$, then $d_{i,j}(t) < \kappa$
for any $t > 0$ in some right-neighborhood of $0$. It follows that, if $z_{i,j}(0) > 0$, $i \ne j$, then $z_{i,j}$ is strictly decreasing,
which implies that the service capacity in pool $j$ is increasing. In turn, $q_j$ must keep decreasing as long as it is strictly positive.
Finally, since $y_i$ is strictly decreasing as long as it is larger than $\lm$ and is strictly increasing otherwise, we have
\bequ \label{yConv}
y_i(t) \ra \lm \qasq t\tinf.\qedhere
\eeq
\end{proof}

\begin{remark}\label{remUnstableStatPt}
{\em
Having $x^*_0$ in \eqref{statPt} be an asymptotically stable stationary point depends critically on the assumption that $\kappa > 0$.
If, instead, $\kappa = 0$, then it is possible for
$x^*_0$ to be an unstable stationary point, so that $x$ oscillates indefinitely for any initial condition $x(0) \ne x^*_0$.
Instability of $x^*_0$ has important consequences for the stochastic system $X^n$,
since stochastic fluctuations may trigger undesirable sharing even if the system is initialized at the neighborhood of $x^*_0$.
Therefore, stochastic fluctuations can quickly lead to {\em fluid-scaled fluctuations}, namely, to an oscillatory behavior.
See the simulations in \S \ref{secSim} below.
The moral is that there is a need to ensure that the activation thresholds in the (finite) stochastic system are large enough to be considered positive in fluid scale.
The size of the stochastic fluctuations of critically-loaded pools with no sharing can be estimated from the established heavy-traffic limit
approximations for the Erlang-A model in \cite{Garnett02}.
}\end{remark}

Ideally, $x^*_0$ in \eqref{statPt} would be a globally asymptotically stable stationary point for the fluid model, since
the system is underloaded ($\lm < 1$) and we want no sharing to take place, and indeed that will be the case with appropriate controls.
However, here we are interested in fluid models with poorly chosen controls.  Then solutions to \eqn{SwitchODE} need not converge to $x^*_0$,
so that $\sS_{x_0^*}^c \ne \phi$, where, for a set $A$, $A^c$ denotes the complement of $A$ and $\phi$ denotes the empty set.


Let $\SS^* := \{\gamma^* \in \SS : \mbox{$\gamma^*$ is a stationary point}\}$. Of course, $\SS^*\ne \phi$ because $x^*_0 \in \SS^*$.

\begin{theorem} \label{thUniqueStatPt}
$\SS^* = \{x^*_0\}$ for $x^*_0$ in \eqref{statPt}; i.e., $x^*_0$ is the unique stationary point of the switching fluid model \eqref{SwitchODE}.
\end{theorem}

\begin{proof} Supppose that
$$\gamma^* = (\gamma^*_i, \gamma^*_{i,j} ; i,j = 1,2) \in \SS^*\quad \mbox{such that } \gamma^* \in \SS_{1,2} \cup \SS_{1,2}^+,$$
so that $\gamma^*_1 \ge \kappa$.
Consider the fluid model initialized at $\gamma^*$, i.e., $x(0) = \gamma^*$. If $z_{2,1} (0) = \gamma^*_{2,1} > 0$, then by the rules of FQR-ART,
$\dot{z}_{2,1}(0) = -\mu_{2,1}z_{2,1}(0) < 0$,
implying that $z_{2,1}$ is strictly decreasing. It follows that $\gamma^*_{2,1} = 0$, so that $\gamma^*_{1,1} = 1$
(because $\gamma^*_1 \ge \kappa > 0$).
But then
\bes
\dot{q}_1(0) = \lm - \mu_{1,1}\gamma^*_{1,1} - \theta q_1(0) < \lm - 1 < 0,
\ees
which contradicts the supposition that $\gamma^*$ is a stationary point.  Hence,
$\SS^* \cap (\SS_{1,2} \cup \SS_{1,2}^+) = \phi$.
Similar arguments apply to $\SS^*\cap (\SS_{2,1} \cup \SS_{2,1}^+)$.
The same reasoning for $\gamma^* \in \SS^*\cap \SS_{1,2}^- \cap \SS_{2,1}^-$ implies that $\gamma^*_{1,2} = \gamma^*_{2,1} = 0$
and $\gamma^*_1 = \gamma^*_2 = 0$.  Then the arguments leading to \eqref{yConv} show that $\gamma^* = x^*_0$.
Hence, we conclude that $\SS^*= \{x^*_0\}$.
\end{proof}

%

Having established Theorem \ref{thUniqueStatPt},
We refer to $x^*_0$ in \eqn{statPt} as the {\em stationary point with no sharing}, or simply as the stationary point.

\subsection{Only Two Possibilities}\label{secTwo}

We now show that there are only two possibilities for the asymptotic behavior.
Let $\sO \subset \SS$ be the set of points such that,
if $x(0) \in \sO$, then the solution $x$ to \eqn{SwitchODE} switches infinitely often as $t\tinf$, i.e., it oscillates indefinitely.

\begin{theorem} \label{thStableStatPt}
$\sO^c = \sS_{x_0^*}$ for $x^*_0$ in \eqref{statPt}; i.e., if $x(0) \in \sO^c$, then $x(t) \ra x^*_0$ as $t\tinf$.
\end{theorem}


\begin{proof}
Since $x(0) \in \sO^c$ there exists a time $t_0 \ge 0$ such that $x(t) \notin \SS_{1,2}^+ \cup \SS_{2,1}^+$ for all $t \ge t_0$.
If $x(t) \in \SS_{i,j}^-$ for all $t \ge t_0$, then
$$\dot{z}_{i,j}(t) = -\mu z_{i,j}(t), \quad \mbox{so that } z_{i,j}(t) = z_{i,j}(t_0) e^{-\mu (t-t_0)}, \quad t \ge t_0.$$
Then both $z_{1,2}$ and $z_{2,1}$ converge to $0$, and it is easy to see from \eqref{qODE} (recall that there is no new sharing
taking place) that both queues will reach $0$ in finite time. Then, after $q_i$ reaches $0$, all arriving fluid moves immediately
into service, so that $\dot{z}_{2,2} = \lm - z_{2,2}$, and we see that $z_{2,2}(t) \ra \lm$ as $t\tinf$.

Now suppose that $x \in \SS_{1,2}$ over an interval $I$. If $z_{2,1} > \tau$ over $I$, then no fluid flows from $q_1$ to pool $2$,
so that both queues evolve independently according to \eqref{qODE}. Since $z_{1,2}$ and $z_{2,1}$ are strictly decreasing over $I$,
the same arguments given above apply in this case.
Therefore, assume that $z_{2,1} \le \tau$ over an interval $J \subseteq I$ so that sharing is allowed.
By Assumption \ref{assRates}, $q_1$ is strictly decreasing on $J$, and the sliding motion implies that $\dot{q}_1(t) - \dot{q}_2(t) = 0$,
so that $q_2$ is strictly decreasing as well (at exactly the same rate as $q_1$).
Now, some of the service capacity of pool $2$ is given to queue-$1$ fluid at any point, so that
\bes
\dot{q}_1(t) < \lm - z_{1,1}(t) - \mu z_{2,1}(t) - \theta q_1(t) \qandq \dot{q}_2(t) > \lm - z_{2,2}(t) - \mu z_{1,2}(t) - \theta q_2(t), \quad t \in J.
\ees
Recalling that $q_1(t) = q_2(t) + \kappa$ and $z_{i,i}(t) = 1-z_{j,i}(t)$ for $t \in J$, we have
\bes
0 = \dot{q}_1(t) - \dot{q}_2(t) < (1-\mu) (z_{2,1}(t) - z_{1,2}(t)) - \theta \kappa < (1-\mu) (z_{2,1}(t) - z_{1,2}(t)),
\ees
so that $z_{1,2}(t) < z_{2,1}(t)$. It follows that $z_{1,2}(t) \le \tau$ and is decreasing on $J$.
In particular, both queues continue decreasing after the sliding motion is over.

The same arguments give that, if $x$ ever slides on $\SS_{2,1}$, then both queues are strictly increasing to $0$.
Hence, the processes $z_{1,2}$ and $z_{2,1}$ never increase above $\tau$ during sliding motion, so that both queues
are strictly decreasing to $0$. After $q_i$ hits $0$, $z_{j,i}$ decreases monotonically to $0$ and $z_{i,i}$ converges to $\lm$.
\end{proof}
%


\subsection{Existence of a Periodic Equilibrium} \label{secPeriodEq}

Theorem \ref{thStableStatPt} shows that a solution $x$ to \eqn{SwitchODE} either converges to $x^*_0$
or oscillates indefinitely.
We now consider what happens if the solution oscillates indefinitely.


\begin{theorem} \label{thPeriodic}
If $\sO \ne \phi$, then there exists a periodic equilibrium $u^* \equiv \{u^*(t) : t \ge 0\}$ to \eqref{SwitchODE}.  In particular, if $\sO \ne \phi$, then there exists
a initial state vector $x(0)$ satisfying Assumption \ref{assInit1} such that $x(0) \in \sO$ and,
for that $x(0)$, $\Sigma_q > \Sigma_2$ and
\bes
(q_1 (\Sigma_4), q_2 (\Sigma_4), z_{2,1} (\Sigma_4)) = (q_1 (0), q_2 (0), z_{2,1} (0)),
\ees
which implies that $T_3 = T_1$, $T_4 = T_2$, so that $\Sigma_4 = 2 \Sigma_2$,
\bequ \label{exist}
\bsplit
(q_1 (\Sigma_4), q_2 (\Sigma_4), z_{2,1} (\Sigma_4)) & = (q_1 (2(T_1 + T_2)), q_2 (2(T_1 + T_2)), z_{2,1} (2(T_1 + T_2))) \\
& = (q_2 (T_1 + T_2), q_1 (T_1 + T_2), z_{1,2} (T_1 + T_2)) \\
& = (q_1 (0), q_2 (0), z_{2,1} (0)).
\end{split}
\eeq
\end{theorem}

It is important that the condition in Theorem \ref{thPeriodic} can be satisfied.
Hence, we also establish the following result, which may be considered harder than Theorem \ref{thPeriodic}.
\begin{theorem} \label{thEndless}
There exist parameter values for \eqref{symModel} and initial conditions satisfying Assumption \ref{assInit1}
for which $\sO \ne \phi$.
\end{theorem}

\subsection{Proofs of Theorems \ref{thPeriodic} and \ref{thEndless}}

To establish these results, we exploit an algorithm for efficiently computing a solution to the
switching model in \eqn{SwitchODE} and efficiently calculating the periodic equilibrium if it exists.
The algorithm improves on the piecewise numerical solution of the piecewise ODE in \eqn{SwitchODE} by exploiting the exact formulas
in \S \ref{secOscFluid}.  We can recursively calculate the values at the switching times $\Sigma_i$ and then afterwards calculate
the trajectory in between.
By iterating, we can easily determine numerically if the
solution converges to the stationary point or not.
 Numerical experience indicates that if the solution oscillates
indefinitely, then it rapidly converges to a periodic equilibrium.  In particular, the algorithm identifies the periodic
equilibrium.  However, more is required to provide a mathematical proof of existence, uniqueness and convergence.

\subsubsection{An Efficient Algorithm for The Periodic Equilibrium} \label{secAlg}

A periodic equilibrium $u^*$ has an important closure property:
If $u^*(t)$ satisfies Assumption \ref{assInit1} for some $t$, then $u^*(t + \Sigma_4) = u^* (t)$.
Due to the symmetry of our model,
we can relate the system state at time $t +\Sigma_2$ to the system state at time $t$.
The state at time $t + \Sigma_2$ should coincide with the state at time $t$
with the labels reversed.  
That is, we should have
\begin{eqnarray}\label{closure}
&& q_1 (t+ \Sigma_2)  =  q_2 (t) >0 , \quad  q_2 (t + \Sigma_2) = q_1 (t) > 0 \nonumber \\
&& z_{1,2} (t + \Sigma_2)   =  z_{2,1} (t) \qandq  z_{2,1} (t +\Sigma_2)  =  z_{1,2} (t) = \tau.
\end{eqnarray}
with the condition that the pools remain full throughout:
$$z_{1,1} (s) + z_{2,1} (s) = 1 \qandq z_{2,2} (s) + z_{1,2} (s) = 1, \quad 0 \le s \le t + \Sigma_2.$$
(Observe that the labels of the processes in the second equality in \eqref{exist} are reversed.)
If indeed we can establish the closure property in \eqref{closure}, then we will have proved that there exists a periodic equilibrium.

It is natural to search for the equilibrium by iterating:
We pick a candidate initial vector $x_3(0) \equiv (q_1 (0), q_2 (0), z_{2,1} (0))$, letting $z_{1,2} (0) = \tau$, so that Assumption \ref{assInit1} holds.
We then solve for $T_1$, $T_2$, and $(q_1 (T_1 + T_2), q_2 (T_1 + T_2), z_{1,2} (T_1 + T_2))$, as indicated above.
we then redefine $(q_1 (0), q_2 (0), z_{2,1} (0))$ to be $(q_2 (T_1 + T_2), q_1 (T_1 + T_2), z_{1,2} (T_1 + T_2))$ and repeat the calculation.

If at some iteration we obtain an unreasonable value for $x_3$, e.g., $q_i < 0$, $i=1$ or $i=2$, or $\Delta \le \kappa$, then the algorithm is stopped
and we conclude that the solution corresponding to the initial condition we chose converges to $x^*_0$ (due to Theorem \ref{thStableStatPt}).
However, a pathological case has $\Delta > \kappa$ for all iterations, but $\Delta \ra \kappa$.
Let $\Delta^*$ and $T^*_1$ denote the limit of $\Delta$ and $T_1$ when the algorithm is iterated indefinitely.
Observe that $\Delta^* = \kappa$ implies $T^*_1 = 0$, so that the corresponding limiting solution $u^*$ is necessarily a constant function.
This case is clearly a pathology, due to the uniqueness of the stationary point $x^*_0$.
The following lemma ensures that such a pathological behavior of the algorithm is not possible.
In particular, if at some iteration of the algorithm $\Delta$ is too close to $\kappa$, then this is also the last iteration
\begin{lemma}\label{lmDelbdd}
There exists $\ep_\kappa > 0$ such that, if $\kappa < \Delta(0) < \kappa+\ep_\kappa$, then $x(\Sigma_2) > -\kappa$.
In particular $x(0) \in \sO^c$, so that $x(t) \ra x^*_0$ as $t \tinf$.
\end{lemma}

\begin{proof}
By Lemma \ref{lmDiff}, $\Delta$ is bounded from above by the linear function $-\Psi_L$. Hence, for any $\delta_1 > 0$ we can find $\ep_1 > 0$
such that, if $\kappa < \Delta(0) < \kappa + \ep_1$, then $0 < T_1 < \delta_1$.
The explicit expressions of $z_{2,1}$ in \eqref{poolsT1} and $T_2$ in \eqref{T2formula} show that,
for any $z_{2,1}(0)$ and $\delta_2 > 0$, we can choose $\delta_1$ sufficiently small to ensure that $T_2 < \delta_2$ (even if $T_2 > 0$).
Hence, for any $\delta > 0$, we can find $\ep > 0$ such that, if $\kappa < \Delta(0) < \kappa + \ep$, then $\Sigma_2 < \delta$,
by first choosing $\delta_2$ and then an appropriate $\delta_1$ to ensure that $\delta_1 + \delta_1 \le \delta$.
The continuity of $\Delta$ implies that there exists a $\delta_\kappa > 0$ such that, if $\Sigma_2 < \delta_\kappa$, then $\Delta(\Sigma_2) > -\kappa$.
It follows that for all $t$ in some right neighborhood of $\Sigma_2$ both $z_{1,2}(t)$ and $z_{2,1}(t)$ are strictly less than $\tau$, so that both queues are strictly
decreasing.

Now, if $x$ ever hits $\SS_{i,j}$, $(i,j) = (1,2)$ or $(i,j) = (2,1)$, after time $\Sigma_2$, then it can not cross it to $\SS_{i,j}^+$.
To see this, suppose for example that $x$ hits $\SS_{2,1}$ at some time $t > \Sigma_2$.
Since $x$ evolves according to the ODE's \eqref{poolsT1} - \eqref{queuesT1} when in $\SS_{2,1}^+$,
the derivative of $\Delta(t) \in \SS_{2,1}^+$ is strictly negative; see Lemma \ref{lmDiff}. Moreover, sharing is allowed
to start immediately because $z_{1,2} < \tau$.
Therefore, if $\Delta(0) < \kappa + \ep_\kappa$, then $x(0) \in \sO^c$, so that $x(t) \ra x^*_0$ as $t\tinf$ by Theorem \ref{thStableStatPt}.
\end{proof}

Let $\Delta^{(k)}$ be the value of $\Delta$ at the $k^{th}$ iteration of the algorithm.
It follows from Lemma \ref{lmDelbdd} that
\begin{coro} \label{corDelBdd}
If $x(0) \in \sO$, then $\Delta^{(k)} \in [\kappa+\ep_\kappa, \lm/\theta]$, $k \ge 1$, for $\ep_\kappa > 0$ in Lemma \ref{lmDelbdd}.
\end{coro}

\subsubsection{Proof Theorem \ref{thEndless}} \label{secProofEndlessA}
\begin{proof}
We first impose conditions on the model parameters and initial conditions so that
the iterative algorithm in \S \ref{secAlg} mapping the initial state vector $x_3(0) \equiv (q_1 (0), q_2 (0), z_{2,1} (0))$ into
the state vector $x_3(\Sigma_2) \equiv (q_1 (\Sigma_2), q_2 (\Sigma_2), z_{1,2} (\Sigma_2))$ and then iterated again to map $x_3(0)$ into
$x_3(\Sigma_4) \equiv (q_1 (\Sigma_4), q_2 (\Sigma_4), z_{2,1} (\Sigma_4))$
is a map of the convex compact subset $\SS_\ep$ of the Euclidean space $\RR_3$ into itself, where $\SS_\ep$ is the subset
$\SS_\ep \equiv[\ep, \lm/\theta] \times [\ep, \lm/\theta] \times [0,\tau]$ for some $\ep > 0$.

For that purpose, we introduce lower and upper bounds on the initial queue difference $\Delta (0)$,
\bequ \label{initDiffBdsA}
0 < \kappa < \Delta_L (0) \le \Delta (0) \equiv q_2 (0) - q_1 (0) \le \Delta_U (0) < \infty,
\eeq
and assume that the smaller queue length $q_1 (0)$ is bounded below as well as above by
\beql{initQ1bdA}
0 < q_1^L (0)  \le q_1 (0) \le q_1^U (0) < \frac{\lambda}{\theta} < \infty,
\eeq
both consistent with Assumption \ref{assInit1}.

We can apply \eqn{DelBd} in Lemma \ref{lmDiff2} to establish upper and lower bounds on $T_1$, as shown in
Corollary \ref{corT1bdsFirst}.  Those bounds are
\beql{T1BdA}
T_1^L  \equiv  \left(\frac{1}{\theta}\right) \log{\left(\frac{\theta \Delta_L (0) + \Psi_U}{\theta \kappa + \Psi_U}\right)}
\le T_1 \le T_1^U  \equiv  \left(\frac{1}{\theta}\right) \log{\left(\frac{\theta \Delta_U (0) + \Psi_L}{\theta \kappa + \Psi_L}\right)}
\eeq
where $\Delta_L (0)$ and $\Delta_U (0)$ come from \eqn{initDiffBdsA} and $\Psi_U$ and $\Psi_L$ are upper bounds on $\Psi$ in
\eqn{queuesT1b} and \eqn{PsiBds}.
We then impose an upper bound on $\tau$ by requiring $\tau < 1 - e^{-T^L_1}$, which imposes an upper bound on $T_2$, i.e.,
\beql{un5A}
T_2 \le T_2^{U} \equiv \frac{\log_e{(z_{2,1} (T_1^{U})/\tau)}}{\mu}.
\eeq
If, in addition,
\beql{un8A}
q_1^L (T_1 + T_2) \equiv \left[\frac{\lambda}{\theta} -
\left(\frac{\lambda}{\theta} - q^L_1 (0)\right)\left(\frac{\theta \kappa +2}{\theta \Delta^L (0) +2}\right)\right]e^{-\theta T_2^{U}}
 > \left(\frac{1-\lambda}{\theta}\right)(1 - e^{-\theta T_2^{U}}),
\eeq
then the two queue lengths both remain positive throughout the interval $[0,T_1 + T_2]$ and $q_1( T_1 + T_2) \ge q_1^L( T_1 + T_2)$ in \eqn{un8A},
as shown in Lemma \ref{lmT1un4}.  (If necessary, we redfine $q_1^L (0)$ so that $q_1^L (T_1 + T_2) \ge q_1^L$ as well as \eqn{initQ1bdA}.)
Finally, if
\bequ \label{initDiffBdsB}
0 < \kappa < \Delta_L (0) \le \Delta (T_1 + T_2) \equiv q_2 (0) - q_1 (0) \le \Delta_U (0) < \infty,
\eeq
then we can iterate without limit, with $\Sigma_q = \infty$.  Condition \eqn{initDiffBdsB} can be checked after the first iteration.
However, sufficient conditions for \eqref{initDiffBdsB} to hold without performing the first iteration are given in
Lemma \ref{lmT1un4}. 
Numerical examples confirm that all these conditions can be satisfied, thus proving Theorem \ref{thEndless}.
\end{proof}

\subsubsection{Proof of Theorem \ref{thPeriodic}}
\begin{proof}
For a solution $x$ with $x(0) \in \sO$, $\Sigma_q = \infty$, so that the algorithm can be iterated indefinitely.
In each iteration, the algorithm acts as a map of the vector $x_3(0) = (q_1(0), q_2(0), z_{2,1}(0))$ to $x_3(\Sigma_4)$
(with $x_3(\Sigma_4)$ serving as the initial condition for the following iteration). Therefore, the algorithm maps the compact and convex set
$[0, \lm/\theta] \times [\kappa, \lm/\theta] \times [0,\tau]$ into itself.
As long as the solution oscillates, we can restrict attention to
the two-dimensional process $x_2 \equiv (\Delta, z_{2,1})$, because $\Delta(0) = \Delta(\Sigma_4) = \kappa$.
In particular, at each iteration of the algorithm we compute $\Delta(\Sigma_2)$ and use it as the initial condition for the next iteration.

Corollary \ref{corDelBdd} implies that for this two-dimensional process $x_2$, the algorithm acts as a map from the space
$\SS_\kappa \equiv [\kappa + \ep_\kappa, \lm/\theta] \times [0, \tau]$ into itself, where $\ep_\kappa > 0$.
The explicit solution to the ODE \eqref{SwitchODE} over $[0, \Sigma_4]$ and to $\Delta$ in \eqref{ode5} and \eqref{queuesT2e}
shows that this map is continuous.
Hence, by Brouwer's fixed point theorem (e.g., Theorem 5.28 in \cite{Rudin91}) there exists a fixed point
to this map in the set $\SS_\kappa$. That fixed point cannot be also a fixed point of \eqref{SwitchODE},
due to Theorem \ref{thUniqueStatPt}, i.e., due to the uniqueness of $x^*_0$.
It follows that there exists a solution to \eqref{SwitchODE} satisfying \eqref{closure} which is not a constant.
Necessarily, such a solution is a non-trivial periodic equilibrium.
\end{proof}

\subsection{Conjectured Bi-Stability}

Recall that $\sS_{x_0^*}$ is the stability set of $x^*_0$ in Definition \ref{defStableStPt}
and $\sS_{u^*}$ denotes the stability set of the periodic equilibrium $u^*$, when it exists, in Definition \ref{defStableCycle}.
By Theorem \ref{thStableStatPt}, $\sS_{x_0^*} = \sO^c$ (the complement of $\sO$),
so that any fluid solution that {\em does not} oscillate indefinitely must converge to $x^*_0$,
and it clearly holds that $\sS_{u^*} \subseteq \sO$.
We conjecture that $\sS_{u^*} \supseteq \sO$ as well, so that $\sS_{u^*} = \sO$. Formally,
\begin{conjecture} \label{conj}
If $x(0) \in \sO$, then there exists a unique periodic equilibrium $u^*$ and $x$ converges to $u^*$ as in \eqref{ConvDef}.
Therefore, $\sS_{x_0^*} \cup \sS_{u^*} = \SS$, namely the fluid model is bi-stable with all fluid trajectories converging to one of the two equilibria
as $t \tinf$.
\end{conjecture}

Extensive numerical trials, some of which are presented in \S \ref{secNumeric} below, indicate that Conjecture \ref{conj} holds.
More importantly, we next derive an approximating dynamical switching system to \eqref{SwitchODE} which is shown to be bi-stable.

\section{Approximating Dynamical System} \label{secApproxDS}

Since we were unable to fully characterize the asymptotic behavior of our initial fluid model, we now develop an approximating fluid model
that can be analyzed more easily; i.e., for which we can establish bistability and calculate the two equilibria.
The approximating system is easier to analyze because
it is essentially a one-dimensional system at the switching times.  However, there are discontinuities at some of the switching times, so the
approximating fluid model is a dynamical system with jumps (alternatively, it can be represented as a hybrid system with jumps);
see \cite{JumpODE} and \cite{stewartRigid}.
The latter reference provides a general framework for defining and analyzing
solutions for dynamical systems with jumps (see \S 1.5 of \cite{stewartRigid}), but the relative simplicity of our approximation
obviates the need for a general theory.
Numerical examples confirm that the
approximating system serves as a useful approximation for the original fluid model, allowing us to rapidly compute a periodic equilibrium.

The approximation is obtained in five steps:  First,
we approximate the solution $x$ to \eqref{SwitchODE} by a solution $x^a$ to
\bequ \label{SwitchODEa}
\dot{x}^a = f_{\sigma}(x^a, \theta^a, \tau^a),
\eeq
for a given initial condition $x^a(0)$, where we supplement the argument $x^a$ of $f_{\sigma}$ in \eqref{SwitchODE}
by the abandonment rate $\theta^a$ and the control parameter $\tau^a$ of the approximating system.
Second, we assume that there is no abandonment, i.e., we let $\theta^a = 0$.
Third, approximate $\tau$ by $0$ on the first and third subintervals, i.e.,
\bequ \label{tau^a}
\tau^a  \equiv
\left\{\begin{array}{ll}
0    & \qforq ~ 0 \le t < \Sigma^a_1 \qandq \Sigma^a_2 \le t < \Sigma^a_3 \\
\tau & \qforq ~ \Sigma^a_1 \le t < \Sigma^a_2 \qandq \Sigma^a_3 \le t < \Sigma^a_4,
\end{array}\right.
\eeq
where the switching times $\Sigma^a_i$ are defined analogously to \eqref{switchTimes}, and are formally defined in \eqref{switchTimesApprox} below.
Fourth, we let the initial condition for the approximating system be defined by
\bequ \label{initApprox}
x^a(0) = \lim_{\tau \ra 0}x(0), \quad \mbox{so that } z^a_{1,2}(0) = z^a_{2,1}(0) = 0,
\eeq
where $x(0)$ is the initial condition in Assumption \ref{assInit1}.
Fifth, and finally, we primarily
focus on the three-dimensional function $x^a_3 \equiv (\Delta^a, z^a_{1,2}, z^a_{2,1})$
that approximates the three-dimensional function
 $x_3 \equiv (\Delta, z_{1,2}, z_{2,1})$
obtained from \eqref{SwitchODE}, ignoring the queue lengths.
We will be assuming that the queue lengths remain positive, which can be checked at the end.
In general, our analysis is valid until a queue length becomes $0$.
First, we focus on the difference function because it is possible to do so
and still have a bonafide dynamical system, which is easier to analyze.
Second, we are motivated to ignore the queue lengths
 because we have less control over them without abandonment; e.g.,
they can easily explode (diverge to infinity).  However, we will also state some results for the
full six-dimensional approximation $x^a$.

Since the approximating queue lengths $q^a_1$ and $q^a_2$ can obtain any nonnegative value, the full state space $\SS \equiv [0,\lm/\theta]^2 \times [0,1]^4$
of the solutions to \eqref{SwitchODE} is replaced with $\SS^a \equiv [0,\infty)^2 \times [0,1]^4$.
Indeed $\SS^a$ is obtained from $\SS$ directly because $\lm/\theta \ra \infty$ as $\theta \ra 0$.
The state space of $x^a_3$ is a-priori $[0,\infty) \times [0,1]^4$, but we will show below that $\Delta$ is bounded from above.

Paralleling \eqref{T1T2}, the switching and holding times, and the intervals between switching times, are defined via
\bequ \label{T1T2approx}
\bsplit
T^a_1 & \equiv \inf{\{t \ge 0: q^a_2(t) - q^a_1(t) \le \kappa\}} \qandq
T^a_2 \equiv \inf\{t \ge 0 : z^a_{2,1}(\Sigma^a_1 + t) \le \tau\}, \\
T^a_3 & \equiv \inf\{t \ge 0 : q^a_1(\Sigma^a_2+t) - q^a_2(\Sigma^a_2+t) \le \kappa\}  \qandq
T^a_4 \equiv \inf\{t \ge 0 : z^a_{1,2}(\Sigma^a_3 + t) \le \tau\}, \\
\end{split}
\eeq
where, with $T^a_0 \equiv \Sigma^a_0 \equiv 0$,
\bequ \label{switchTimesApprox}
\Sigma^a_k \equiv \sum_{i=0}^k T^a_i \qandq \I^a_i \equiv [\Sigma^a_{i-1}, \Sigma^a_i), \quad k = 1,2,3,4.
\eeq
Paralleling \eqref{Sigmaq}, we let
\bequ \label{Sigmaqa}
\Sigma^a_q \equiv \inf\{t > 0 : q^a_1(t) \wedge q^a_2(t) = 0\}.
\eeq
Our analysis will be valid for the full six-dimensional system on the interval $[0,\Sigma^a_q]$, but
we will not examine $\Sigma^a_q$ until the end.
In particular, we will show that the system quickly converges to the (unique)
periodic equilibrium, when it exists, for any initial condition that is associated with an oscillating solution.
We can therefore initialize the queues (which are unbounded) at large values so that there is no time
for them to reach $0$ by the time convergence to the periodic equilibrium is observed.

In examples we see that the approximating system approximates our original system very well when
the parameters $\theta$ and $\tau$ are suitably small.
For this approximating system, we establish the following result.
Let $\Sigma^{a,(k)}_4$ and $\Delta^{a,(k)}$ be the values of the $k^{\rm th}$ iteration,
where we apply the approximation above in the $k^{\rm th}$ subinterval after making $\Sigma^{a,(k-1)}_4$ equal to time $0$.

\begin{theorem}\label{thApprox}  Consider the approximating system defined above.

\vspace{0.04in}
$($a$)$  The unique stationary point $x^*_0$ in \eqref{statPt} for the fluid model in \S \ref{secOscFluid}
is also the unique stationary point in $\RR^6$ for the approximating system.

\vspace{0.04in}
$($b$)$  If $\Delta^a (0) \le \kappa$ or if $\Delta^{a, (k)} (0) \le \kappa$ for some $k \ge 1$, then $x^a (t) \ra x^*_0$ in $\RR^6$ for $x^*_0$ in \eqref{statPt}.

\vspace{0.04in}
$($c$)$  Whenever $x^a (t) \ra x^*_0$ in $\RR^6$ for $x^*_0$ in \eqref{statPt},
$x^a_3 (t) = (0,0,0)$ for all sufficiently large $t$.

\vspace{0.04in}
$($d$)$  If $\Delta^{a, (k)} (0) > \kappa$ for all $k$, then $\Delta^{a, (k)} (0) \ra \Delta^{a, (\infty)} (0) \in [\kappa + \ep^a_\kappa, (1-\mu)(1-\tau)/\mu]$ as $k \ra \infty$,
where $\ep^a_\kappa \equiv - \log(1-\tau) > 0$.

\vspace{0.04in}
$($e$)$  If the condition in part (d) holds, 
and if $\Sigma^a_q =  \infty$, then $($i$)$
there exists a unique periodic equilibrium $u^{a*}_3$ to the three-dimensional approximating system and $($ii$)$ the approximating system is
bistable:  There are initial conditions for which $x^a (t) \ra x^*_0$ in $\RR^6$ for $x^*_0$ in \eqref{statPt} $($which may include having $\Sigma^a_q < \infty)$;
there are other initial conditions for which $\Sigma^a_q =  \infty$ and
$x^a (t)$ fails to converge in $\RR^6$ in the usual sense of pointwise convergence, but
 $x^a_3 (t) \ra u^{a*}_3$ in $\RR^3$
in the sense of Definition \ref{defStableCycle}; and there are no other possibilities.

\vspace{0.04in}
$($f$)$  For any given pair of control parameters $(\kappa, \tau)$, there exists $\mu^* \equiv \mu^*(\kappa, \tau)$ such that, for any service rate $\mu \in (0,\mu^*)$,
the condition in part (d) holds with $\Delta^{a, (\infty)} (0) > \kappa$, so that
the conclusions of part (e) hold, provided that $\Sigma^a_q = \infty$.
\end{theorem}

The condition $\Sigma^a_q = \infty$ is easy to check directly by solving the simple equations for the full six-dimensional
equation \eqref{SwitchODEa}.
However, in \S \ref{secCC} below we show that, whether or not this condition holds can be determined a posteriori by a simple calculation
that depends only on the periodic equilibrium, and does not depend on the transient behavior of the fluid model. 



In \S \ref{sectx1} and \S \ref{sectx2} we derive the solution to the approximating system
over the first and second intervals, $[0, \Sigma^a_1)$ and $[\Sigma^a_1, \Sigma^a_2)$, respectively.
In \S \ref{sectxAfter} we construct the solution after $\Sigma^a_2$.
In \S \ref{secPfabc}, \S \ref{secPfde} and \S \ref{secPff}, respectively, we prove Theorem \ref{thApprox} (a)-(c), (d)-(e) and (f).
In \S \ref{secApproxT1} we consider a simple heuristic to provide an approximate explicit formula for the switching time $T^a_1$
to facilitate computations.  We conclude in \S \ref{secCC} by showing how to apply the explicit formula in
\S \ref{secApproxT1} to determine if there will be congestion collapse.
We establish conditions for a stronger geometric rate of convergence and exponential stability in
\S \ref{secGeom} in the appendix.

\subsection{The Approximation Over the First Interval $\I^a_1 = [0, \Sigma^a_1)$} \label{sectx1}

The ODE's for $x^a$ over $[0, \Sigma^a_1)$ are just as in \eqref{poolsT1}-\eqref{queuesT1}, but with $\theta = \tau = 0$.
Just as in \S \ref{secInt1}, $q^a_1$ is increasing while $q^a_2 \ge q^a_1 + \kappa$, so $\Sigma^a_q > \Sigma^a_1$.

It follows from \eqref{poolsT1} that, for $x^a(0)$ in \eqref{initApprox},
\bequ \label{poolsTa1}
z_{1,2}(t) = 0 \qandq z_{2,1}(t) = 1-e^{-t}, \quad \mbox{so that } z_{1,1}(t) = e^{-t} \qandq z_{2,2}(t) = 1, \quad 0 \le t < \Sigma^a_1.
\eeq

The value of $T^a_1$ is determined by the process $\Delta^a \equiv q^a_2 - q^a_1$, approximating the corresponding difference process $\Delta$.
Taking $\theta = \tau = 0$ and $z_{2,1}(0) = 0$ in \eqref{queuesT1a}-\eqref{queuesT1b}, we have
\bequ \label{DelAppT1}
\dot{\Delta}^a(t) = -(1+\mu) - (1-\mu)e^{-t} \quad \mbox{so that ~} \Delta^a(t) = \Delta^a(0) -(1+\mu)t + (1-\mu)(1-e^{-t}), \quad 0 \le t < \Sigma^a_1.
\eeq
Since $\Delta^a(T^a_1) = \kappa$ by definition, it follows that
\bequ \label{T1approx}
T^a_1 = \frac{\Delta^a(0) - 1 + \mu - \kappa}{1+\mu} + \frac{1-\mu}{1+\mu} e^{-T^a_1}.
\eeq

\begin{lemma} \label{lmTa1}
For any fixed $\Delta^a(0) > \kappa$ there exists a unique $T^a_1 > 0$ satisfying \eqref{T1approx}.
Furthermore, $T^a_1$ is strictly increasing in $\Delta^a(0)$.
\end{lemma}

\begin{proof}
Define the function $F: B \ra \RR_+$, where
\bequ \label{impFnc}
B \equiv (\kappa, \infty) \times (0,\infty) \qandq F(\Delta, T) \equiv \frac{\Delta - 1 + \mu - \kappa}{1+\mu} + \frac{1-\mu}{1+\mu} e^{-T} - T,
\eeq
and the function
\bes
h(T) \equiv \Delta - 1 + \mu - \kappa + (1-\mu)e^{-T} - (1+\mu) T.
\ees
Note that $h(0) > 0$ and $h(T) \ra - \infty$ as $T \ra +\infty$.
Furthermore, $h'(T) < 0$, so that $h(T)$ is strictly decreasing.

It follows that for any fixed $\Delta > \kappa$, there exists a unique $T > 0$, such that
$(\Delta, T) \in B$ and $F(\Delta, T) = 0$.
In addition, it clearly holds that $\frac{\partial F}{\partial \Delta}$ and $\frac{\partial F}{\partial T}$ exist in $B$ and are continuous,
and that $\frac{\partial F}{\partial T} \ne 0$ for all real $T$.
Then by the implicit-function theorem there exists a unique continuously-differentiable function $T(\Delta)$,
such that $F(\Delta, T(\Delta)) = 0$ over the domain $B$, and
\bes
\frac{dT}{d \Delta} = -\frac{\frac{\partial} {\partial \Delta}F}{\frac{\partial} {\partial T}F} = \frac{1}{(1-\mu)e^{-T} + (1+\mu)} > 0,
\ees
so that $T$ is strictly increasing in $\Delta$.

In passing we note that the point $(\Delta_0, T_0) \equiv (1-\mu+\kappa, 0)$ satisfies $F(\Delta_0, T_0) = 0$. However, this point is not in $B$,
so there is no contradiction to the claim that there exists a function $T(\Delta)$ as in the proof of Lemma \ref{lmTa1}.
\end{proof}

It follows from \eqref{poolsTa1} that for $\Sigma^a_1 \equiv T^a_1$,
\bequ \label{txT1}
x^a_3(\Sigma_1) = (\kappa, 0, 1-e^{-T^a_1}),
\eeq
which is well-defined by Lemma \ref{lmTa1}.

\subsection{The Approximation Over the Second Interval $\I^a_2 = [\Sigma^a_1, \Sigma^a_2)$} \label{sectx2}

The equations for the service process over $[\Sigma^a_1, \Sigma^a_2)$ are obtained from \eqref{poolsT2}, but with
$T^a_1$ replacing $T_1$ and $z^a_{i,j}(T^a_1)$ replacing $z_{i,j}(T_1)$, $i,j = 1,2$.
As in \S \ref{secInt2}, it is possible to have $\Sigma^a_1 < \Sigma^a_q \le \Sigma^a_2$, but we do not
check that now.

Since the process $z_{1,2}$ in \eqref{poolsT2} keeps decreasing and $z^a_{1,2}(T^a_1) = 0$, it follows from \eqref{txT1} and \eqref{poolsT2} that
\bequ \label{poolsTa2}
z^a_{1,2}(T^a_1 + t) = 0 \qandq z^a_{2,1}(T^a_1 + t) = (1-e^{-T^a_1})e^{-\mu t}, \quad 0 \le t < T^a_2.
\eeq
Taking $\theta \da 0$ and inserting the values of $z^a_{1,2}(T^a_1)$ and $z^a_{2,1}(T^a_1)$ from \eqref{txT1} in \eqref{queuesT2a}, we see that
\bequ \label{DelAppT2}
\dot{\Delta^a}(\Sigma^a_1 + t) = -z^a_{2,1}(T^a_1)(1-\mu) e^{-\mu t} = -(1-e^{-T^a_1})(1-\mu) e^{-\mu t}, \quad 0 \le t < T^a_2,
\eeq
where $\Delta^a(\Sigma^a_1) = \kappa$.

By \eqref{T1T2approx}, $T^a_2$ is the first time after $\Sigma^a_1$ that $z^a_{2,1}$ hits $\tau$, so that, paralleling \eqref{T2formula},
\bequ \label{Ta2}
T^a_2 = \frac{\log(z^a_{2,1}(T^a_1)/\tau)}{\mu} = \frac{\log((1-e^{-T^a_1})/\tau)}{\mu}.
\eeq
Clearly, if $\tau \da 0$ then $T^a_2 \ra \infty$, which is why we cannot replace $\tau$ with $0$ over the second interval $[\Sigma^a_1, \Sigma^a_2)$.

Inserting the value of $T^a_2$ into the solution to \eqref{DelAppT2} we obtain
\bes
\Delta^a(\Sigma^a_2-) = \kappa -\frac{z_{2,1}(T^a_1)(1-\mu)}{\mu}\left(1-\frac{\tau}{z^a_{2,1}(T^a_1)} \right) =
\kappa-\frac{1-\mu}{\mu}(1- e^{-T^a_1} - \tau),
\ees
where $y(t-) \equiv \lim_{s\ua t} y(s)$ denotes the left limit at time $t$ of a function $y$.
Hence,
\bequ \label{txT2}
x^a_3(\Sigma^a_2-) = \left(\kappa-\frac{1-\mu}{\mu}(1- e^{-T^a_1} - \tau), 0, \tau \right).
\eeq

\subsection{Continuing Beyond $\Sigma^a_2$} \label{sectxAfter}

As before, we can use the symmetry of $x^a_3$ and take $x^a_3(\Sigma^a_2)$ to be the ``initial condition''
by reversing the labels.
This means that, as in \eqref{initApprox}, we take $\tau \da 0$ in $x^a_3(\Sigma^a_2)$.
It follows immediately from \eqref{txT2} that $\lim_{\tau \da 0} x^a_3(\Sigma^a_2) \ne x^a_3(\Sigma^a_2-)$.
Hence, the approximation $x^a_3$, and therefore $x^a$, has a jump at time $\Sigma^a_2$, since the values of $\Delta^a(\Sigma^a_2-)$
and $z_{2,1}(\Sigma^a_2-)$ both depend on $\tau$. However, we can easily avoid having jumps in the process $\Delta^a$, which we want to avoid because it causes ambiguities
about the behavior of the queues at the jump times.
To that end, we simply define
$$\Delta^a(\Sigma^a_2) \equiv \Delta^a(\Sigma^a_2-) = \kappa-\frac{1-\mu}{\mu}(1- e^{-T^a_1} - \tau) \qandq z_{2,1}(\Sigma^a_2) = \lim_{\tau \da 0} z_{2,1}(\Sigma^a_2) = 0,$$
so that
we have
\bequ \label{txt22}
x^a_3(\Sigma^a_2) = \left(\kappa-\frac{1-\mu}{\mu}(1- e^{-T^a_1} - \tau), 0, 0 \right).
\eeq
As a consequence, only $z_{2,1}$ jumps at the second switching time $\Sigma^a_2$.
That discontinuity makes our fluid model a switching dynamical system with jumps, as mentioned at the beginning of the section.

If
$\Delta^a(\Sigma^a_2) > \kappa,$
then $T^a_3 > 0$, and paralleling
\eqref{T1approx} and Lemma \ref{lmTa1}, $T^a_3$ is the unique strictly positive solution to
\bes
T^a_3 = \frac{\Delta^a(\Sigma^a_2) - 1 + \mu - \kappa}{1+\mu} + \frac{1-\mu}{1+\mu} e^{-T^a_3}.
\ees
Furthermore, paralleling \eqref{Ta2},
$$T^a_4 = \frac{\log((1-e^{-T^a_3})/\tau)}{\mu},$$
so that
\bes
\Delta^a(\Sigma^a_4-) = \frac{1-\mu}{\mu}(1- e^{-T^a_3} - \tau) - \kappa, \quad z^a_{1,2}(\Sigma^a_4-) = \tau \qandq z^a_{2,1}(\Sigma^a_4-) = 0.
\ees
If $\Delta^a(\Sigma^a_4-) > \kappa$ we define $\Delta^a(\Sigma^a_4) \equiv \Delta^a(\Sigma^a_4-)$ and
$z^a_{1,2}(\Sigma^a_4) = \lim_{\tau \da 0} z^a_{1,2}(\Sigma^a_4-) = 0$ and start over.

The preceding shows that, just as for the original system, we can exploit the symmetry of the model and consider only the half cycle $[0, \Sigma^a_2)$.
In particular, for a given initial condition $\Delta^a(0)$ we solve up to time $\Sigma^a_2$ and take
\bequ \label{StartOver}
-x^a_3(\Sigma^a_2) = \left(\frac{1-\mu}{\mu}(1- e^{-T^a_1} - \tau) - \kappa, 0, 0 \right)
\eeq
to be a new initial condition to solve beyond time $\Sigma^a_2$.
It immediately follows that
\begin{lemma} \label{corDelAppBd}
$\Delta^a$ is bounded over $[0, \Sigma^a_q)$. In particular, if $\Sigma^a_4 < \Sigma^a_q$, then
$\Delta^a(\Sigma^a_4) < \Delta^a_{bd} \equiv \frac{1-\mu}{\mu}(1-\tau)$.
\end{lemma}

It is significant that at the switching times, $x^a_3$ depends only on the known control parameters $(\kappa, \tau)$ and the one unknown $T^a_1$.
Therefore, {\em the approximating system is reduced to an essentially one-dimensional system at the switching times}.

\paragraph{The Approximating Three-Dimensional System.} 
From the above, $x^a_3 = (\Delta^a, z_{1,2}, z_{2,1})$ is the unique solution over $[0,\Sigma^a_q)$, for $\Sigma^a_q$ in \eqref{Sigmaqa}, to
\bequ \label{ODEapprox}
\dot{x}^a_3 = f^3_{\sigma(x^a_3)}(x^a_3,\theta, \tau^a) = f^3_{\sigma(x^a_3)}(x^a_3, 0, \tau^a), \quad \sigma(x^a_3) = 1,2,3,4,
\eeq
with initial condition \eqref{initApprox} and $\tau^a$ in \eqref{tau^a},
where $f^3_1$ is defined in \eqref{poolsT1} and \eqref{DelAppT1}, $f^3_2$
is defined in \eqref{poolsT2} and \eqref{DelAppT2}, $f^3_3$ satisfies the equations of $f^3_1$, but with the labels reversed, and $f^3_4$
satisfies the equations of $f^3_2$, with the labels of the processes reversed.

\subsection{Proof of Theorem \ref{thApprox} (a)-(c)} \label{secPfabc}

Recall that the ODE \eqref{ODEapprox} is solved until time $\Sigma^a_4$,
and can then be continued beyond that time by taking $x^a_3(0) \equiv x^a_3(\Sigma^a_4)$ to be a new initial condition
provided that $x^a_3(\Sigma^a_4)$ satisfies \eqref{initApprox}, i.e. if $\Delta^a(\Sigma^a_4) > \kappa$.
However, if $\Delta^a(\Sigma^a_4) \le \kappa$,
then the ODE does not follows the switching pattern in \eqref{ODEapprox}.
The next lemma shows that, in this case, the solution will converge to $x^*_0$ and will therefore cease to oscillate.

\begin{lemma} \label{lmStatptApp}
If $\Delta^a(0) \le \kappa$, but all other conditions in \eqref{initApprox} hold, then $x^a(t) \ra x^*_0$ for $x^*_0$ in \eqref{statPt}.
\end{lemma}
Note that the lemma considers the full six-dimensional approximation $x^a$, and not only the three-dimensional restriction $x^a_3$.

\begin{proof}
The initial condition has $z^a_{1,2}(0) = z^a_{2,1}(0) = 0$, so that $z^a_{1,1}(0) = z^a_{2,2}(0) = 1$.
Hence, both pools serve only their own fluid queues, as long as $q_i(t)-q_j(t) < \kappa$, for both $(i,j)=(1,2)$ and $(i,j) = (2,1)$.
Therefore (see \eqref{qODE})
\bes
\dot{q}_1(t) = \dot{q}_2(t) = \lm - 1 < 0, \quad 0 \le t < \Sigma^a_q, 
\ees
so that $\dot{\Delta}^a(t) = 0$ on $[0, \Sigma^a_q)$, and no sharing can begin during that interval.
At time $\Sigma^a_q$ at least one of the queues hits $0$, say $q^a_i$.
If the other queue is still positive at that time, then it continues to decrease at the same constant rate as before.
Since  $|q^a_i(\Sigma^a_q) - q^a_j(\Sigma^a_q)| = q^a_j(\Sigma^a_q) < \kappa$, $j \ne i$,
the difference between the two queues can never become larger than $\kappa$, so that the positive queue must also hit $0$ at a finite time after $\Sigma_q$.
Therefore, letting $t_{j}$ denote the time at which queue $j$ hits $0$, $i = 1,2$, we have
\bes
q_i(t) = 0 \qandq \dot{z}_{i,i}(t) = \lm - z_{i,i}(t), \qforallq t > t_j \ge \Sigma^a_q. \quad \mbox{Furthermore, } t_j < \infty.
\ees
It follows that $z_{i,i}(t) \ra \lm$ as $t\tinf$, so that $x^a(t) \ra x^*_0$ as stated.
\end{proof}
It follows from \eqref{StartOver} and Lemma \ref{lmStatptApp} that, if at the end of cycle we have $-\Delta^a(\Sigma^a_2) \le \kappa$,
then $\Sigma^a_q < \infty$ and $x^a(t) \ra x^*_0$ as $t\tinf$. In addition, $\Delta^a(t)$ was just shown to reach $0$ in finite time, and $z^a_{1,2}$
and $z^a_{2,1}$ each reach $0$ in finite time by construction. Therefore, $x^a_3(t)$ reaches $(0,0,0)$ in finite time.
Using similar arguments to those in Theorem \ref{thUniqueStatPt}, we can prove that

\begin{lemma} \label{thStatPtApprox}
$x^*_0$ in \eqref{statPt} is the unique stationary point of the approximating system. Furthermore,
if $x^a_3$ does not oscillate indefinitely, then $x^a_3(t) = (0,0,0)$ for all large enough $t$, so that $x^a(t) \ra x^*_0$ as $t\tinf$.
\end{lemma}
Lemmas \ref{lmStatptApp} and \ref{thStatPtApprox} together complete the proof of Theorem \ref{thApprox} (a)-(c).

\subsection{Proof of Theorem \ref{thApprox} (d) and (e)} \label{secPfde}

To study possible oscillatory behavior of the approximating system in \eqref{ODEapprox} we use an iterative algorithm,
similar to the one in \S \ref{secAlg}, based on the arguments in \S \ref{sectxAfter}.

\paragraph{An Iterative Algorithm for the Approximating System.}
In the iterative algorithm each (half) cycle of $x^a$ corresponds to an iteration.
We use a superscript $(k)$ denote the $k^{th}$ iteration of the algorithm, and drop the superscript ``$a$'' for ease of notation,
e.g., $T_1^{(1)}$ is the value of $T^a_1$ in \eqref{T1approx} in the first cycle of $x^a$, or equivalently, the first iteration of the algorithm.

We start by choosing a value $\Delta^{(0)} \equiv \Delta(0) > \kappa$
and use it to numerically compute $T_1^{(1)}$ via \eqref{T1approx}. The obtained value of $T^a_1$ is then used to compute
$\Delta^{(1)} \equiv \Delta^a(\Sigma^a_4) = - \Delta^a(\Sigma^a_2)$ via \eqref{txt22}.
We continue iterating this way until one of two things occur:
either we see $\Delta^{(k)} > \kappa$ for all $k$
or else we observe $\Delta^{(k)} \le \kappa$ for some $k \ge 1$,
{\em in which case the algorithm is stopped}.

Similar to Lemma \ref{lmDelbdd} and Corollary \ref{corDelBdd} we can show that there exists $\ep^a_\kappa > 0$
such that, if the algorithm can be iterated indefinitely, then $\Delta^{(k)} > \kappa + \ep^a_\kappa$ for all $k \ge 1$.
Of course, for the approximating system we can characterize $\ep^a_\kappa$ explicitly, and its value can serve as an approximation for the value
of $\ep_\kappa$ in Corollary \ref{corDelBdd}.
\begin{lemma} \label{lmNecEndless}
A necessary condition for endless oscillation is that, for all $k \ge 1$, $\Delta^{(k)} > \kappa + \ep^a_\kappa$, where $\ep^a_\kappa \equiv -\log(1-\tau)$.
In particular, if $\kappa < \Delta^{(k)} < \kappa - \log(1-\tau)$ for some $k \ge 1$, then $\Delta^{(k+1)} < 0$, so that the algorithm is stopped.
\end{lemma}

\begin{proof}
For $\ep^a_\kappa$ in the statement of the lemma,
assume that $\kappa < \Delta^{(k)} \le \kappa + \ep^a_\kappa$, for some $k \ge 1$. Then by \eqref{T1approx}
\bes
\bsplit
T^{(k+1)}_1 & \le \frac{\kappa + \ep^a_\kappa - 1 + \mu - \kappa}{1+\mu} + \frac{1-\mu}{1+\mu} e^{-T^{(k+1)}_1}
< \frac{\ep^a_\kappa - 1 + \mu}{1+\mu} + \frac{1-\mu}{1+\mu}
< \frac{\ep^a_\kappa}{1+\mu}.
\end{split}
\ees
Therefore, $T^{(k+1)}_1 < \ep^a_\kappa \equiv -\log(1-\tau)$. It follows from \eqref{StartOver} that $\Delta^{(k+1)} < 0$.
\end{proof}

As was mentioned above, the approximating fluid model is a switching dynamical system with jumps. In this new setting,
the approximating fluid solutions are elements in the space $\D \equiv \D[0,\infty)$ of real-valued right-continuous functions with limits everywhere,
which we endow with the Skorohod $J_1$ topology, which we denote by $d_t$. Specifically, we consider the topological space $(\D, J_1)$, as in \S 3.3 of \cite{W02}.
We have $x_k \ra x$ in $(D, J_1)$ as $k \ra \infty$ if, for each $t$ that is a continuity point of $x$,
\bes
d_t (x_k, x) \equiv || x_k (\lambda_k (\cdot)) - x ||_t \vee ||\lambda_k - e||_t \ra 0 \qasq n \ra \infty,
\ees
where $e: [0,t] \ra [0,t]$ is the identity function $e(s) \equiv s$, $0 \le s \le t$, $\lambda_k$ is a homeomorphism of $[0,t]$
and $|| \cdot ||_t$ is the uniform norm applied to functions on the finite interval $[0,t]$.
Note that convergence in $J_1$ reduces to uniform convergence over bounded intervals whenever the limit function is continuous,
as is the case for all the solutions of \eqref{SwitchODE}.

We generalize Definition \ref{defStableCycle} by replacing the uniform metric in \eqref{ConvDef} with the Skorohod metric.
We then say that a solution $x^a$ spirals towards $u^a_*$ if \eqref{ConvDef} holds for $x^a$ and $u^a_*$, {\em but with the Skorohod $J_1$ metric replacing the uniform metric}.
In our application we will let $\lambda_k (\Sigma^{(k)}_0) = \Sigma^{*(k)}_0$.
After making that small perturbation of the switching times, so that they are aligned, we have uniform convergence over $[0, t]$.

The next lemma shows that spiraling of a solution $x^a$ to $u^a_*$ follows from
the first limit in \eqref{ConvDef} and convergence of $x^a$ to $u^a_*$ at the four switching times.
Its elementary proof is omitted.
\begin{lemma} \label{lmAsyCycle}
Suppose that a periodic equilibrium $u^a_*$, having period $T$, exists for \eqref{ODEapprox}.
If
$$(I)~ \lim_{k\tinf} T^{(k)}_i = T^*_i \qandq (II)~ \lim_{k\tinf} \|x(\Sigma^{(k)}_i) - u(\Sigma^{*(k)}_i)\| = 0, \quad 1 \le i \le 4,$$
for some solution $x^a \ne u^a_*$, then $x^a$ spirals towards $u^a_*$. In particular,
\bes
\lim_{k\tinf} d_t(x(\Sigma^{(k)}_0 + \cdot), u^a_*(\Sigma^{(k)}_* + \cdot)) = 0, \quad \mbox{for each continuity point $t$ of $x(\Sigma^{(k)}_0 + \cdot)$.}
\ees
\end{lemma}

We are now prepared to prove Theorem \ref{thApprox} (d) and (e).
\begin{proof}[Proof of Theorem \ref{thApprox} (d) and (e)]
Lemma \ref{lmNecEndless} implies that a solution to the approximating system that oscillated indefinitely is bounded away from $\kappa$.
Together with Lemma \ref{corDelAppBd}, this implies that $\Delta^{(k)}$ is confined to the compact interval $I_\Delta \equiv [\kappa + \ep^a_\kappa, (1-\mu)(1-\tau)/\mu]$.
Moreover, $\Delta^{(k)}$ is strictly monotone in $T^{(k)}_1$ by \eqref{StartOver}, which is itself strictly monotone in $\Delta^{(k-1)}$
by Lemma \ref{lmTa1}, $k \ge 1$. Hence, the sequence $\{\Delta^{(k)} : k \ge 0\}$ is monotone and bounded, and therefore converges to a limit $\Delta^{a, (\infty)} \in I_\Delta$.
Since $x^*_0$ is the unique stationary point of the approximating system and $\Delta^{a, (\infty)} > \kappa$ cannot be part of a stationary solution,
the limit $\Delta^{a, (\infty)}$ must be a point on a periodic equilibrium, which is clearly unique. This proved (d).
Part (e) of the theorem follows from Lemma \ref{lmAsyCycle}, together with Lemma \ref{lmNecEndless} and parts (a)-(c) of the theorem.
\end{proof}

\subsection{Proof of Theorem \ref{thApprox} (f)} \label{secPff}

It remains to show that the conditions of part (e) can be satisfied, i.e., there exist parameters for which $\Delta^{(k)} > \kappa$
for all $k \ge 0$ and $\Delta^{(k)} \ra \Delta^{(\infty)} > \kappa$.
To prove this, consider $\Delta^{(k-1)} > 1-\mu+\kappa$ and observe that, since $(1-\mu)/(1+\mu) < 1$, \eqref{T1approx} implies that
\bequ \label{Ta1bd}
0 < \frac{\Delta^{(k-1)} - 1 + \mu - \kappa}{1+\mu} < T^{(k)}_1 
< \frac{\Delta^{(k-1)} - 1 + \mu - \kappa}{1+\mu} + 1, \quad k \ge 1.
\eeq
By Lemma \ref{corDelAppBd}, $\Delta^{(k-1)}$ is bounded from above by $\Delta^a_{bd}\equiv (1-\mu)(1-\tau)/\mu$.
Therefore, consider $\Delta^{(0)} \in [\Delta_\mu^m, \Delta_\mu^M]$,
where
\bequ \label{DelaBd}
\Delta_\mu^m \equiv 1 - \mu + \kappa \qandq \Delta_\mu^M \equiv \Delta^a_{bd}\equiv (1-\mu)(1-\tau)/\mu. 
\eeq
Note that $\Delta_\mu^m > \kappa + \ep^a_\kappa$ for $\ep^a_\kappa$ in Lemma \ref{lmNecEndless} if $\tau$ is small, as we assume, and
$1- \mu > \ep^a_\kappa$, which we require.
The requirement that $\Delta_\mu^m < \Delta_\mu^M$, gives rise to quadratic equation in $\mu$ whose roots are
\bequ \label{roots}
\mu_{1} = \frac{2+\kappa-\tau - \sqrt{(\kappa-\tau)^2 + 4 \kappa}}{2} \qandq \mu_{2} = \frac{2+\kappa-\tau + \sqrt{(\kappa-\tau)^2 + 4 \kappa}}{2},
\eeq
which are easily seen to satisfy $0 < \mu_1 < 1 < \mu_2$.
Therefore, we henceforth consider $\mu \in (0, \mu_1)$ such that $1-\mu > \ep^a_\kappa \equiv -\log(1-\tau)$, so that $\mu < 1+\log(1-\tau)$.

Next, we introduce a mapping taking $\Delta (0) = \Delta$ into a function of $T^a_1$,
where $T^a_1 \equiv T^a_1(\Delta)$ is the unique positive solution to \eqref{T1approx};
specifically, let
\bequ \label{Tmap}
\sT : \Delta \mapsto -\kappa - \frac{1-\mu}{\mu} e^{-T^a_1} + \frac{1-\mu}{\mu}(1-\tau),
\eeq
so that $\sT(\Delta^{(k-1)}) = \Delta^{(k)}$, $k \ge 1$.

For fixed $\mu \in (0,\mu_1)$ and $0 < \delta_\mu < \Delta_\mu^M - \Delta_\mu^m$ to be specified below, let
\bequ \label{SSmu}
\SS_\mu \equiv [\Delta_\mu^M - \delta_\mu, \Delta_\mu^M]. 
\eeq
Note that the end points of $\SS_\mu$ depend on $\mu$, and that $\bigcup_{\mu}\SS_\mu = [1+\kappa, \infty)$,
where the union is taken over all the values of $\mu \in (0,\mu_1)$, for $\mu_1$ in \eqref{roots}.
In particular, the left end point of $\SS_\mu$ is bounded from below whereas its right end point is unbounded as $\mu \da 0$.
Nevertheless, $\SS_\mu$ is compact for any fixed $\mu \in (0,\mu_1)$.
\begin{lemma}{$($sufficient condition for endless iterations$)$} \label{thEndlessAprox}
For a given pair of control parameters $(\kappa, \tau)$ and $\mu_1$ in \eqref{roots},
there exists $\mu_* \in (0, \mu_1)$ such that $\sT: \SS_\mu \ra \SS_\mu$\, for all $\mu \le \mu_*$.
\end{lemma}

\begin{proof}
Observe that by \eqref{Ta1bd} and \eqref{Tmap}
\bequ \label{TmapBd}
\bsplit
\sT(\Delta) & = - \kappa - \frac{1-\mu}{\mu} e^{-T^a_1} + \frac{(1-\mu)(1-\tau)}{\mu} \\
& > - \kappa + \frac{1-\mu}{\mu} (1 - \tau - e^{-\frac{\Delta - 1 + \mu - \kappa}{1+\mu}}), 
\end{split}
\eeq
so that $\T(\Delta) > \Delta_\mu^m$ if and only if
\bequ \label{mapstoS}
\xi(\Delta) \equiv e^{-\frac{\Delta - 1 + \mu - \kappa}{1+\mu}} < 1 - \tau + \mu(1 - 2\kappa/(1-\mu)).
\eeq
Note that $\xi(\Delta)$ decreases to $0$ as $\Delta$ increases to $\infty$ and that the right-hand
side of \eqref{mapstoS} is bounded from below by $1-\tau$ as $\mu$ decreases to $0$.
Since $\Delta_\mu^m \ra 1+\kappa$ and $\Delta_\mu^M \ra \infty$ as $\mu \da 0$, we can find $\mu_*$ small enough and $\Delta$
large enough such that, for all $\mu \le \mu_*$, $\Delta_\mu^m < \Delta < \Delta_\mu^M$ and \eqref{mapstoS} holds for that $\Delta$.

Choose $c > 0$ such that $1 - \tau - c > 0$ and fix $0 < \ep < c$.
Take $\mu_*$ smaller if needed, so that for any $\mu \in (0,\mu_*)$, it holds that $\xi(\Delta) < \ep$ whenever $\Delta > \frac{1-\mu}{\mu}(1-\tau-c) - \kappa$.
Then by \eqref{TmapBd}
\bes
\sT(\Delta) > \frac{1-\mu}{\mu}(1- \tau - \ep) - \kappa > \frac{1-\mu}{\mu}(1- \tau - c) - \kappa.
\ees
The statement of the theorem follows by taking
\bequ \label{delta_mu}
\delta_\mu \equiv \frac{1-\mu_*}{\mu_*}c + \kappa,
\eeq
where we take $\mu_*$ sufficiently small to have $\Delta_\mu^M - \delta_\mu > \Delta_\mu^m$, i.e., $\frac{1-\mu}{\mu}(1- \tau - c) - \kappa > 1-\mu+\kappa$,
which clearly holds for all sufficiently small $\mu$ for any fixed $c < 1-\tau$.
\end{proof}

Lemma \ref{thEndlessAprox} and its proof can be used to show that, for a range of values of $\mu$,
the iterative algorithm in \S \ref{secPfde} converges geometrically fast to the point
$\Delta^a_*$ on the periodic equilibrium, when $u^a_* \in \SS_\mu$; see \S \ref{secGeom}.
We also prove a stronger result, stating that the rate of convergence to the periodic equilibrium (in continuous time) is exponential.
Rapid convergence to the equilibrium is seen in the numerical experiments in \S \ref{secNumeric}.
Finally, by Lemma \ref{lmAsyCycle}, the three-dimensional solution $x^a_3$ to \eqref{ODEapprox} ``spirals'' toward $u^a_*$.

\subsection{A Simple Heuristic Approximation for Computation} \label{secApproxT1}

The approximating system we have developed in this section has been useful to estalbish the
strong theoretical results in Theoreem \ref{thApprox}, which supports what we see for the original system in numerical examples.
However, it is still not easy to compute the periodic equilibrium of the approximating system.
We must either numerically solve the ODE's or numerically solve for $T^a_1$ in \eqref{T1approx} in order to evaluate the values of $x^a$ at the switching times.

Hence, in the present section we develop a simple heuristic approximation for $T^a_1$ in \eqref{T1approx}.
In particular, our approximation is obtained by simply omitting the second exponential term on the right in \eqref{T1approx},
which produces the approximation
\bequ \label{T1approx2}
T^a_1 \approx \frac{\Delta - 1 + \mu - \kappa}{1+\mu}.
\eeq

Approximation \eqn{T1approx2} can be justified by observing that equation \eqref{T1approx}
can be expressed abstractly as  $T^a_1 = A + Be^{-T^a_1}$ for $A > 0$ and $0 < B < 1$.
Since $T^a_1 > A$ and $T^a_1 - A < Be{-A}$, $T^a_1 \approx A$ whenever $B$ is suitably small or $A$ is suitably large.
In particular, the error is asymptotically negligible as $A$ increases.
We remark that approximation \eqn{T1approx2} also coincides with $-\log{(\xi)}$
$\xi \equiv \xi(\Delta)$ in \eqref{mapstoS}, which can provide another way to derive the approximation.
We can combine \eqref{txT1} and \eqn{T1approx2} ot obtain an associated approximation for
$z^a_{2,1}(T^a_1)$.

With this heuristic approximation for $z^a_{2,1}(T^a_1)$, we have by \eqref{Ta2} that
\bequ \label{T2approx}
T^a_2 \approx \frac{\log \left((1 - \xi)/\tau \right)}{\mu},
\eeq
so that \eqref{txT2} and \eqref{txt22} are respectively approximated by
\bequ \label{SolApproxT2}
x^a(\Sigma^a_2-) \approx \left(\kappa - \frac{1-\mu}{\mu} \left(1 - \xi - \tau \right),0,\tau \right)
\qandq
x^a(\Sigma^a_2) \approx \left(-\kappa + \frac{1-\mu}{\mu} \left(1 - \xi - \tau \right),0, 0 \right),
\eeq
and $x^a(\Sigma^a_2)$ serves as the initial condition for the following cycle.

We can use this heuristic approximation to easily approximate whether a periodic equilibrium exists, and to approximate its values
at the switching times, using the iterative algorithm described in \S \ref{secPfde}.
We start by choosing a value $\Delta(0)$ such that $\xi^{(1)} \equiv \xi$ in \eqref{mapstoS} is sufficiently small (e.g., $\xi^{(1)} < 0.05$) and $T_1^{(1)}$
in \eqref{T1approx2} is strictly positive.
Given $\xi^{(1)}$, we compute $\Delta^{(1)}(\Sigma_2^{(1)})$ in \eqref{SolApproxT2}, and take $\Delta^{(1)}(0) = -\Delta^{(1)}(\Sigma_2^{(1)})$
in order to compute $\xi^{(2)}$ via \eqref{mapstoS}.
As before, we continue iterating until we see convergence to a legitimate value, i.e., $\Delta^{(k)}$ converges to some $\Delta^a_* > \kappa$
and $\xi^{(k)}$ converges to a value $\xi_* < 1$, or we obtain an illegitimate value at some iteration, i.e.,  $\Delta^{(k)} < \kappa$ or $\xi^{(k)} > 1$
for some $k \ge 1$. In the latter case, the algorithm is stopped.
The latter case indicates that the solution $x^a$ converges to $x^*_0$.
If the initial condition for the algorithm is extreme, i.e., $\Delta^{(0)}$ is taken to be very large,
then stopping the algorithm suggests that a periodic equilibrium does not exist.


\subsection{Checking For Congestion Collapse}\label{secCC}

When there is no abandonment,
we cannot expect that the queues in an oscillating system will remain finite as time increases. Indeed, if
\bequ \label{condOL}
\lim_{t\tinf} \frac{1}{t}\int_0^t (z_{i,i}(s) + \mu z_{i,j}(s)) ds < \lm, \quad i,j = 1,2,
\eeq
then the queues are not {\em rate stable}, i.e., the long-run average input rate $\lm$ is larger than the long-run average throughput rate,
so that the queues will increase without bound. We now show how to estimate whether \eqref{condOL} holds.

In particular, we now show that the simplified heuristic approximation in \S \ref{secApproxT1} facilitates verification of \eqref{condOL}
for a system that is known to converge to the unique periodic equilibrium.
Let $\Sigma^*_i$ and $T^*_i$ denote the switching and holding times of the periodic equilibrium, $1 \le i \le 4$.
Without loss of generality, consider pool $1$. (Due to the symmetry, it is sufficient to check whether \eqref{condOL} holds for one of the pools.)
Then, for
\bes
\zeta(s) \equiv z^a_{1,1}(s) + \mu z^a_{2,1}(s) = 1-(1-\mu)z^a_{2,1}(s),
\ees
\eqref{condOL} becomes
\bes
L \equiv \lim_{t\tinf} \frac{1}{t}\int_0^t \zeta(s) ds = \frac{1}{\Sigma^*_4} \int_0^{\Sigma^*_4} \zeta(s) ds =
\frac{1}{\Sigma^*_4} \left[\int_0^{T^*_1} \zeta(s)ds + \zeta(T^*_1)\int_0^{T^*_2} \zeta(s)ds + (\Sigma^*_4 - \Sigma^*_2) \right],
\ees
where the first equality follows from the asymptotic periodicity of the solution, and the second equality follows from the symmetry of the model.
Recall also that $z^a_{2,1} \equiv 0$, so that $z^a_{1,1} = 1$ over $[\Sigma^*_2, \Sigma^*_4]$, which gives the last term in the square brackets.
We can use the last value of $\xi^{(k)}$ obtained from the algorithm above to serve as our approximation for $\xi^* \equiv \xi(\Delta^a_*)$,
for $\xi(\cdot)$ in \eqref{mapstoS}, together with \eqref{poolsTa1} and \eqref{poolsTa2} to approximate $L$.

Using the fact that $\Sigma^*_4 = 2 \Sigma^*_2$, we have (since $\Sigma^*_4 - \Sigma^*_2 = \Sigma^*_2$)
\bequ \label{L}
\bsplit
L & = 1 - \frac{1-\mu}{2\Sigma^*_2} \left[\int_0^{\Sigma^*_2}z^a_{2,1}(s) ds + \Sigma^*_2\right] \\
& \approx \frac{1+\mu}{2} - \frac{1-\mu}{2[-\log{(\xi^*)} + \log{((1-\xi^*)/\tau)}/\mu]}
\left[\int_0^{-\log{(\xi^*)}} (1-e^{-s}) ds + (1-\xi^*)\int_0^{\frac{\log{\left(\frac{1-\xi^*}{\tau}\right)}}{\mu}} e^{-\mu s} ds \right] \\
& = \frac{1+\mu}{2} - \frac{(1-\mu)[-\log{(\xi^*)} + \xi^* - 1 + (1+\xi^*-\tau)/\mu]}{2[-\log{(\xi^*)} + \log{((1-\xi^*)/\tau)}/\mu]} ,
\end{split}
\eeq
with the approximation following by, first noting that
$\Sigma^*_2 = T^*_1 + T^*_2$ and, second, replacing $T^*_1$ and $T^*_2$ with \eqref{T1approx2} and \eqref{T2approx}, respectively.

Note that, unlike the original system \eqref{SwitchODE},
in the approximating system we can first compute the periodic equilibrium, when it exists, via the iterative algorithm,
and then check whether the system goes through congestion collapse. The heuristic approximation given here facilitates this inspection, via the computation in \eqref{L}.
More specifically, if a periodic equilibrium of \eqref{ODEapprox} is found,
and if this periodic equilibrium is associate with congestion collapse, then the queues necessarily increase to infinity as time increases,
provided that $x^a_3$ converges to $u^a_*$ before either queue hits $0$. 
We can then make sure that $\Sigma^a_q = \infty$ simply by 
initializing the two queues of the six-dimensional vector $x^a(0)$ at sufficiently large values, so that
either queue does not reach state $0$ during the first few cycles (i.e., before $x^a_3$ is sufficiently close to $u^a_*$). 
Here, congestion collapse means that the queues will have an increasing trend in the sense
that each queue will be larger at the beginning of a cycle than its value at the beginning of the previous cycle.
On the other hand, if the periodic equilibrium is not associated with congestion collapse, i.e., the total average service rate during the periodic cycle
is smaller than the arrival rate,  
then the queues will have a decreasing trend, so that they must eventually reach $0$, regardless of their initial condition.
We conclude that there is no need to actually determine the exact values of the initial queue lengths, or to check wether $\Sigma^a_q = \infty$,
but only to check wether a periodic equilibrium is associated with congestion collapse.

\section{Numerical Examples} \label{secNumeric}

In this section we report the results of numerical experiments based on numerical algorithms (numerical solutions of the dynamical systems) and simulations.
Throughout this section we consider symmetric systems with parameters as in \eqref{symModel}. In all our examples, $\lm = 0.98$, $\tau = 0.01$ and $\kappa = 0.1$,
but we vary the parameters $\theta$ and $\mu$.
The initial condition in the numerical examples is taken in accordance with Assumption \ref{assInit1}.

We emphasize at the outset that $\mu$ in our numerical examples is taken to be extremely small.
(We also consider systems with no abandonment, or with very small abandonment rate, but this is prevalent in modeling.)
However, as our simulation experiments below demonstrate, the oscillating fluid models for systems with extreme parameters
suggest possible bad oscillatory dynamics in systems with more realistic parameters.
In these more realistic setting the behavior cannot be predicted analytically,
since the stochastic system is too complicated. Moreover, oscillations may even be overlooked in practice,
because sufficient abandonment keep the queues relatively small, so that congestion collapse may fail to be noticed.
Thus, we obtain important practical insights by rigorously studying extreme cases.

The rest of this section is organized as follows.
In \S \ref{secExampleNoAbd} we consider a system with no abandonment ($\theta = 0$) and compare the results to the heuristic approximating model in \S \ref{secApproxT1}.
We consider a similar system in \S \ref{secExampleBif} but increase $\mu$ to show that $x^*_0$ is globally asymptotically stable, thus showing
the dependence on $\mu$ of the long-run behavior of the fluid model, as was established in \S \ref{secApproxDS}.
We add abandonment in \S \ref{secExampleWithAbd} in comparison to the system in \S\ref{secExampleNoAbd} to numerically support the reasoning
for the development of the approximating system in \S \ref{secApproxDS}.
Finally, in \S \ref{secSim} we present simulations of stochastic systems for which the fluid limit has no oscillatory solutions, and show
that stochasticity may lead to substantial oscillations.

%

\subsection{A System with No Abandonment} \label{secExampleNoAbd}

We start with a system that has no abandonment, i.e., $\theta = 0$.
The other parameters are $\lm = 0.98$, $\tau = 0.01$, $\kappa = 0.1$ and $\mu = 0.1$.
The initial condition is $q_1(0) = 1$ and $q_2(0) = 1.2$, so that $\Delta(0) = 0.1$.
We further take $z_{1,2}(0) = \tau$ and $z_{2,1}(0) = \tau/2 = 0.005$.

\begin{figure}[h!]
  \hfill
  \begin{minipage}[t]{.4\textwidth}
    \begin{center}
      \includegraphics[scale=0.35]{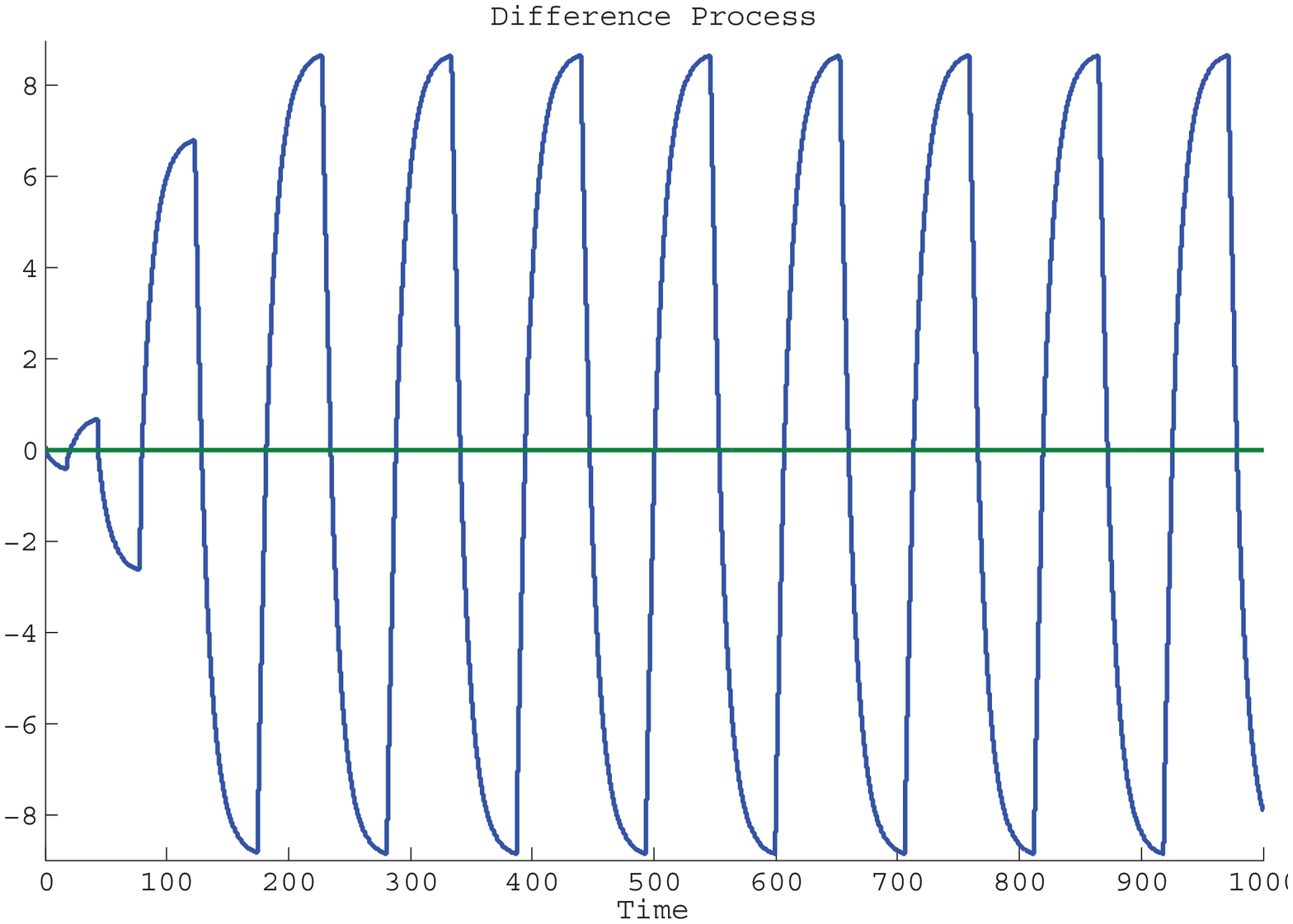}
      \caption{$\Delta$ process; no abandonment and $\mu = 0.1$.}
      \label{figd21NoAbd}
    \end{center}
  \end{minipage}
  \hfill
  \begin{minipage}[t]{.4\textwidth}
    \begin{center}
      \includegraphics[scale=0.35]{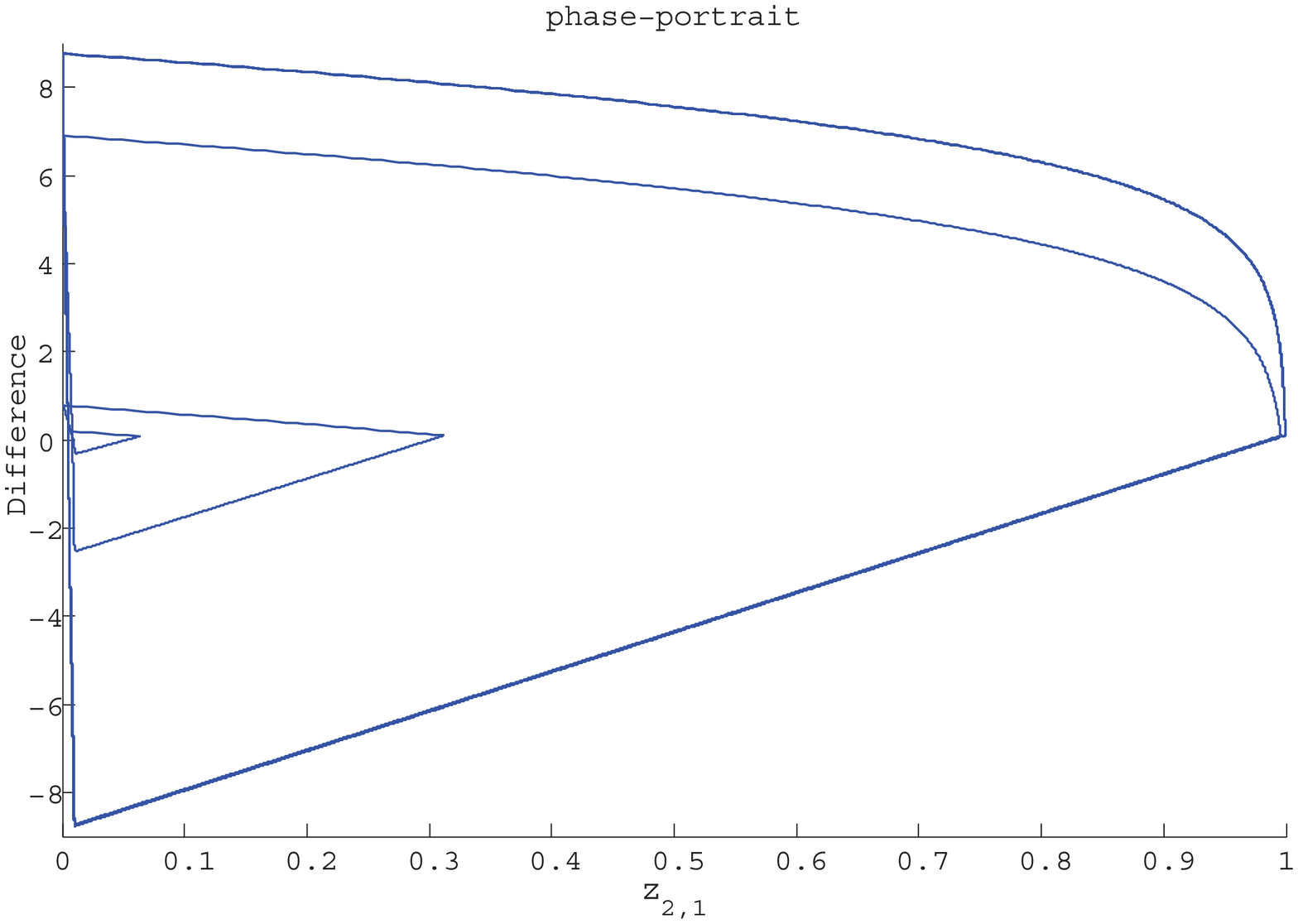}
      \caption{spiraling of $(z_{2,1}, \Delta)$ outward towards the periodic equilibrium; no abandonment and $\mu=0.1$.}
      \label{figPhaseNoAbd}
    \end{center}
  \end{minipage}
  \hfill
\end{figure}

The time-dependent behavior of $\Delta$ is shown in Figure \ref{figd21NoAbd}, whereas
Figure \ref{figPhaseNoAbd} plots the image of $(z_{2,1}, \Delta)$ (with time suppressed).
As can be easily seen from Figure \ref{figd21NoAbd}, there are ten full cycles plotted in this example.
However, there are four loops visible in Figure \ref{figPhaseNoAbd}, with each loop being a full cycle, where
a full cycle begins at a time $t_0$ when $z_{1,2}(t_0)$ hits $\tau$ from above, such that Assumption \ref{assInit1} is satisfied at that hitting time.
In this example, the two variables $(\Delta, z_{2,1})$ spiral outward to the periodic equilibrium, namely, the first cycle is
the inner (smallest) loop, the second cycle is the second smallest loop, etc.
The fact that only four cycles are clearly visible in Figure \ref{figPhaseNoAbd} suggests
that convergence to the periodic equilibrium is extremely fast in terms of the number of periods.
The fast convergence is also visible by in Figure \ref{figd21NoAbd} itself.
Theoretical support for the fast convergence is given in \S \ref{secGeom}.

Of course, the stability of $(\Delta, z_{1,2}, z_{2,1})$ does not imply stability of system.
Indeed, Figure \ref{figQ1NoAbd} suggests that $q_1$ increases without bound, and by symmetry, so is $q_2$.
Figure \ref{figZsNoAbd} shows that a substantial proportion of each pool has fluid from the other class for a non-negligible amount of time, which is the cause
for the congestion collapse observed in Figure \ref{figQ1NoAbd}. See \S \ref{secCC}.

\begin{figure}[h!]
  \hfill
  \begin{minipage}[t]{.4\textwidth}
    \begin{center}
      \includegraphics[scale=0.35]{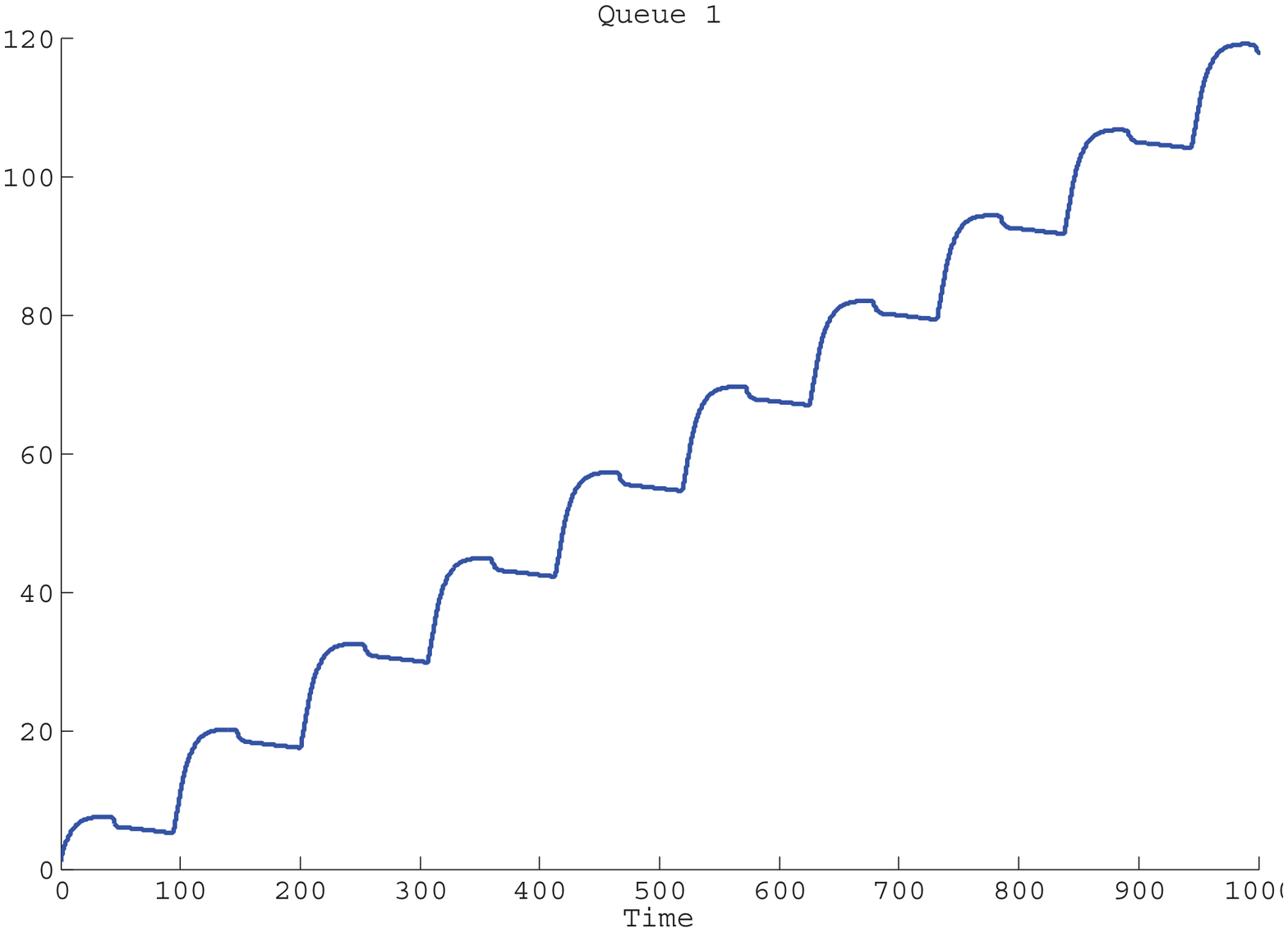}
      \caption{Trajectory of $q_1$, no abandonment.}
      \label{figQ1NoAbd}
    \end{center}
  \end{minipage}
  \hfill
  \begin{minipage}[t]{.4\textwidth}
    \begin{center}
      \includegraphics[scale=0.35]{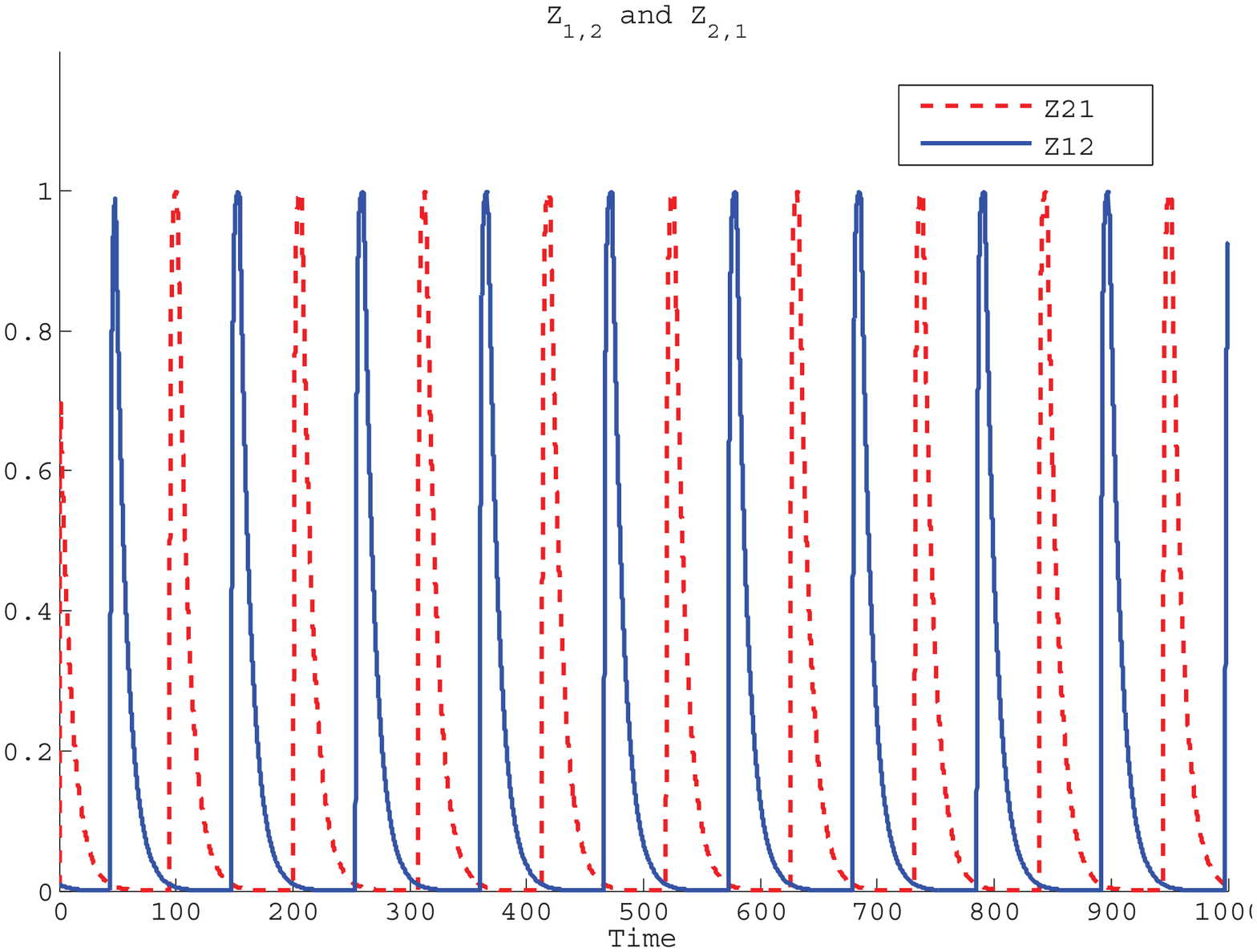}
      \caption{The sharing in both pools, no abandonment.}
      \label{figZsNoAbd}
    \end{center}
  \end{minipage}
  \hfill
\end{figure}

Finally, in Table \ref{Table1} we compare the numerical solution to the iterative algorithm in \S \ref{secAlg}
(in the ``original sys.'' row), to the heuristic approximations developed in \S \ref{secApproxT1}. We note that $L \approx 0.44 < \lm = 0.98$
for $L$ in \eqref{L}.
\begin{table}[h!]
\begin{center}
\begin{tabular}{|l|c|c|c|c|}
\hline
~             & $\Delta(0)$ & $z(T_1)$ & $T_1$    & $T_2$   \\ \hline \hline
approximation  & $8.802$     & $0.9992$ & $7.093$  & $46.044$ \\ \hline
original sys.  & $8.663$     & $0.9992$ & $7.270$  & $46.044$ \\ \hline
\end{tabular}
\caption{comparisons of the values obtained from the iterative algorithm for the approximating system in \S \ref{secApproxDS},
to those of the iterative algorithm in \S \ref{secAlg} for the original system.} \label{Table1}
\end{center}
\end{table}


\subsection{Bifurcation: $\mu = 0.3$} \label{secExampleBif}

The term ``bifurcation'' refers to a change in the equilibrium behavior of a dynamical system as the value of one of its parameters varies, while all other parameters
remain unchanged. Following the analysis in \S \ref{secApproxDS}, we now take the same system considered in \S \ref{secExampleNoAbd} but change the value of $\mu$.
We do not carry out a full bifurcation analysis to find the bifurcation point in which the equilibrium
behavior of the system changes, but instead consider a single value $\mu = 0.3$.
To see how the system converges to the stationary point with no sharing, we change the initial condition in \S \ref{secExampleNoAbd} and take
$\Delta(0) = 20$. The trajectory of $\Delta$ is shown in Figure \ref{figd21Bif}.
(Note however, that we cut the vertical axis in this figure at the value $3$ to make the oscillations more apparent.)
Figure \ref{figPhaseBif} shows the spiraling towards that equilibrium point in the $(z_{2,1}, \Delta)$ plane.
Unlike the case depicted in Figure \ref{figPhaseNoAbd}, now spiraling is ``inward'', i.e.,
the largest loop corresponds to the first cycle, and each of the four cycles is shorter than the previous one.
we remark that the heuristic approximation in \S \ref{secApproxT1} was stopped in the fifth iterations since $\Delta^{(5)} < 0$.
\begin{figure}[h!]
  \hfill
  \begin{minipage}[t]{.4\textwidth}
    \begin{center}
      \includegraphics[scale=0.35]{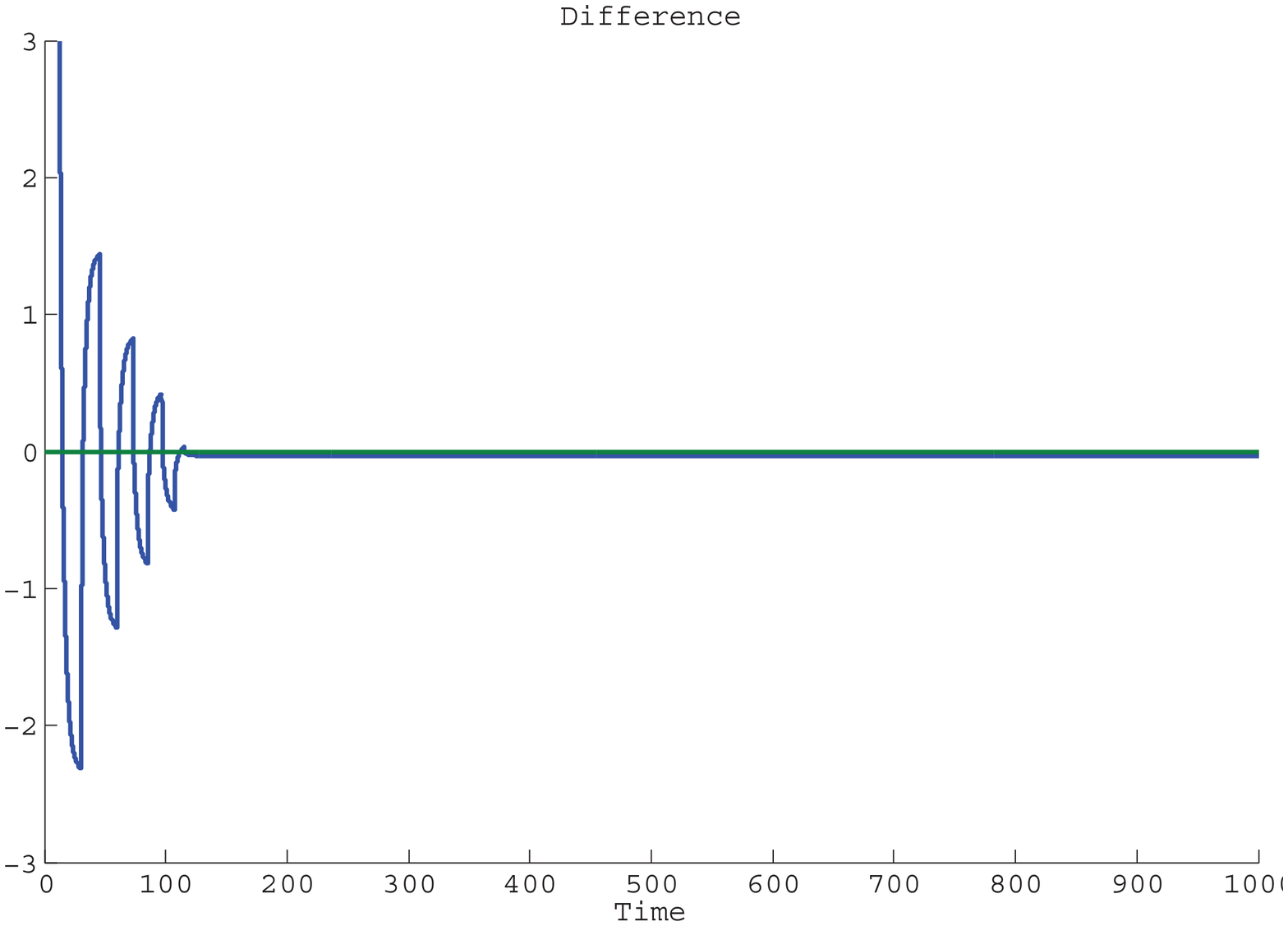}
      \caption{$\Delta$ process, $\mu = 0.3$.}
      \label{figd21Bif}
    \end{center}
  \end{minipage}
  \hfill
  \begin{minipage}[t]{.4\textwidth}
    \begin{center}
      \includegraphics[scale=0.35]{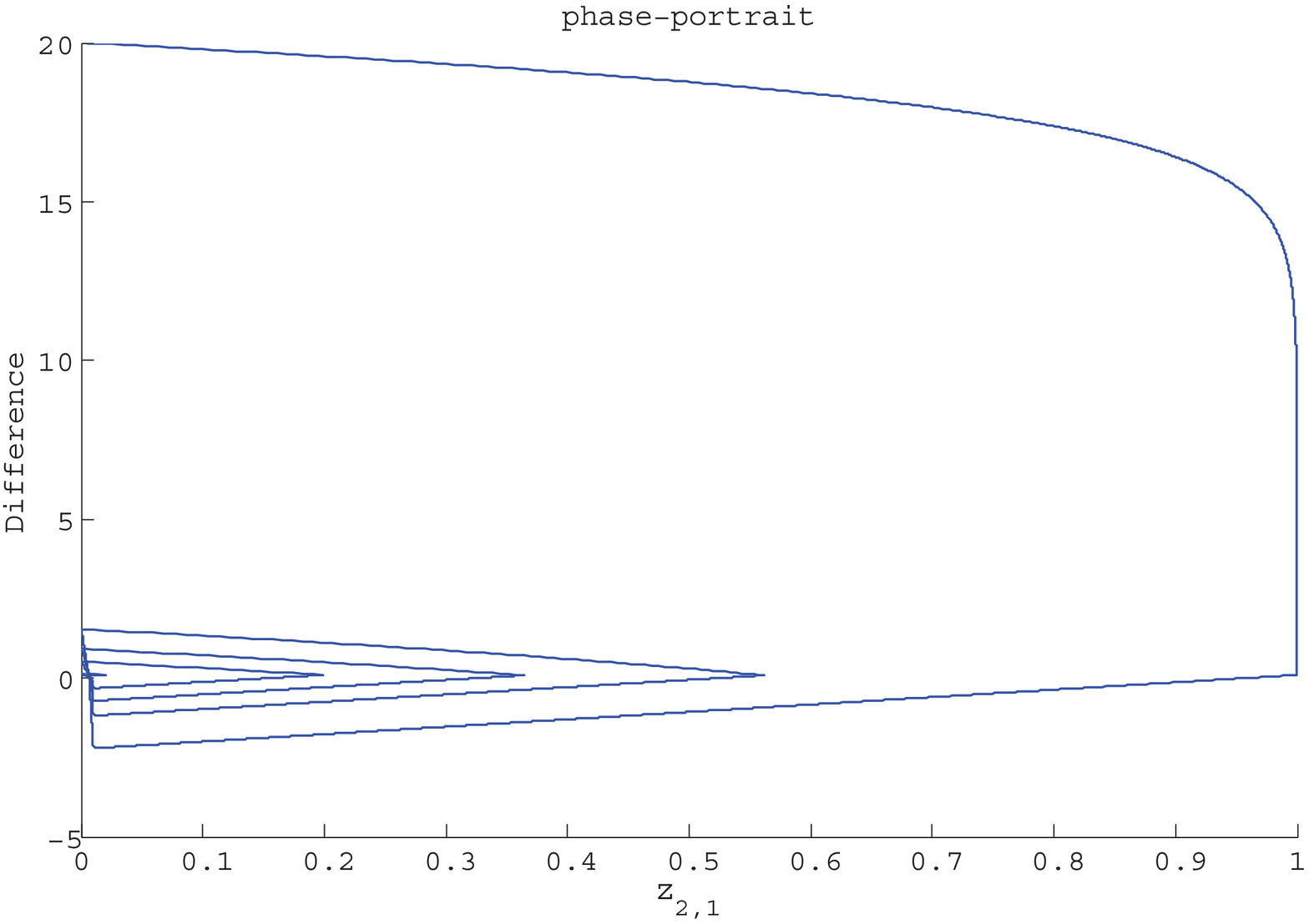}
      \caption{spiraling inward to $x^*_0$, $\mu = 0.3$.}
      \label{figPhaseBif}
    \end{center}
  \end{minipage}
  \hfill
\end{figure}

Observe that even though the convergence to the stationary point is fast in terms of the number of oscillations, it is very slow in
continuous time. In particular, the system oscillates for more than a hundred time units before it ceases to oscillate.

\subsection{Adding Abandonment} \label{secExampleWithAbd}

For a numerical depiction of the approximating solution,
we now consider a system with $\mu = 0.1$ as in \S \ref{secExampleNoAbd} but add abandonment, taking $\theta = 0.01$.
As can be seen by comparing Figures \ref{figd21Abd} and \ref{figPhaseAbd} to Figures \ref{figd21NoAbd} and \ref{figPhaseNoAbd},
the system with no abandonment serves as a reasonable approximation for the a system with a small abandonment rate, but the oscillations are
smaller, as is intuitively expected.

\begin{figure}[h!]
  \hfill
  \begin{minipage}[t]{.4\textwidth}
    \begin{center}
      \includegraphics[scale=0.35]{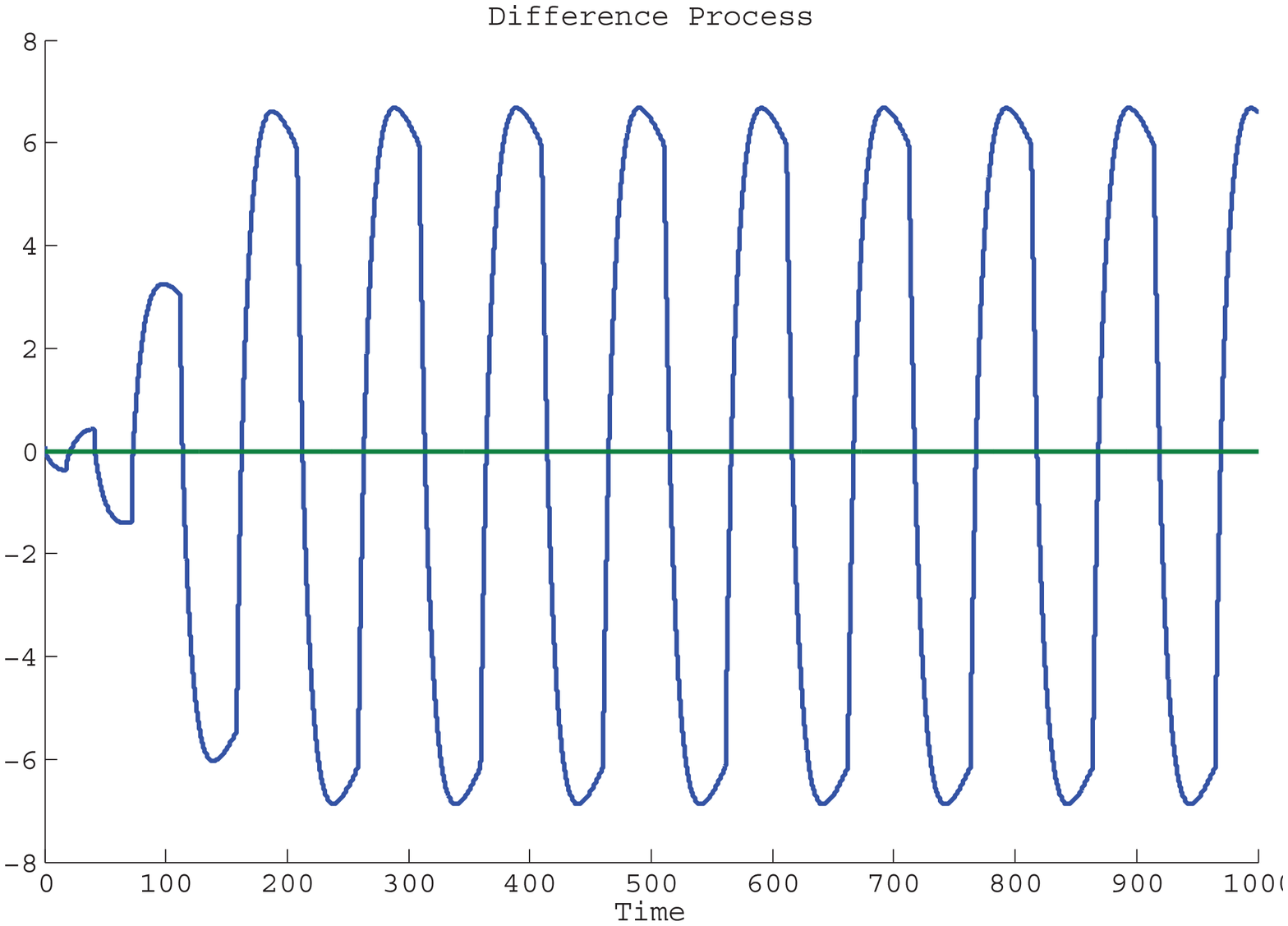}
      \caption{$\Delta$ process; $\mu = 0.1$ and $\theta = 0.01$.}
      \label{figd21Abd}
    \end{center}
  \end{minipage}
  \hfill
  \begin{minipage}[t]{.4\textwidth}
    \begin{center}
      \includegraphics[scale=0.35]{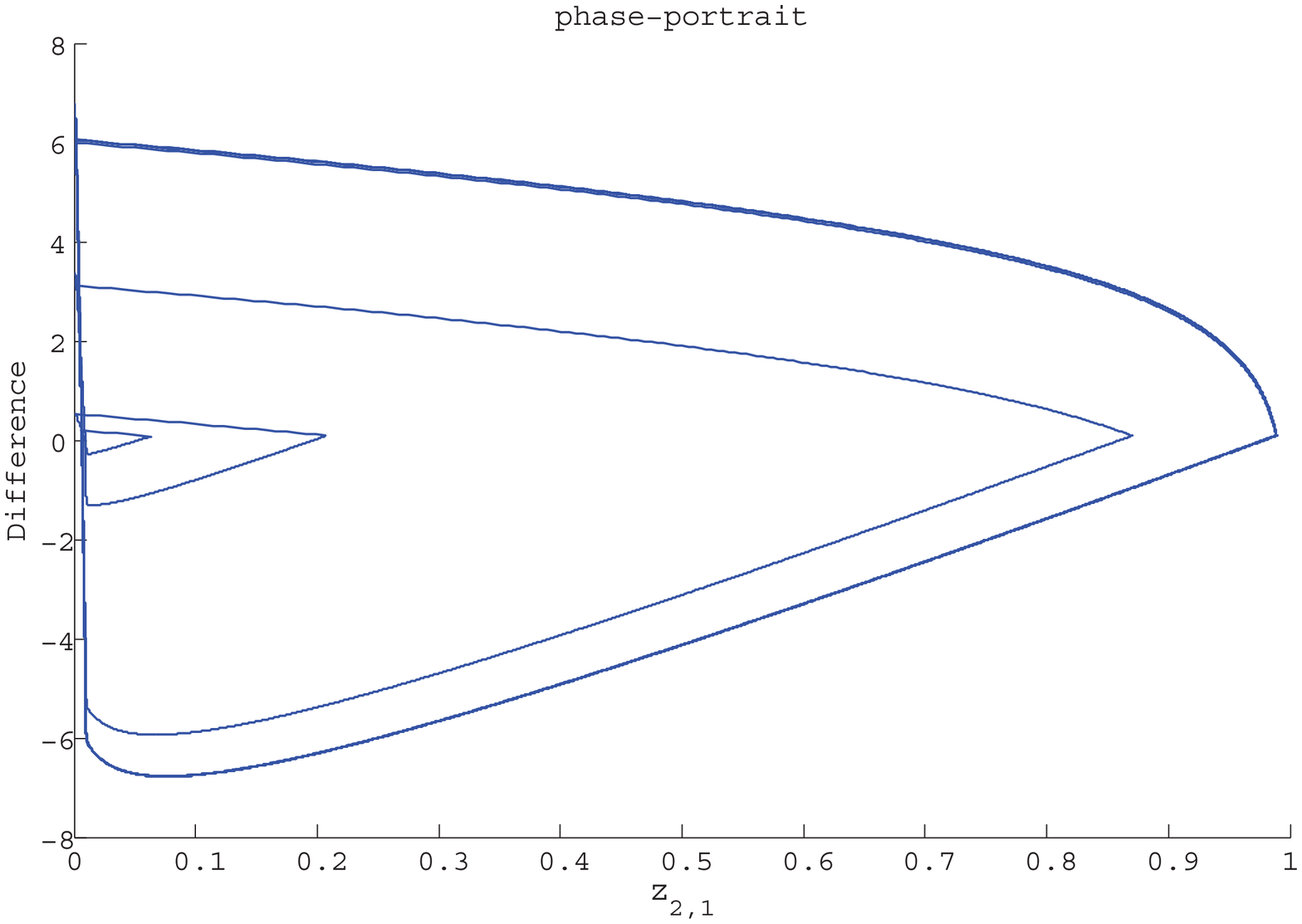}
      \caption{the image of $(z_{2,1}, \Delta)$ spiraling outward to the periodic equilibrium; $\mu = 0.1$ and $\theta = 0.01$.}
      \label{figPhaseAbd}
    \end{center}
  \end{minipage}
  \hfill
\end{figure}

\subsection{Simulations of Systems with non-oscillating Fluid Limits} \label{secSim}

So far we considered the fluid model (limit) alone. The numerical examples above show that congestion collapse can occur
for very extreme parameter values $\mu$ and $\theta$.
In this section we show that the extreme examples provide important insights for cases for which the fluid limit
never oscillates.

It is significant that for a given stochastic system $X^n$ which is approximated by a fluid model $x$, there is freedom
in how to choose the limiting thresholds. For example, if $n = 100$, then activation thresholds $k^n_{i,j} = 10$ can be considered as being
$\sqrt{n}$ or as $0.1n$.
In the latter case, the stochastic fluctuations are considered negligible with respect to the activation thresholds, and $\kappa = 0.1$.
However, in the first case, $\kappa = 0$, and so the stochastic fluctuations are significant. 
Specifically, if $\kappa = 0$, then oscillations are much more likely to occur because $\SS_{1,2} = \SS_{2,1}$ in that case;
see Remark \ref{remUnstableStatPt}.

\subsubsection*{System with a Practically Unstable Stationary Point}


We simulated a system with similar parameters to those in \S \ref{secExampleNoAbd} taking $n = 100$,
so that there are $100$ agents in each pool and $\lm^n = 98$.
As above, $\theta = 0.01$. Since $\kappa^n = 0.1n$, we take $\kappa^n = 10$, which we can also think of as being $\sqrt{n}$, i.e., $\kappa = 0$.

Figures \ref{figQunstable} and \ref{figZunstable} show a single sample path of the $Q^n_1$ process and the shared-customers processes
for a system starting empty.
Due to symmetry of the parameters and the initial condition of the two pools, the fluid model will unambiguously
move through $x^*_0$. Once $x^*_0$ is hit, and since there is no sharing at that hitting time, the fluid model must remain at that point.
However, random noise in the stochastic system causes sharing to begin, leading to extreme oscillations.
From the fluid model perspective, this suggests that random fluctuations (that are negligible in fluid scale)
quickly push the fluid limit from $x^*_0$ to a state $\gamma \in \sO$, leading to fluid-scaled fluctuations.

\begin{figure}[h!]
  \hfill
  \begin{minipage}[t]{.4\textwidth}
    \begin{center}
      \includegraphics[scale=0.35]{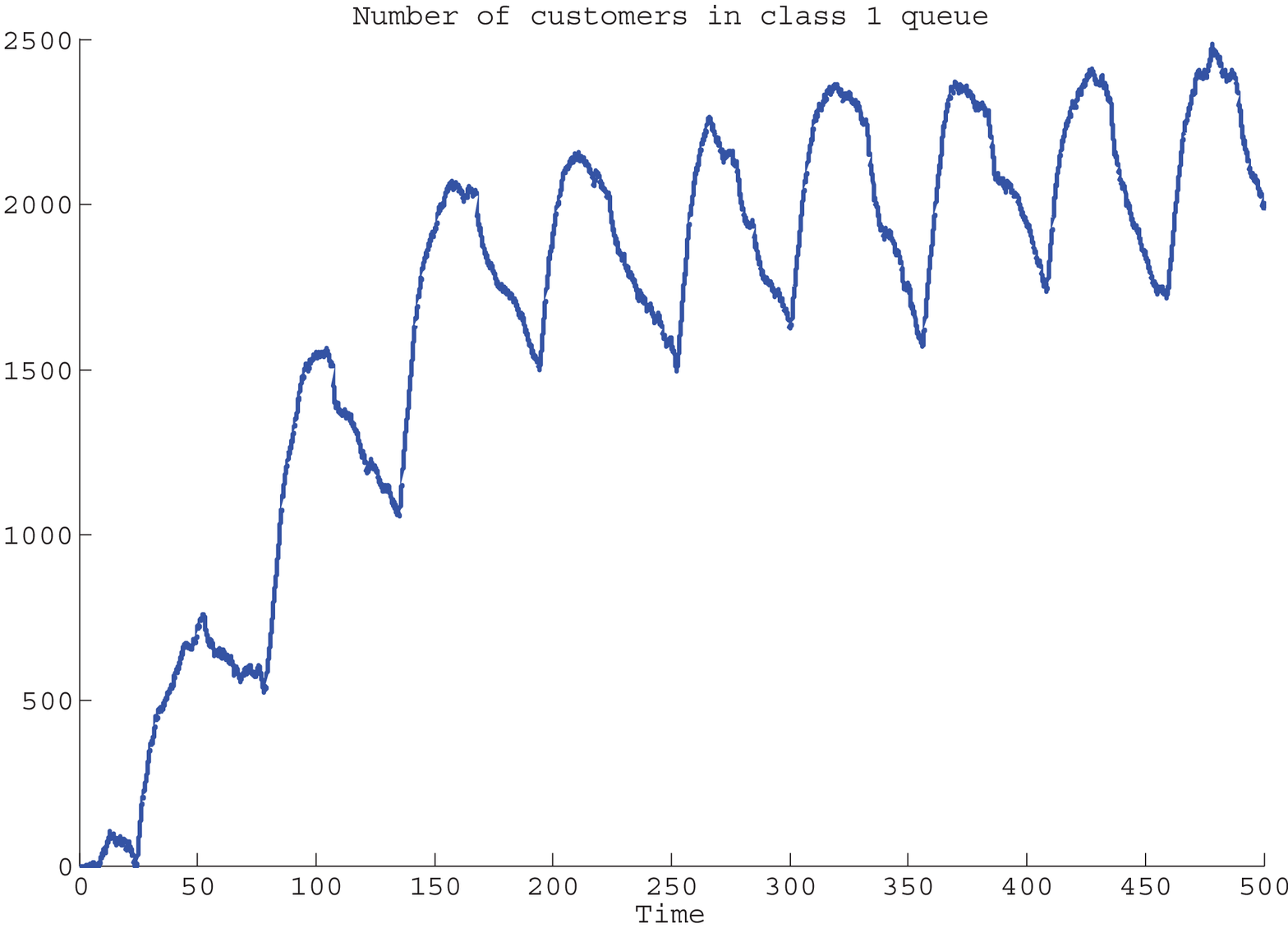}
      \caption{$Q^n_1$ when $\kappa^n = 10$, $\theta = 0.01$, $\mu = 0.1$}
      \label{figQunstable}
    \end{center}
  \end{minipage}
  \hfill
  \begin{minipage}[t]{.4\textwidth}
    \begin{center}
      \includegraphics[scale=0.35]{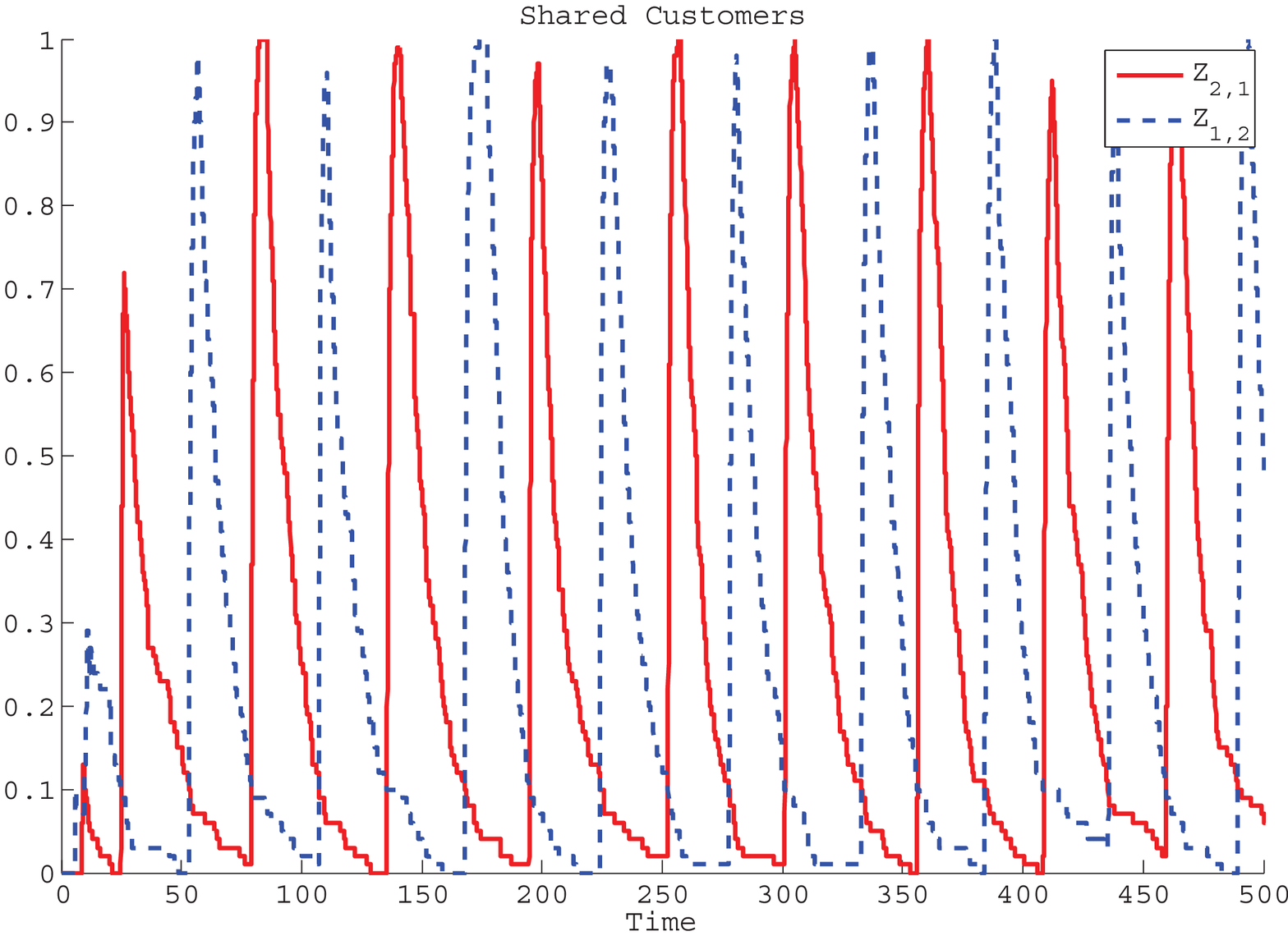}
      \caption{Shared customers in service when $x^*_0$ is unstable; $\kappa^n = 10$, $\theta = 0.01$, $\mu = 0.1$.}
      \label{figZunstable}
    \end{center}
  \end{minipage}
  \hfill
\end{figure}

\subsubsection*{System with no Oscillating Solutions (${\mathbf \sO = \phi}$)}
The fluid model gives important insight that cannot be obtained analytically even for systems with $\sO = \phi$, i.e.,
systems that do not have oscillating fluid limits.
We now take

\noindent $n = 100$: $\lm^n = 98$, $\mu = 0.5$,  $\theta = 0.5$, $\tau^n = 1$ and $k^n_{i,j} = 10$,

\noindent with the rest of the parameters being the same as in \S \ref{secExampleNoAbd}.
The parameters $\theta$ and $\mu$ here are more likely in a practical call-center setting than the parameters in the examples above.

To show that $\sO = \phi$ we solve the fluid model for an extreme example with $q_1(0) = 1$ and $q_2(0) = 1000$,
$z_{1,2} = \tau$ and $z_{2,1} = 0$.
In the simulation however, we have $Z^n_{2,1} = 20$ and $Z^n_{1,2} = 0$, which is a likely initial condition for a system recovering
from an overload in queue $2$. (The initial conditions of the stochastic system and the fluid model do not match because we want to show
that the fluid model does not oscillate, and has no periodic equilibrium.)

Figure \ref{figZnoOsc} shows a single sample path of the shared-customers processes from a single simulation run,
and Figure \ref{figZalwaysStable} shows the fluid model of the system with the initial condition specified above.
We only show figures of the shared customers service process, because both queues monotonically decrease to $0$ in the fluid model,
whereas customer abandonment make the oscillations of the queue processes unobservable in the simulation.
From the practical point of view, this means that oscillations may be hard to detect in real time, unless one knows to look for them.

We note that Figure \ref{figZnoOsc} shows only the time interval $[0,100]$ for clarity, but that the oscillations continued for the full run time of the
simulation, which lasted $1500$ time units. (As before, time here is measured in service time units $\mu_{i,i} = 1$, $i = 1,2$.)

In ending we remark that the bad behavior shown here can be easily avoided by increasing $k^n_{i,j}$, as was discussed in Remark \ref{remUnstableStatPt}.
A numerical example, related to the one given here, is given in Section 4.1 in \cite{PW14}; see Figure 9 in that reference.

\begin{figure}[h!]
  \hfill
  \begin{minipage}[t]{.4\textwidth}
    \begin{center}
      \includegraphics[scale=0.3]{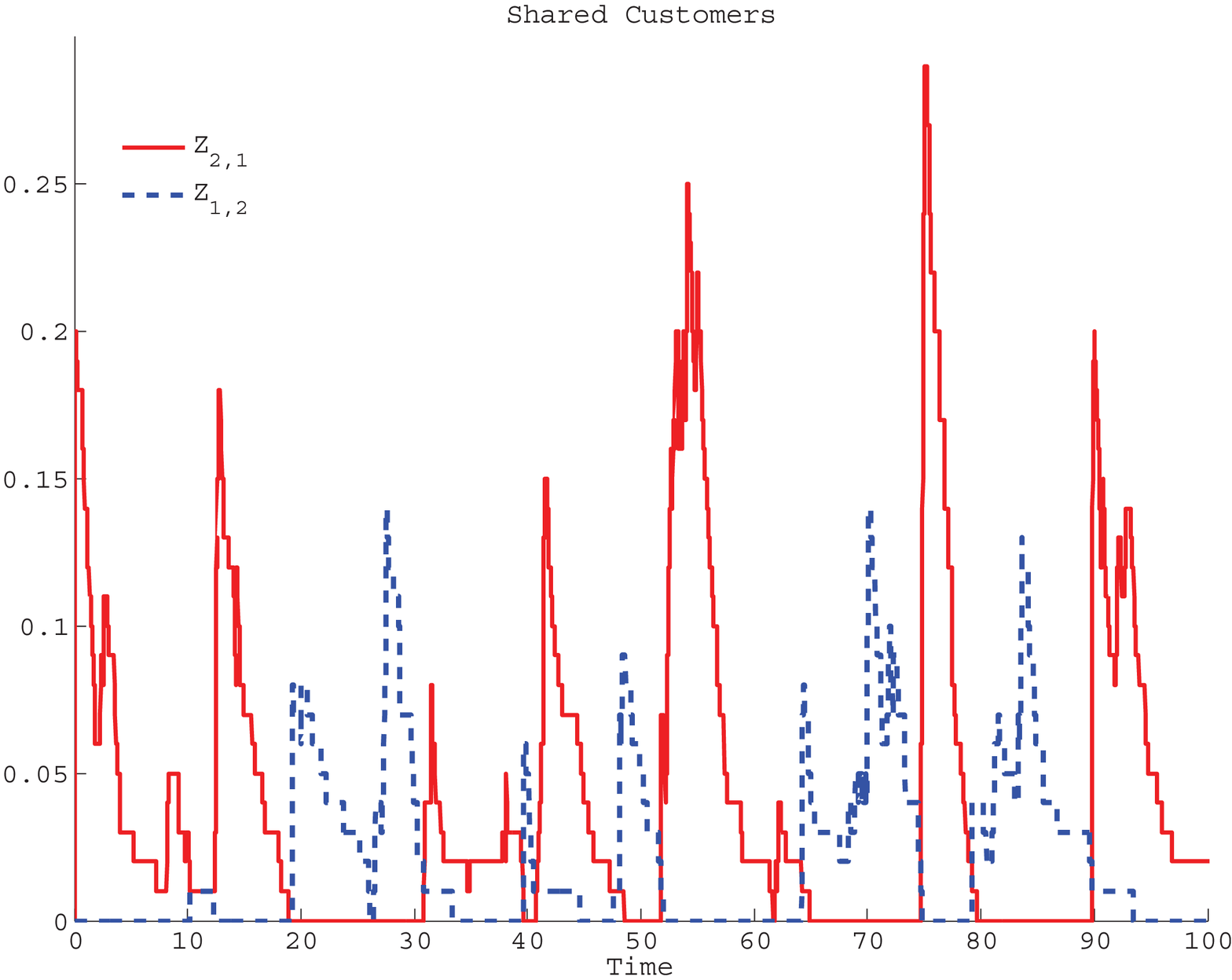}
      \caption{Simulation when $\sO = \phi$; $\theta = \mu = 0.5$}
      \label{figZnoOsc}
    \end{center}
  \end{minipage}
  \hfill
  \begin{minipage}[t]{.4\textwidth}
    \begin{center}
      \includegraphics[scale=0.5]{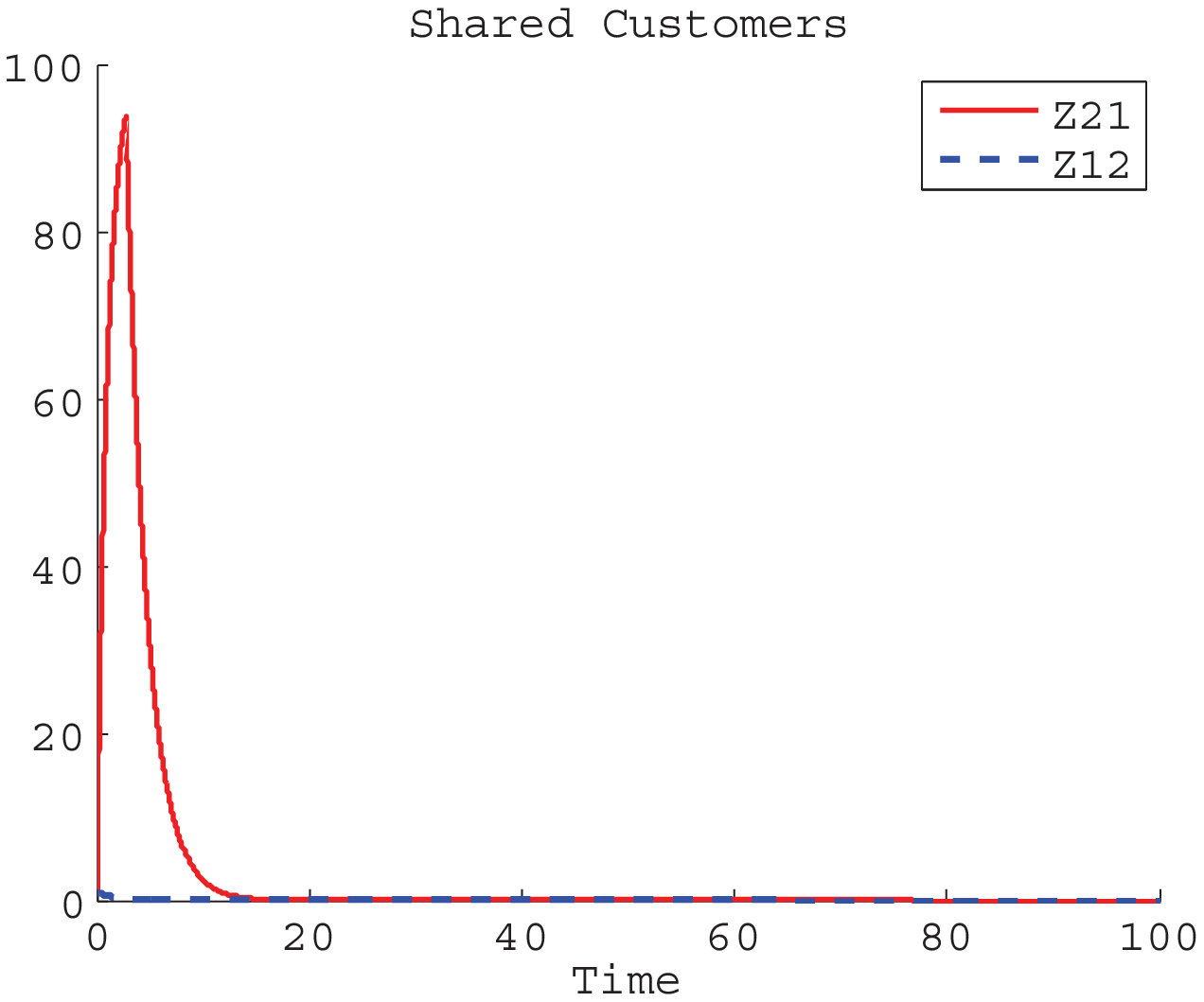}
      \caption{Fluid model when $\sO = \phi$; $\theta = \mu = 0.5$}
      \label{figZalwaysStable}
    \end{center}
  \end{minipage}
  \hfill
\end{figure}

\section{Summary} \label{secConclude}

In this paper we considered the FQR-ART overload control applied to the cyclic X model, when the control parameters are badly chosen.
For the dynamical-system (fluid) limit, the purpose of the control is to attract any fluid trajectory to one of two sliding manifolds during overload
periods, so as to maintain a pre-specified ratio between the two queues.

{\bf Switching Fluid Limit. }We have shown that possible delays in activation and release of the control
can lead to chattering and resulting oscillations, which translates to fluid-scaled fluctuations in the underlying stochastic system.
The pathological oscillatory behavior can be analyzed via a switching dynamical system, as in Definition \ref{DefFluidEq},
within the framework of the many-server heavy-traffic FWLLN (Theorem \ref{thFWLLN} in \S \ref{secFWLLN}).
Theorems \ref{thUniqueStatPt} and \ref{thPeriodic}, respectively, prove that the fluid limit has a unique stationary point and 
a non-trivial periodic
equilibrium that is associated with the oscillatory motion.
Sufficient conditions for endless oscillations were provided in Theorem \ref{thEndless}.

{\bf Fluid Stability. }In Theorem \ref{thStableStatPt} it was shown that any fluid trajectory that ceases to oscillate must converge to the unique stationary point.
A convenient approximating dynamical system to the fluid limit was developed and shown to be bi-stable in \S \ref{secApproxDS}.
Specifically, all the trajectories of the approximating system were shown to converge to one of the two equilibria
-- the stationary point $x^*_0$ in \eqref{statPt}, or a unique non-trivial periodic equilibrium.
Finally, a simple heuristic construction in \S \ref{secApproxT1} can be used to approximate the values of the solutions to \eqref{SwitchODE}
at the switching times, and in particular, the values of the periodic equilibrium at the switching times, when it exists.

{\bf Implications. }Numerical examples in \S \ref{secNumeric} show the effectiveness of the approximating system.
The simulation experiment in \S \ref{secSim} demonstrates that our fluid model provides important insights
into the untractable behavior of the underlying stochastic system,  even when the fluid approximation itself is not oscillating.
Further implications of the results to the stochastic system are considered in the appendix.

From the practical perspective, the most important conclusion is that the control parameters must be chosen with caution. 
For example, the bad oscillatory behavior presented in \S \ref{secSim} (which may be hard to detect in real time) 
can be avoided by choosing appropriate activation thresholds. 
We again refer to \cite{PW14} for further a discussion.


\newpage

\begin{appendix}

\begin{center}
\large{APPENDIX}
\end{center}

This appendix contains supplementary material for the main paper.
First, in \S \ref{secNotation} we give notation for sets used in the paper.
In \S \ref{secBoundsOSC} we establish bounds on the component functions in the state vector
$x$ to guarantee oscillating behavior.
In \S \ref{secGeom} we establish stronger forms of convergence of solutions to the approximating system to their equilibrium behavior.
In particular, we show that the iterative algorithm in \S \ref{secPfde} converges geometrically fast, and conclude that the approximating solutions
converge exponentially fast to equilibrium.
In \S \ref{secAsymptotic} we show that the fluid model we considered in the main paper arises
as the fluid limit in a many-server heavy-traffic fluid limit of the underlying model. The proof of the FWLLN is given in \S \ref{secFWLLN},
after a brief expansion on the stochastic model and many-server scaling in \S \ref{secReview}.
Finally, in \S \ref{secImp} we discuss implications of our results here for the control of the stochastic system.


\section{Notation of Sets}\label{secNotation}
Below is a list of the different sets that appear in the paper. Their first appearance is in parenthesis.
\bi
\item $\SS^*$ -- the set of all stationary points (\S \ref{secStatPt}).
\item $\sM$ -- switching (or sliding) manifold in a general system (\S \ref{secIntro}).
\item $\sO$ -- the invariant set of oscillating solution, i.e., if $x(0) \in \sO$, then $x$ oscillates indefinitely (\S \ref{secStatPt}).
\item $\sP_{u^*}$ -- the image of the periodic equilibrium $u^*$ (\S \ref{secQualitative}).
\item $\SS \equiv [0,\lm/\theta]^2 \times [0,1]^4$ -- the state space of the fluid model (\S \ref{secFluidModel}).
\item $\SS_{i,j}$ -- the sliding manifold where $d_{i,j} = \kappa$ (\S \ref{secFluidModel})
\item $\sS_{u^*}$ -- the stability region of the periodic equilibrium $u^*$ (\S \ref{secQualitative}).
\item $\sS_{x^*}$ -- the stability region of a stationary point $x^*$ (\S \ref{secStatPt}).
\item $\sS_{x_0^*}$ -- the stability region of the stationary point $x^*_0$ in \eqref{statPt} (Theorem \ref{thUniqueStatPt}).
\item $\SS_\ep \equiv [\ep, \lm/\theta]^2 \times [0,\tau]$, $\ep > 0$ -- the state space of of solutions in $\sO$ (\S \ref{secProofEndlessA})
\item $\SS_\kappa \equiv [\kappa + \ep_\kappa, \lm/\theta] \times [0,\tau]$, where $\ep_\kappa > 0$ (Proof of existence part of Theorem \ref{thPeriodic})
\item $\SS^a \equiv [0,\infty)^2 \times [0,1]^4$ -- the state space of the approximating system (\S \ref{secApproxDS}).
\item $\SS_\mu \equiv [\Delta_\mu^M - \delta_\mu, \Delta_\mu^M]$, where $\Delta_\mu^M$ is defined in \eqref{DelaBd} and $\delta_\mu$ in \eqref{delta_mu}
(Equation \eqref{SSmu} in \S \ref{secPff}).
\ei

\section{Bounds to Guarantee Oscillations} \label{secBoundsOSC}

We now provide supporting details for the proof of Theorem \ref{thEndless}, providing sufficient conditions for endless oscillations
of solutions to \eqref{SwitchODE} and congestion collapse.
In \S\S \ref{secBddT1} and \ref{secBddT2} we construct simple bounds on $T_1$ and $x(T_1)$, and bounds on $T_2$ and the values of $x$ over $[\Sigma_1, \Sigma_2)$, respectively.
Universal bounds on the solution $x$ and the holding times, and a numerical example, are given in \S \ref{secUniversalBd}.
Finally, in \S \ref{secTighterBd} we show that, after ensuring that a solution oscillates indefinitely,
we can apply Theorem \ref{thEndless} to obtain tighter bounds
on the values of $\Delta$ at switching epochs.

\subsection{Bounds on $T_1$ and $x(T_1)$} \label{secBddT1}


We can apply \eqref{DelBd} to obtain bounds on $T_1$.

\begin{coro}{$($bounds on $T_1)$}\label{corT1bdsFirst}
Under the initial conditions in Assumption \ref{assInit1}, the interval end time $T_1$ is bounded above and below by
\beql{T1T2bdFirst3}
0 < \frac{\theta \kappa + \Psi_L}{\theta \Delta (0) + \Psi_L} \le e^{-\theta T_1} \le \frac{\theta \kappa + \Psi_U}{\theta \Delta (0) + \Psi_U} < 1,
\eeq
for $\Psi_L$ and $\Psi_U$ in {\em \eqn{PsiBds}}, from which we deduce that
$$1 < \frac{\theta \Delta (0) + \Psi_U}{\theta \kappa + \Psi_U} \le e^{\theta T_1} \le \frac{\theta \Delta (0) + \Psi_L}{\theta \kappa + \Psi_L} < \infty,$$
and
$$0 < \log{\left(\frac{\theta \Delta (0) + \Psi_U}{\theta \kappa + \Psi_U}\right)} \le \theta T_1
\le \log{\left(\frac{\theta \Delta (0) + \Psi_L}{\theta \kappa + \Psi_L}\right)} < \infty.$$
The associated bounds on $T_1$, denoted by $T_1^L \equiv T_1^L (\Delta (0))$ and $T_1^U \equiv T_1^U (\Delta (0))$, are both strictly increasing functions of $\Delta (0)$,
both approaching $0$ as $\Delta (0) \downarrow \kappa$ and $\infty$ as $\Delta (0) \uparrow \infty$. In particular,
\beas
T_1^L & \equiv & \left(\frac{1}{\theta}\right) \log{\left(\frac{\theta \Delta (0) + \Psi_U}{\theta \kappa + \Psi_U}\right)}
= \left(\frac{1}{\theta}\right) \log{\left( 1 + \frac{\Delta (0) - \kappa}{(\Psi_U/\theta) + \kappa}\right)} \le \frac{\Delta (0) - \kappa}{\Psi_U + \theta \kappa}
\eeas
and
\beas 
T_1^U & \equiv & \left(\frac{1}{\theta}\right) \log{\left(\frac{\theta \Delta (0) + \Psi_L}{\theta \kappa + \Psi_L}\right)}
= \left(\frac{1}{\theta}\right) \log{\left( 1 + \frac{\Delta (0) - \kappa}{(\Psi_L/\theta) + \kappa}\right)} \le \frac{\Delta (0) - \kappa}{\Psi_L + \theta \kappa}
\eeas 
so that
\beas 
0 < T_1^U - T_1^L & = & \left(\frac{1}{\theta}\right) \left(\log{\left( 1 + \frac{\Delta (0) - \kappa}{(\Psi_L/\theta) + \kappa}\right)}
-  \log{\left( 1 + \frac{\Delta (0) - \kappa}{(\Psi_U/\theta) + \kappa}\right)}\right) \nonumber \\
& = &  \left(\frac{1}{\theta}\right)\left(\log{\left(\frac{\theta \Delta (0) +\Psi_L}{\theta \kappa  +\Psi_L}\right)\left(\frac{\theta \kappa  +\Psi_U}{\theta \Delta (0) +\Psi_U}\right)}\right).
\eeas 
\end{coro}

\begin{proof}
Exploit \eqref{DelBd} with the equation $\Delta (T_1) = \kappa$ characterizing $T_1$.
\end{proof}

The bounds we have just obtained on $T_1$ can be used to obtain bounds on $q_1 (T_1)$.  Recall that $\kappa < \Delta (0)$ and
$\Psi_L < \Psi_U < 0$.
Applying \eqref{T1T2bdFirst3} with \eqref{queuesT1exp2}, we immediately obtain
\begin{coro}{$($bounds on $q_1 (T_1))$}\label{corQ1bd}
$q_1 (t)$ is bounded from below by $q_1^L$ and from above by $q_1^U$, where, for $\Psi_L$ and $\Psi_U$ in {\em \eqn{PsiBds}},
\beas 
0 < q_1^L (T_1) & \equiv & \frac{\lambda}{\theta} - \left(\frac{\lambda}{\theta} - q_1 (0)\right)\left(\frac{\theta \kappa + \Psi_L}{\theta \Delta (0) + \Psi_L}\right) \\
& \le & q_1 (T_1) \le  \frac{\lambda}{\theta} - \left(\frac{\lambda}{\theta} - q_1 (0)\right)\left(\frac{\theta \kappa + \Psi_U}{\theta \Delta (0) + \Psi_U}\right) \equiv q_1^U (T_1)
< \infty. \nonumber
\eeas 
\end{coro}

Similarly, Applying \eqn{poolsT1}, we have 

\begin{coro}{$($bounds on $z_{2,1} (T_1))$}\label{corz21bd}
\beas 
0 < z^{L}_{2,1} (T_1)) & \equiv & 1 - e^{-T_1} < z_{2,1} (T_1)) < 1 - (1-\tau) e^{-T_1} \equiv z^{U}_{2,1} (T_1)) < 1.
\eeas 
\end{coro}

\subsection{Bounds on $T_2$ and $\{x(t) : T_1 \le t \le T_1 + T_2)$} \label{secBddT2}

For bad oscillatory behavior, we will want to see that $q_2 (T_1 + t)$ remains positive and, furthermore that $d_{2,1} < 0$.
to ensure that the initial conditions in Assumption \ref{assInit1} hold at the switching time $\Sigma_2 \equiv T_1 + T_2$ with the index labels reversed.
From Corollary \ref{corQ1bd}, we obtain the following

\begin{coro}{$($lower bounds on the queue lengths on $[T_1, T_1 + T_2)$$)$}\label{corLowerQbds}
\bes 
q_2 (T_1) - \kappa = q_1 (T_1 ) \ge q_1^L (T_1) = \frac{\lambda}{\theta} - \left(\frac{\lambda}{\theta} - q_1 (0)\right)\left(\frac{\theta \kappa +2}{\theta \Delta (0) +2}\right),
\ees 
so that, for $i = 1,2$,
\beas 
q_i (T_1 + t) & \ge & q_1^L (T_1) e^{-\theta t} - \left(\frac{1-\lambda}{\theta}\right)(1 - e^{-\theta t}) \nonumber \\
&& \quad =  \left(\frac{\lambda}{\theta} - \left(\frac{\lambda}{\theta} - q_1 (0)\right)\left(\frac{\kappa +2}{\Delta (0) +2}\right)\right)e^{-\theta t}
- \left(\frac{1-\lambda}{\theta}\right)(1 - e^{-\theta t}),
\eeas 
which is a strictly decreasing function of $t$.
As a consequence, a sufficient condition for both $q_1 (t)$ and $q_2 (t)$ to remain positive throughout $[T_1, T_1 + T_2]$ is for
\bes 
\left(\frac{\lambda}{\theta} - \left(\frac{\lambda}{\theta} - q_1 (0)\right)\left(\frac{\theta \kappa +2}{\theta \Delta (0) +2}\right)\right)e^{-\theta T_2}
> \left(\frac{1-\lambda}{\theta}\right)(1 - e^{-\theta T_2}),
\ees 
for which a sufficient condition is
\bes 
\left(\frac{\lambda}{\theta} - \left(\frac{\lambda}{\theta} - q_1 (0)\right)\left(\frac{\theta \kappa +2}{\theta \Delta (0) +2}\right)\right)e^{-\theta T^U_2}
> \left(\frac{1-\lambda}{\theta}\right)(1 - e^{-\theta T^U_2}),
\ees 
where
\bes 
T^U_2 \equiv \frac{\log_{e}{([1 - (1 - z_{2,1} (0))e^{-T^U_1}]/\tau)}}{\mu} \le \frac{\log_{e}{([1 - (1 - \tau)e^{-T^U_1}]/\tau)}}{\mu}.
\ees 
for $T^U_1$ in Corollary {\em \ref{corT1bdsFirst}}.
\end{coro}

\subsection{Universal Bounds} \label{secUniversalBd}

We now consider the performance over a range of initial conditions.
First, we introduce lower and upper bounds on the initial difference $\Delta (0) \equiv q_2 (0) - q_1 (0)$.
We assume that
\bequ \label{initDiffBds}
0 < \kappa < \Delta_L (0) \le \Delta (0) \le \Delta_U (0) < \infty
\eeq
uniformly enforcing Assumption \ref{assInit1}.
We also assume that the smaller queue length is bounded below and above by
\beql{initQ1bd}
0 < q_1^L (0)  \le q_1 (0) \le q_1^U (0) < \frac{\lambda}{\theta} < \infty,
\eeq
again uniformly enforcing Assumption \ref{assInit1}.

Now let $T_1^{L*}$ be the lower bound $T_1^L$ for $T_1$ in Corollary \ref{corT1bdsFirst} when $\Delta (0) = \Delta_L (0)$
and let $T_1^{U*}$ be the lower bound $T_1^U$ for $T_1$ in Corollary \ref{corT1bdsFirst} when $\Delta (0) = \Delta_U (0)$.

\begin{lemma}{$($universal bounds on $T_1)$}\label{lmT1un}
For all initial conditions satisfying {\em \eqn{initDiffBds}} and {\em \eqn{initQ1bd}},
\bes 
0 < T_1^{L*} \le T_1 \le T_1^{U*} < \infty.
\ees 
\end{lemma}

\begin{proof}
Apply Corollary \ref{corT1bdsFirst}.
\end{proof}

\begin{lemma}{$($universal bounds on $z_{2,1} (T_1)$ and $T_2)$}\label{lmT1un2}
If, together with {\em \eqn{initDiffBds}} and {\em \eqn{initQ1bd}},
\beql{un2}
1 - e^{-T_1^{L*}} > \tau,
\eeq
then
$$1 - e^{-T_1} > \tau, \quad
\tau <  z_{2,1} (T_1^{L*}) \le z_{2,1} (T_1) \le z_{2,1} (T_1^{U*})$$
and
\beql{un5}
T_2^{L*} \equiv \frac{\log_e{(z_{2,1} (T_1^{L*})/\tau)}}{\mu} \le T_2 \le \frac{\log_e{(z_{2,1} (T_1^{U*})/\tau)}}{\mu} \equiv T_2^{U*}
\eeq
for all initial conditions satisfying {\em \eqn{initDiffBds}} and {\em \eqn{initQ1bd}}.
\end{lemma}

\begin{proof}
Apply \eqn{poolsT1} and \eqn{T2formula} together with Lemma \ref{lmT1un2}.
\end{proof}

If a periodic equilibrium exists, then the value of $z_{1,2} (\Sigma_2)$ will equal to $z_{2,1}(\sigma_2)$ on that equilibrium,
as explained below \eqref{switchTimes} in \S \ref{secOscFluid}. See also \eqref{exist} in Theorem \ref{thPeriodic}.
We put the results above together to obtain bounds on $z_{1,2} (T_1 + T_2)$, which will serve as the new value of $z_{2,1} (0)$
in a continuation of the algorithm beyond time $\Sigma_2 = T_1 + T_2$.

\begin{lemma}{$($universal bounds on $z_{1,2} (\Sigma_2)$}\label{lmT1un3}
If conditions {\em \eqn{initDiffBds}}, {\em \eqn{initQ1bd}} and {\em \eqn{un2}} hold, then
$$0 <  z^{L*}_{1,2} (T_1 + T_2) \equiv e^{- \mu T_1^{U*}}z_{2,1} (T_1^{U*}) \le z_{1,2} (T_1 + T_2) \le e^{- \mu T_1^{L*}}z_{2,1} (T_1^{L*}) \equiv z^{U*}_{1,2} (T_1 + T_2) < \tau$$
for all initial conditions satisfying {\em \eqn{initDiffBds}} and {\em \eqn{initQ1bd}}.
\end{lemma}

\begin{proof}
Apply \eqn{z12End} together with the lemmas above.
\end{proof}

Next we consider the queue lengths at time $T_1 + T_2$.  
\begin{lemma}{$($universal lower bounds on the queue lengths at time $T_1 + T_2)$}\label{lmT1un4}
If {\em \eqn{initDiffBds}}, {\em \eqn{initQ1bd}} and {\em \eqn{un2}} hold, then
$$q_2 (T_1) - \kappa = q_1 (T_1 ) \ge q_1^{L*} (T_1) \equiv \frac{\lambda}{\theta} -
\left(\frac{\lambda}{\theta} - q^L_1 (0)\right)\left(\frac{\theta \kappa +2}{\theta \Delta^L (0) +2}\right),$$
for all initial conditions satisfying {\em \eqn{initDiffBds}} and {\em \eqn{initQ1bd}}, where
$q^L_1 (0)$ and $\Delta^L (0)$ are given in {\em \eqn{initDiffBds}} and {\em \eqn{initQ1bd}}.
If, in addition,
\beql{un8}
q_1^{L*} (T_1+ T_2) \equiv q_1^{L*} (T_1) e^{-\theta T_2^{U*}} > \left(\frac{1-\lambda}{\theta}\right)(1 - e^{-\theta T^{U*}_2}),
\eeq
then the two queue lengths $q_1 (t)$ and $q_2 (t)$ remain positive throughout $[T_1, T_1 + T_2]$
for all initial conditions satisfying {\em \eqn{initDiffBds}} and {\em \eqn{initQ1bd}}.
\end{lemma}

\begin{proof}
Apply Corollary \ref{corLowerQbds} and \eqn{un5}.
\end{proof}

Finally, we obtain lower and upper bounds on the queue difference at time $T_1 + T_2$.
\begin{lemma}{$($universal bounds on the queue difference at time $T_1 + T_2)$}\label{lmT1un4}
If conditions {\em \eqn{initDiffBds}}, {\em \eqn{initQ1bd}} and {\em \eqn{un2}} hold, then
\begin{eqnarray}\label{un9}
\Delta_L (T_1 + T_2) & \equiv & \kappa e^{-\theta T_2^{U*}} - A_U \left(\frac{e^{-\theta T_2^{L*}} - e^{-\mu T_2^{U*}}}{\mu - \theta}\right) \nonumber \\
&& \quad \le \Delta (T_1 + T_2) \le \Delta_U (T_1 + T_2) \equiv  \kappa e^{-\theta T_2^{L*}} - A_L \left(\frac{e^{-\theta T_2^{U*}} - e^{-\mu T_2^{LU*}}}{\mu - \theta}\right)
\end{eqnarray}
for all initial conditions satisfying {\em \eqn{initDiffBds}} and {\em \eqn{initQ1bd}},
where $T_2^{L*}$ and $T_2^{U*}$ are given in {\em \eqn{un5}} and
$$A_L \equiv (1- \mu)(z_{1,2}^{L*} (T_1) -  z^{U*}_{2,1} (T_1)) \le A \le (1- \mu)(z_{1,2}^{U*} (T_1) -  z^{L*}_{2,1} (T_1)) \equiv A_U$$
for $A$ in {\em \eqn{queuesT2b}}.
\end{lemma}

\paragraph{A Numerical Example.}
Consider the bounds in Lemma \ref{lmT1un4}. Since $\kappa$ is taken to be relatively small,
\bes
\Delta_L(T_1+T_2) \approx A_U \left(\frac{e^{-\theta T_2^{L*}} - e^{-\mu T_2^{U*}}}{\mu - \theta}\right).
\ees
In addition, $A_U \le (1-\mu)(\tau-1)$, so that, for given $\mu$ and $\tau$, $A$ in this lemma is bounded from above by a constant.
These observations help to determine an initial value $\Delta_L(0)$ for which \eqref{initDiffBdsB} will be satisfied.
For the same parameters in \S \ref{secNumeric} $\mu = 0.1$, $\lm = 0.98$, $\tau = 0.01$, $\kappa = 0.1$ and $\theta = 0.01$, the constant bound of $A_U$
is $-0.891$ and $\Delta_L(T_1 + T_2) \ge 6.21$.
Hence, \eqref{initDiffBdsB} holds for some values of $\Delta(0)$ in the interval $(\kappa, 6.21)$.
For example, taking $\Delta_L(0) = 4$, $\Delta_U(0) = 7$ and $q^L_1(0) = 1$, we obtain
$\Delta_L(\Sigma_1) \approx 6 > \Delta_L(0)$ and $q_1^L(\Sigma_1) = 1.8 > q_1^L(0)$.

\subsection{Tighter Bounds} \label{secTighterBd}

We can apply Theorem \ref{thEndless} to obtain tighter bounds on the queue difference associated with each successive iteration.
Let $\Delta^{(n)} (0)$ be $q_2 (0) - q_1 (0)$ at the beginning of the $n^{\rm th}$ iteration, so that
we start with $\Delta^1 (0) = \Delta (0)$.  Let $\Delta_{L}^n $ and $\Delta_{U}^n$ be the lower and upper bound on $\Delta^{(n)} (0)$, respectively,
so that $\Delta_{L}^{(1)}  = \Delta_{L}$ and $\Delta_{U}^{(1)}  = \Delta_{U}$.

We exploit the fact that, under the conditions of Theorem \ref{thEndless}, we can let
$\Delta_{L}^{(2)}  = \Delta_{Le}$ and $\Delta_{U}^{(2)}  = \Delta_{Ue}$.
We can thus apply mathematical induction to deduce the following corollary.
\begin{coro}{$($nested bounds$)$}\label{corNest}
Under the conditions of Theorem \ref{thEndless},
\bes 
\Delta_{L}^{(n)} \le \Delta^{(n)} (0) \le \Delta_{U}^{(n)} \qforallq n \ge 1,
\ees 
where $\{\Delta_{L}^{(n)}: n \ge 1\}$ is a strictly increasing sequence with finite upper limit $\Delta_{L}^{\infty}$
and $\{\Delta_{U}^{(n)}: n \ge 1\}$ is a strictly decreasing sequence with limit $\Delta_{U}^{\infty}$ such that, for $n > 2$,
\beas 
\Delta_{Le} \equiv \Delta_{L}^{(2)} & < & \Delta_{L}^{(n)}  < \Delta_{L}^{\infty} \le \Delta_{U}^{\infty} 
                            <  \Delta_{U}^{(n)}  < \Delta_{L}^{(2)} \equiv \Delta_{Ue}.
\eeas 
Hence the queue difference $\Delta (0)$ associated with any periodic equilibrium and all limit points of the sequence $\Delta^{(n)} (0)$
necessarily lie in the interval $[\Delta_{L}^{\infty}, \Delta_{U}^{\infty}]$.
\end{coro}

We cannot expect that $\Delta_{L}^{\infty} = \Delta_{U}^{\infty}$ because the bounds were created by ignoring some terms.

\section{Stronger Notions of Convergence and Stability}\label{secGeom}

In Lemma \ref{thEndlessAprox} we showed that for any $\kappa$ and $\tau$ we can find $\mu_*$, such that the iterative algorithm
for the approximating system acts as a map from the space $\SS_\mu$ in \eqref{SSmu} into itself, thus ensuring that the algorithm can be iterated
indefinitely.
We now use Lemma \ref{thEndlessAprox} and its proof to show that
the iterative algorithm in \S \ref{secPfde} converges geometrically fast to the point
$\Delta^a_*$ on the periodic equilibrium, when $u^a_* \in \SS_\mu$.
The fast monotone convergence to equilibrium is seen also in the numerical experiments in \S \ref{secNumeric}.

\begin{theorem}{$($geometric rate of convergence$)$.} \label{thGeometric}
Fix $c \in (0, 1-\tau)$ and consider $\mu \le \mu_*$, for $\mu_*$ in Lemma \ref{thEndlessAprox}.
Consider the solution $x^a$ to the approximating system for a given initial condition $\Delta(0) = \Delta^{(0)} \in \SS_\mu$.
Then for any $\rho \in (0,1)$ there exists a $\mu_{**} \le \mu_*$ such that, for all $\mu \le \mu_{**}$ and $\delta_\mu$ in \eqref{delta_mu},
\bes
|\Delta^{(k)} - \Delta^a_*| \le \frac{\rho^k}{1-\rho} |\Delta^{(1)} - \Delta^{(0)}| \le \delta_\mu \frac{\rho^k}{1-\rho}.
\ees
In particular, $x^a_3$ converges to $u^a_*$ geometrically fast in the number of cycles.
\end{theorem}
Note that the statement of the theorem implies that there exists a unique asymptotically-stable periodic equilibrium in $\SS_\mu$,
as we already know.

\begin{proof}
For any $\mu \le \mu_*$, $\T$ maps $\SS_\mu$ into itself by Lemma \ref{thEndlessAprox}, in which case,
for any $\Delta_1, \Delta_2 \in \SS_\mu$, \eqref{Tmap} gives
\bequ \label{rateConv}
\bsplit
|\T(\Delta_1) - \T(\Delta_2)| & = \frac{1-\mu}{\mu} e^{\frac{1-\mu+\kappa}{1+\mu}} |e^{-\Delta_1/(1+\mu)} - e^{-\Delta_2/(1+\mu)}| \\
& \le \frac{1-\mu}{\mu} e^{\frac{1-\mu+\kappa}{1+\mu}} e^{-\frac{1-\mu}{\mu}(1-c)+\kappa} \frac{1}{1+\mu}|\Delta_1 - \Delta_2|.
\end{split}
\eeq
The inequality follows because, for $g(\Delta) \equiv e^{-\Delta/(1+\mu)}$,
$$|\dot{g}(\Delta)| \le K \equiv \frac{1}{1+\mu} e^{-\frac{1-\mu}{\mu}(1-c) + \kappa}, \quad \Delta \in \SS_\mu \equiv [\Delta_\mu^M - \delta_\mu, \Delta_\mu^M],$$
for $\delta_\mu$ in \eqref{delta_mu}, 
implying that $g(\cdot)$ is Lipschitz continuous with a best Lipschitz constant that is no larger than $K$ over the domain $\SS_\mu$.

The RHS of the inequality in \eqref{rateConv} clearly decreases to $0$ as $\mu \da 0$ for any two fixed $\Delta_1$ and $\Delta_2$.
Hence, for any $\rho \in (0,1)$ we can find $\mu_{**}$ small enough, such that $|\T(\Delta_1) - \T(\Delta_2)| < \rho |\Delta_1 - \Delta_2|$
for all $\mu \le \mu_{**}$.
In particular, if $\mu \le \mu_{**}$, then $\T$ is a contraction mapping from the compact interval $\SS_\mu$ into itself.

Let $\T^{(k)}$ denote the $k^{th}$ iteration of the map \eqref{Tmap}, i.e., $\T^{(k)} \equiv \T \circ \cdots \circ \T$,
where the composition map $\circ$ is taken $k$ times. Then $\T^{(k)}(\Delta^{(0)}) = \Delta^{(k)}$, $k \ge 1$,
and the claim follows from the Banach fixed point theorem.
\end{proof}

By Lemma \ref{lmAsyCycle}, the three-dimensional solution $x^a_3$ to \eqref{ODEapprox} ``spirals'' toward $u^a_*$.
Using Theorem \ref{thGeometric},
we next prove a stronger result, stating that the rate of convergence of an oscillating solution to the approximating system (in continuous time) is exponential.


Let $\sP^a_*$ denote the image of the periodic equilibrium $u^a_*$;
\bes
\sP^a_* \equiv \{\gamma \in \SS^a : \gamma = u^a_*(t), ~ 0 \le t < \Sigma_4^*\},
\ees
where $\SS^a$ in \S \ref{secApproxDS} is the state space of the approximating system.
Recall that the convergence of $x^a_3$ to $u^a_*$ holds under the Skorohod metric defined in \S \ref{secPfde}.
\begin{theorem}{$($exponential stability$)$} \label{thExpStable}
Under the conditions of Theorem \ref{thGeometric} $u^a_*$ is exponentially stable, i.e., there exist constants $\vartheta, \beta > 0$ such that
\bes
\inf_{u \in \sP^a_{*}}\|x^a_3(\lm(t)) - u\| < \vartheta e^{-\beta t}, \quad t \ge 0,
\ees
where $\lm(\cdot)$ is a homeomorphism of $[0,t]$ satisfying $\lm(\Sigma^{(k)}_0) = \Sigma^{*(k)}_0$ for all $k \ge 1$
such that the $k^{\rm th}$ cycle falls in $[0,t]$.
\end{theorem}

\begin{proof}
It follows from Lemma \ref{lmAsyCycle} and Theorem \ref{thGeometric} that, for all $k \ge 1$ and $t > \Sigma^{(k)}_*$,
\bes
\|x^a_3(\lm(t)) - u^a_*(t)\| < \|x^a_3(\lm(\Sigma^{(k)}_0)) - u^a_*(0)\|
\le \frac{\|x^a_3(\Sigma^{(0)}_0) - u^a_*(0)\|}{1-\rho} e^{k\log{(\rho)}}. 
\ees
Since $x^a_3$ and $u^a_*$ are uniformly bounded from above by $\Delta_\mu^M$ in \eqref{DelaBd},
the upper bound in \eqref{Ta1bd} together with \eqref{Ta2} give
\bes 
\Sigma^{(k)}_2 - \Sigma^{(k)}_0 = T^{(k)}_1 + T^{(k)}_2 < \frac{\Delta_\mu^M - 1 + \mu - \kappa}{1+\mu} + 1 + \frac{\log (1/\tau)}{\mu} \equiv R,
\ees
so that $\Sigma^{(k)}_4 - \Sigma^{(k)}_0 < 2R$, for all $k \ge 1$.
In particular, the length of any full cycle of any possible solution, including the periodic equilibrium,
is smaller than $2R$.
Since $\|x^a_3(\Sigma^{(0)}_0) - u^a_*(0)\| \le \delta_\mu$, for $\delta_\mu$ in \eqref{delta_mu}, the statement of the theorem follows by taking
\bes
\vartheta \equiv \delta_\mu/(1-\rho) \qandq \beta \equiv -\log(\rho)/2R.\qedhere
\ees
\end{proof}

In ending we remark that the exponential bound on the rate of convergence to $u^a_*$ should in general depend on the initial condition,
as seen in the proof of Theorem \ref{thExpStable}. In particular, exponential stability
should in general be defined via $\|x^a_3(t) - u^a_*(t)\| < \vartheta \|x^a_3(0) - u^a_*(0)\| e^{-\beta t}$ for $\beta, \vartheta > 0$.
However, we obtain the bound in the statement of the theorem since all the solutions we consider have values in $\SS_\mu$,
and are therefore uniformly bounded.

\section{Asymptotic Results for the Stochastic Model} \label{secAsymptotic}

The focus of the paper is on a fluid approximation for the stochastic X model under FQR-ART.
In this section we prove that the switching fluid model arises as a many-server heavy-traffic fluid limit when a fluid-scaled sequence of these stochastic systems is considered.
The proof of the {\em functional weak law of large numbers} (FWLLN)
is given in \S \ref{secFWLLN}, but we first expand on the stochastic model and many-server scaling in \S \ref{secReview}.
We emphasize that, unlike the fluid limit proved in \cite{PW13},
the proof of the FWLLN here is standard because it does not include the stochastic averaging principle.

\subsection{More on the Stochastic Model and Heavy-Traffic Scaling}\label{secReview}

We now briefly expand on the review of the stochastic model, which was described in \S \ref{secModel}, and the heavy-traffic scalings.
We consider a Markovian model, i.e., we assume that both arrival processes are independent (time-homogeneous) Poisson processes,
and that service times, as well as patience times of customers waiting in queue, are exponentially distributed.
Specifically, we assume that the class-$i$ arrival rate in system $n$ is $\lm^n_i$,
a class-$i$ customer receives an exponentially-distributed service time in pool $j$ with mean $1/\mu_{i,j}$,
and a class-$i$ customer has exponentially distributed patience with mean $1/\theta_i$, $i,j = 1,2$. Customers who do not enter
service before running out of patience will abandon the queue. (There is no abandonment from service.)
All random variables are independent of each other and of the two arrival processes.
Since FQR-ART is a Markovian control, in that the routing and scheduling decisions are a function of the state of the system
and are independent of its history, it is easy to see that $X^n$ in \eqref{Xn}
is a six-dimensional time-homogeneous CTMC.

Due to abandonment of waiting customers, defining overloads is not entirely straightforward because a service pool can
be considered normally loaded even if the traffic intensity to that pool is larger than $1$. Our definition of overloads
is taken from an asymptotic perspective.
In particular, pool $i$ is considered overloaded if $\rho_i > 1$, where
$$\rho_i \equiv \lim_{n\tinf} \rho^n_i \equiv \lim_{n\tinf} \lm^n_i/(\mu_{i,i}m^n_i), \quad i = 1,2.$$
On the other hand, we can have $\rho_i \le 1$ with class $i$ overloaded because there are many shared customers in pool $i$.
This latter type of overload may be intentional, if sharing is deemed beneficial and is employed to alleviate an overload in the other class,
or it may be caused by a harmful execution of the control, namely it is due to congestion collapse.

For any fixed $n$ we must take $k^n_{1,2}$ to be sufficiently large so as to ensure that sharing begins only when the corresponding pool
is genuinely overloaded due to a high arrival rate.
In addition, $\tau^n_{1,2}$ should be sufficiently small to ensure that there is
only a negligible amount of simultaneous two-way sharing. (Simultaneous sharing can occur because the direction of overload switches.)
On the other hand, $\tau^n_{1,2}$ must be sufficiently large to be hit in a reasonable time.
We refer to \S\S 2.2 and 3.2 in \cite{PW14} for elaborations on the reasonings behind the way we choose the thresholds.
For our purposes here we simply enforce the following scaling assumption:

\begin{assumption}{$($scaling parameters$)$} \label{assHT}
For strictly positive numbers $m_i$, $\lm_i$, $k_{i,j}$ and $\tau_{i,j}$, $i,j = 1,2$,
$$m^n_i/n \ra m_i, \quad \lm^n_{i}/n \ra \lm_i, \quad k^n_{i,j}/n \ra k_{i,j} \qandq \tau^n_{i,j}/n \ra \tau_{i,j} \qasq n \tinf.$$
\end{assumption}
Note that the first two limits in this assumption put us in the many-server heavy-traffic framework.
The assumption that $\tau_{i,j} > 0$ will be relaxed for the approximating system for the fluid limit. See also \ref{secRescale} below.


\subsection{The FWLLN} \label{secFWLLN}

Paralleling \eqref{T1T2}, we define for each $n \ge 1$
\bes 
\T^n_1 \equiv \inf\{t \ge 0 : Q^n_2(t) - r Q^n_1(t) \le \kappa^n\} \qandq \T^n_2 \equiv \inf\{t \ge 0 : Z^n_{2,1}(\T^n_1 + t) = \tau^n \}.
\ees
We also defined stopping times $T^n_3$, $T^n_4$ and $\Sigma^n_i$, $1 \le i \le 4$ corresponding to the remaining
holding times and switching times in \eqref{switchTimes}.

Let
\bes 
\Sigma^n_q  := \inf\{t \ge 0 : \min\{Q^n_1(t), Q^n_2(t)\} = 0\} \qandq \Sigma_q := \inf\{t \ge 0 : \min\{q_1(t), q_2(t)\} = 0\}.
\ees 
As before, $\inf (\phi) \equiv \infty$.
Since FQR-ART is non-idling, there cannot be any idleness in the system as long as both queues are strictly positive, i.e., if both queues
are initially positive, then
\bes
Z^n_{1,1}(t) + Z^n_{2,1}(t) = Z^n_{2,2}(t) + Z^n_{1,2}(t) = n \qforallq t \le \Sigma^n_q.
\ees

\paragraph{Notation.}
To present our results, we need to introduce some basic notation and refer to \cite{W02} for background.
For $d \ge 1$, let $\D_d[0,t]$ denote the space of real-valued and right continuous $\RR_d$-valued functions on an interval $[0,t] \subseteq \RR_+$
that have limits from the left everywhere, endowed with the usual $J_1$ Skorohod topology.
Let $\C_d[0,t] \subset \D_d[0,t]$ denote the (sub)space of $\RR_d$-valued continuous functions defined on $[0,t]$.
Recall that the $J_1$ topology is equivalent to the uniform topology in $\C_d(I)$ for any compact interval $I$.
We use $\Ra$ to denote convergence in distribution.
We let $e$ denote the identity function, $e(t) = t$, and $a\wedge b \equiv \min\{a,b\}$.
Finally, we add a `bar' to any fluid-scaled element (process or random variable), e.g., $\barx^n \equiv X^n/n$.

\begin{theorem}{$($FWLLN$)$} \label{thFWLLN}
If $\barx^n(0) \Ra x(0)$ in $\RR_6$\ for some deterministic element $x(0) \in \RR_6$\ satisfying Assumption \ref{assInit1}, then
$$\barx^n \Ra x \qinq \D_6[0,  \Sigma_4 \wedge \Sigma_q \wedge t] \qasq n\tinf, \qforallq t \ge 0,$$
where $x$ is a deterministic element of $\C_6$\ and is the unique solution to the switching ODE $\dot{x} = f_\sigma(x)$, for $f_\sigma$ in \eqref{SwitchODE}.
Moreover,
$$n^{-1}(\T^n_i, \Sigma^n_i, \Sigma^n_q ; 1 \le i \le 4) \Ra (T_i, \Sigma_i, \Sigma_q ; 1 \le i \le 4) \qinq \RR_9 \qasq n\tinf,$$
with $+\infty$ being a possible value as a limit of these stopping times.
\end{theorem}
By $+\infty$ being a possible value, e.g., $\Sigma^n_q \Ra +\infty$, we mean that $P(\Sigma^n_q > M) \ra 1$ as $n \tinf$ for all $M > 0$.

Note that, if $x(0)$ satisfies Assumption \ref{assInit1}, then necessarily $\Sigma_q > 0$.
If, in addition, the fluid model is in the invariant set $\sO$, then the convergence can be extended in an obvious way to any compact interval of $[0, \infty)$
because $\Sigma_q \equiv \infty$.
Otherwise, $\Sigma_q < \infty$ and since $\lm < 1$, class-$i$ fluid will stop flowing to pool $j$, $i \ne j$. Since $P(|\Sigma^n_q - \Sigma_q| > \ep) \ra 0$
as $n\tinf$ (recall that convergence in distribution is equivalent to convergence in probability when the limit is deterministic),
this show that sharing of customers will end at approximately time $\Sigma_q$ in a large stochastic system.

The proof of Theorem \ref{thFWLLN} follows standard pre-compactness arguments, combined with applications of the continuous-mapping theorem.
We again refer to \cite{W02} for the general framework.
We therefore start by representing the sample paths of $X^n$ in terms of independent Poisson processes; see \cite{PTW07}.

To simplify notation, let
\bes 
\bsplit
\mA^n_{1,2}(s) \equiv \{\{D^n_{1,2}(s) > 0\} \cap \{Z^n_{2,1}(s) \le \tau^n\}\} \qandq
\mA^n_{2,1}(s) \equiv \{\{D^n_{2,1}(s) > 0\} \cap \{Z^n_{1,2}(s) \le \tau^n\}\},
\end{split}
\ees 

\begin{lemma}{$($martingale representation of $X^n$$)$} \label{lmMartgRep}
If $\min\{Q^n_1(0), Q^n_2(0)\} > 0$, then on the random interval $[0, \Sigma^n_q]$,
\bequ \label{RepXstoc}
\bsplit
Q^n_1(t) & = M^n_1(t) + \lm t - \int_0^t\theta Q^n_1(s)ds - \int_0^t \1_{\mA^n_{1,2}(s)} \left(Z^n_{1,1}(s) + \mu Z^n_{1,2}(s) + \mu Z^n_{2,1}(s) + Z^n_{2,2}(s)\right)ds \\
& \quad - \int_0^t (1 - \1_{\mA^n_{1,2}(s)} - \1_{\mA^n_{2,1}(s)})
\left(Z^n_{1,1}(s) + \mu Z^n_{2,1}(s) \right) ds, \\
Q^n_2(t) & = M^n_2(t) + \lm t - \int_0^t \theta Q^n_2(s) ds - \int_0^t \1_{\mA^n_{2,1}(s)} \left(Z^n_{1,1}(s) + \mu Z^n_{1,2}(s) + \mu Z^n_{2,1}(s) + Z^n_{2,2}(s)\right) ds \\
& \quad - \int_0^t (1 - \1_{\mA^n_{1,2}(s)} - \1_{\mA^n_{2,1}(s)})
\left(Z^n_{2,2}(s) + \mu Z^n_{1,2}(s) \right) ds, \\
Z^n_{1,2}(t) & = M^n_{1,2}(t) + \int_0^t \1_{\mA^n_{1,2}(s)} Z^n_{2,2}(s) ds - \int_0^t (1 - \1_{\mA^n_{1,2}(s)})\mu Z^n_{1,2}(s) ds, \\
Z^n_{2,1}(t) & = M^n_{2,1}(t) + \int_0^t 1_{\mA^n_{1,2}(s)} Z^n_{1,1}(s) ds - \int_0^t (1 - \1_{\mA^n_{2,1}(s)}) Z^n_{2,1}(s)) ds, \\
Z^n_{1,1}(t) & = n - Z^n_{2,1}(t), \\
Z^n_{2,2}(t) & = n - Z^n_{1,2}(t),
\end{split}
\eeq
where $M^n_i$ and $M^n_{i,j}$, $i,j = 1,2$, are square-integrable martingales.
\end{lemma}
The expressions for all martingale terms in \eqref{RepXstoc} can be inferred from \eqref{martg} below.
They are not presented explicitly since, as will be argued in the proof of Theorem \ref{thFWLLN} below,
they are asymptotically negligible under fluid scaling, and therefore play no role in the fluid limit.

\begin{proof}
We use independent unit-rate Poisson processes to represent each of the component processes in \eqref{RepXstoc}.
For example, the representation of $Q^n_1$ over $[0, \Sigma^n_q]$ is
\bes
\bsplit
Q^n_1(t) & = N^a_{1} \left(\lm^n_1 t\right) - N^u_1 \left(\theta_1 \int_0^t Q^n_1(s) ds \right) \\
& \quad - N^+_{1} \left(\int_0^t \1_{\mA^n_{1,2}(s)} \left(\mu_{1,1} Z^n_{1,1}(s) + \mu_{1,2} Z^n_{1,2}(s) + \mu_{2,1} Z^n_{2,1}(s) + \mu_{2,2} Z^n_{2,2}(s)\right) ds \right) \\
& \quad - N^-_1 \left(\int_0^t (1 - \1_{\mA^n_{1,2}(s)} - \1_{\mA^n_{2,1}(s)}) \left(\mu_{1,1} Z^n_{1,1}(s) + \mu_{2,1} Z^n_{2,1}(s) \right) ds\right),
\end{split}
\ees
where $N^a_1, N^u_1, N^+_1$ and $N^-_1$ are mutually independent unit rate (homogeneous) Poisson processes.

Next, we exploit the fact that each of the Poisson processes in \eqref{RepXstoc} minus its random intensity function constitutes a square-integrable martingale
by Lemma 3.2 in \cite{PTW07}, e.g.,
\bequ \label{martg}
M^{n,u}_1 \equiv N^u_1 \left(\theta_1 \int_0^t Q^n_1(s) ds \right) - \theta_1 \int_0^t Q^n_1(s) ds
\eeq
is a square-integrable martingale.
Thus, subtracting and then adding all the random intensities of the Poisson processes,
and using the fact that a sum of martingales is again a martingale, we achieve the representation
in the statement for $Q^n_1$ over the said interval. The representations for the other processes follow similar arguments.
\end{proof}

\begin{proof}[Proof of Theorem \ref{thFWLLN}]
Minor adjustments to the proof of Theorem 5.2 (and Corollary 5.1) in \cite{PW13} give that $\{\barx^n : n \ge 1\}$ is $\C$-tight
in $\D_6$ with all limits being almost-everywhere differentiable. Those modifications to the aforementioned proof
are straightforward, and are therefore omitted.

Next, by Doob's martingale inequality, the fluid-scaled martingales in \eqref{RepXstoc} are asymptotically negligible, namely,
$\barm^n_i \Ra 0e$ and $\barm^n_{i,j} \Ra 0e$ in $\D$, $i,j = 1,2$, since these martingales are square integrable.

Given the initial condition, we have $\1_{\mA^n_{1,2}(s)} = 0$ and $\1_{\mA^n_{2,1}(s)} = 1$ over the interval $[0, \T^n_1 \wedge \Sigma^n_q)$.
Since any limit point of $\barx^n$ is continuous, we must have that $P(\T^n_1 \wedge \Sigma^n_q > \ep) \ra 1$ for some $\ep > 0$.
Therefore, it is easy to see from the representation of $\barx^n$ with the indicator functions being constants
over the interval $[0, \ep)$, that any limit point of $\barx^n$ satisfies
to the integral version of the ODE's in \eqref{poolsT1} and \eqref{queuesT1}, whose unique solution implies that $\barx^n$
converges to that solution $x$ over $[0, \ep)$.

If $T_1 < \Sigma_q$, then the initial interval of convergence can be extended to $[0, T_1)$, and by Theorem 13.6.4 in \cite{W02}, it holds that
$\T^n_1 \Ra T_1$ in $\RR$ as $n \tinf$. Moreover, we have $\barx^n(T_1) \Ra x(T_1)$, so that
$\1_{\mA^n_{1,2}(s)} = \1_{\mA^n_{2,1}(s)} = 0$ over the interval $[\T^n_1, (\T^n_1 + \T^n_2)\wedge \Sigma^n_q)$ implies that
\bes
\lim_{n\tinf} P(\1_{\mA^n_{1,2}(s)} = \1_{\mA^n_{2,1}(s)} = 0\; ;\; s \in (T_1, \Sigma_2 \wedge \Sigma_q) = 1.
\ees
Once again, plugging the constant values of the indicator functions to the representation \eqref{RepXstoc} shows
that any limit point of $\barx^n$ satisfies the integral version of the ODE's in \eqref{poolsT2} and \eqref{queuesT2},
whose unique solution on $[T_1, (T_1 + T_2) \wedge \Sigma_q)$ implies convergence of the sequence $\barx^n$ to $x$.
Moreover, we again have $\T^n_2 \Ra T_2$ in $\RR$ as $n \tinf$.
Since $T_1$ and $T_2$ are deterministic, joint convergence of $(\T^n_1, \T^n_2)$ to $(T_1, T_2)$ holds in $\RR_2$ (e.g., Theorem 11.4.5 in \cite{W02}),
so that $\T^n_1 + \T^n_2 \equiv \Sigma^n_2 \Ra \Sigma_2$ in $\RR$ as $n\tinf$.

The weak convergence of $\barx^n$ to $x$ and $\Sigma^n_i$ to $\Sigma_i$ can be extended to any compact subinterval of $[0, \Sigma_4 \wedge \Sigma_q]$
by exactly the same arguments. If $\Sigma_q > \Sigma_4$ we can then take $x(\Sigma_4)$ as a new initial condition and continue the proof inductively
for all compact subinterval of $[0, \Sigma_q)$.
\end{proof}




\subsection{WLLN for Stationary Distributions} \label{secWLLNstat}

Since for each fixed $n \ge 1$ $X^n$ is clearly an irreducible and positive recurrent CTMC,
it possesses a unique stationary distribution which is also its limiting distribution. Hence, for some random variable $X^n(\infty)$ with values in $\RR_6$
\bes
X^n(t) \Ra X^n(\infty) \qasq t\tinf.
\ees

The uniform convergence over compact intervals of $\barx^n$ to $x$ in Theorem \ref{thFWLLN} implies that,
if the fluid limit of $\barx^n$ experiences oscillations, then $X^n$ will itself oscillate for a long time intervals when $n$ is large.
Only after the oscillations end will $X^n$ start approaching its stationary distribution.
It follows that the convergence to stationarity of large systems with oscillating fluid limits can be exceptionally slow, as we rigorously
show \S \ref{secImplications} below.

We now prove a {\em weak law of large numbers} (WLLN) for the sequence $\{\barx^n(\infty) : n \ge 1\}$.
regardless of the initial condition and the possible fluid limits. In particular, the sequence of ``fluid-scaled'' stationary distributions
converges to the stationary point $x^*_0$ with no sharing,
even if $\sO \ne \phi$, i.e., the fluid limit may not converge to its stationary point $x^*_0$.

\begin{theorem}{$($WLLN for stationary distributions$)$}\label{thWLLNstat}
$\barx(\infty) \Ra x^*_0$, i.e., for each continuous and bounded function $f: \RR_6 \ra \RR$,
\bes
\lim_{n\tinf} \lim_{t\tinf} E[f(\barx^n(t))] = f(x^*_0).
\ees
\end{theorem}

Note that taking the limits in Theorem \ref{thWLLNstat} in the reverse order, namely, first taking $n \tinf$ and then taking $t\tinf$,
is not possible when $\sO$ is not empty, because the limit of $x(t)$ as $t\tinf$ does not exist for all initial conditions.
We therefore cannot prove Theorem \ref{thWLLNstat} using standard arguments,
as were laid out in the proof of Theorem 4 in \cite{HW81}.

\begin{proof}
For each $n\ge 1$ consider the CTMC $X^n$ initialized with its stationary distribution, namely, $X^n(0) \deq X^n(\infty)$, $n \ge 1$.
The sequence $X^n(\infty)$ is tight in $\RR_6$ because each sequence of elements in the vector $\barx^n$ is tight in $\RR$.
This follows immediately for $\barz^n_{i,j}(0)$, which are bounded from below by $0$
and from above by some $c > 1$, $i,j = 1,2$.
Tightness of $\barq^n_1(0)$ and $\barq^n_2(0)$ follows from the infinite-server stochastic-order bound on the queues in Lemma A.5 in \cite{PW13}.
In particular, $\barq_i^n \le_{st} \barq^n_{i,bd}$ pathwise, where $Q^n_{i,bd}$ is the number-in-system process in an $M/M/\infty$
queue with arrival rate $\lm^n_i$ and service rate $\theta$. See also the proof of Theorem \ref{thExpErg} where a similar
bound is constructed.

By Theorem \ref{thFWLLN}, the sequence of processes $\{\barx^n : n \ge 1\}$ is tight in $\D_6$, and we can therefore consider
a converging subsequence of processes, whose initial conditions $\barx^{n'}(0) \deq \barx^{n'}(\infty)$ also converge to some limit
$$\barx(0) \equiv (\barq_i(0), \barz_{i,j}(0); i,j = 1,2) \qinq \RR_6.$$

Since the initial condition is distributed according to the stationary distribution of $\barx^n$, each of the CTMC's in the prelimit is stationary,
and it follows that any limit of $\barx^n$ must also be stationary process. In particular,
\bes 
\barz_{i,j}(t) \deq \barz_{i,j}(0) \qforallq t \ge 0 \qandq~ (i,j) = (1,2) \qorq~ (i,j) = (2,1).
\ees 
First observe that, if $\barz_{1,2}(0) = \barz_{2,1}(0) = 0$ and $\barq_i(0) < \kappa$ w.p.1,
then the two pools and their associated queues operate as two independent underloaded $M/M/m_i$ systems and therefore $\barx(0) = x^*_0$ w.p.1,
implying that $\barx^n(\infty) \Ra x^*_0$.

It follows from the routing rules of FQR-ART that for any sample path for which both $\barz_{1,2}(0)$ and $\barz_{2,1}(0)$ are strictly positive,
at least one of these processes must be strictly decreasing over some interval $(0, \ep)$, $\ep > 0$, contradicting the stationarity of $\barx$.
Therefore, if $\barz_{i,j}(0) > 0$, then $\barz_{j,i}(0) = 0$, $i \ne j$ w.p.1.

Assume, for example, that $P(\barz_{1,2}(0) > 0) > 0$. Then there exists a measurable set $B_{1,2}$ in the underlying probability space,
such that all the sample paths in $B_{1,2}$ have $\barz_{1,2}(0) > 0$ and $\barz_{2,1}(0) = 0$.
Now, if $d_{1,2}(0) \ne 0$, where
$$d_{1,2}(t) \equiv \barq_1(t) - r \barq_2(t) - \kappa,$$
then $\barz_{1,2}$ is strictly increasing or strictly decreasing over some right neighborhood of $0$, because $d_{1,2}$ is necessarily continuous
by Theorem \ref{thFWLLN}. Hence, $d_{1,2}(t) = 0$, so that $q_1(t) \ge \kappa$ w.p.1 for all $t \ge 0$. In turn, $\barz_{1,1}(t) = m_1$ w.p.1 for all $t \ge 0$.
However, this is impossible, because $\lm_1 < \mu_{1,1}m_1$, so that $\barq_1(t)$ must be strictly decreasing if $\barq_1(0) > 0$.
It follows that $P(B_{1,2}) = 0$.
Symmetric arguments give that $P(\barz_{2,1}(0) > 0) = 0$ as well.

It follows that, if $\barq_i(0) > 0$, then $\barq_i$ must be strictly decreasing on some right neighborhood of $0$, because $\barz_{i,i}(0) = m_i$.
Hence, $\barq_i(0) = 0$.
Then the X model is asymptotically two independent $M/M/n+M$ systems with service rate equals to $1$ and arrival rate $\lm^n < n$.
Paralleling \eqref{yConv}, we conclude that $\barx(0) = x^*_0$ w.p.1, so that $\barx^{n'}(\infty) \Ra x^*_0$ as $n' \ra \infty$.
The statement of the theorem follows because the converging subsequence we considered was arbitrary.
\end{proof}

\section{Implications for the Control of the Stochastic System}\label{secImp}

\subsection{Rescaling the Thresholds} \label{secRescale}

\paragraph{Implications to the Activation Thresholds.}
As indicated in Assumption \ref{assHT}, the activation thresholds are asymptotically positive in fluid scale.
This requires us to consider extreme cases with small abandonment rates and service rates for shared customers.
In the {\em worst case} (leading to the biggest buildup of queues) the abandonment rate is strictly smaller than the service rate of shared customers
(and both are small). Formally,

For a given stochastic system there is freedom in choosing how to model the scaling of the thresholds.
It is important that this freedom leads to ambiguities that must be accounted for.
For example, if for $n = 100$, $m^n_1 = m^n_2 = 100$ and we take $k^n_{1,2} = k^n_{2,1} = 10$,
then we can think of the activation thresholds as being equal to $0.1n$ or $\sqrt{n}$. From the fluid perspective, there are important
difference between the two scalings. If the latter holds, then $\kappa = 0$ so that $\SS_{1,2} = \SS_{2,1}$ and the fluid model can cross from $\SS_{1,2}^-$
to $\SS^+_{2,1}$, and vice versa, in zero time. In this case, chattering and oscillations, as defined above, coincide, and are
clearly more likely to occur.
In particular, this suggests that oscillations can occur in the stochastic system even if
a fluid approximation with $\kappa > 0$ does not oscillate at all,
because a more appropriate approximation for the given system would be to assume that $\kappa = 0$;
see Remark \ref{remUnstableStatPt} below.

\paragraph{Implications to the Release Thresholds.}
There are important inconsistencies regarding the rescaling of the release thresholds.
For example, in a system having $100$ agents in each pool and arrival rate $\lm^n = 98$, we may take $\tau^n_{i,j}=3$.
With these parameters, and regardless of the value of $\mu$, pool $j$ is clearly not overloaded at time $t$ if $Z^n_{i,j}(t) \le \tau^n$,
and the fluctuations of the queue must therefore be considered to be of order $o(n)$.
However, the fluctuations of the queue will often be larger than $\tau^n$, which is considered to be asymptotically positive under fluid scaling.
Specifically, whereas
$$\|Q^n\|_T/\tau^n \Ra 0 \qasq n\tinf, \qforallq T > 0, \quad \mbox{where} \quad \|Q^n\|_T \equiv \sup_{0 \le t \le T}Q^n(t),$$
we have $\|Q^n\|_T >> \tau^n$ for any reasonable value of $n$ (which is not unrealistically large) and over intervals $[0, T]$,
with $T = O(1)$ (e.g, $T \approx 1/\mu_{1,1}$.)
It follows that, {\em relative to the stochastic fluctuations}, it is appropriate to think of the release thresholds as being $o(n)$
(even $O(1)$!). On the other hand, from a fluid-limit perspective, $\tau^n$ must satisfy Assumption \ref{assHT}, namely be
strictly positive asymptotically in fluid scale,
since otherwise $\barz^n_{i,j} := Z^n_{i,j}/n$ will not be hit this threshold in finite time when it is strictly decreasing; see \S 3.2 in \cite{PW14}.

We can think of the release thresholds as having a duality property in the fluid model:
When $z_{i,j} \le \tau$ their affect on the system's performance is negligible, and we can consider them to be $0$, i.e., $\tau^n_{i,j} = o(n)$.
Whenever $z_{i,j} > \tau$ and is decreasing, we must think of $\tau$ as being strictly positive, so that $\tau^n$ is as in Assumption \ref{assHT},
to ensure that $z_{i,j}$ can hit $\tau$ in finite time.
We take advantage of this duality property when constructing an approximation for the fluid model in \S \ref{secApproxT1}.

\subsection{Other Implications of the Results to Stochastic Systems} \label{secImplications}

We now provide rigorous results that show the implications of the fluid analysis to the prelimit processes.
Theorems \ref{thEndless}, \ref{thFWLLN} and \ref{thWLLNstat}
suggest that the state space of the {\em irreducible} CTMC $X^n$ is nearly decomposable into two regions when $\sO \ne \phi$.
In particular, the chain may spend a long time in one region before eventually moving to the second region.
For example, if $\barx^n(0) \approx x(0) \in \sO$ for $n$ large, then $X^n$ will approximately track the fluid trajectory
with that initial condition. The oscillations of $\barx^n$ can continue for arbitrarily large time periods as $n$ increases.

On the other hand, if $X^n$ is initialized with no sharing and no queues, then hitting the activation thresholds
is a rare event asymptotically, and oscillations will not begin for a long time. However, the chain being irreducible, must eventually visit a state
in an ``oscillating region'' for the CTMC, triggering oscillations that, as explained in the paragraph above, will take a long time before finally ending,
if $n$ is large.


To make this discussion rigorous,
consider a sequence of initial conditions $\{X^n(0) : n \ge 1\}$ such that $\barx^n(0) \Ra x(0) \in \sO$ as $n\tinf$.
Since $\barx^n \Ra x$ uniformly over compact intervals, and $x$ is oscillating, we see that for any fixed $t > 0$
we can find $N$ large enough, such that
\bequ \label{tvDist}
\|\barx^n(t) - \barx^n(\infty)\|_{tv} > \ep, \quad \mbox{for all $n > N$ and for some $\ep > 0$,}
\eeq
where $\|\cdot\|_{tv}$ denotes the total-variation norm (here given in terms of the random variables
instead of their distributions); see, e.g., \cite{Durrett91}. 
In particular, despite the fact that $\barx^n(t) \Ra \barx^n(\infty)$ as $t\tinf$ for any given $n$,
and moreover, the convergence rate to stationarity is exponentially fast as we show below,
the convergence rate to stationarity can be arbitrarily slow for a sufficiently large system.

To see that \eqref{tvDist} indeed holds for all $n$ large enough, note that convergence in total variation implies
convergence in distribution (the two notions of convergence are in fact equivalent on countable state spaces).
We can use the L\'{e}vy metric to measure distances between random variables corresponding to convergence in distribution.
Specifically, we let the distance between two random variables $X$ and $Y$
with respective cumulative distribution functions $F_X$ and $F_Y$, be
\bes
d_L(X,Y) \equiv d_L(F_X, F_Y) \equiv \inf\{\ep > 0 : F_X(x-\ep) - \ep \le F_Y(x) \le F_X(x+\ep) + \ep \mbox{~ for all } x\}.
\ees
Then, for random variables $Y$ and $\{Y^n : n \ge 1\}$, $Y^n \Ra Y$ is equivalent to $d_L(Y^n, Y) \ra 0$,
and as mentioned above, if $\|Y^n-Y\|_{tv} \ra 0$, then $d_L(Y^n,Y) \ra 0$ as $n \tinf$.

Now, take the contradictory assumption to \eqref{tvDist}, namely assume that there exists a time $t > 0$, such that
\bes
\|\barx^n(t) - \barx^n(\infty)\|_{tv} < \ep \quad \mbox{for all $n \ge 1$ and $\ep > 0$}.
\ees
Then for this specific time $t$ and for all $n$ large enough, we have by the triangular inequality that
\bes
\bsplit
d_L(x(t), x^*_0) & \le d_L(x(t), \barx^n(t)) + d_L(\barx^n(t), \barx^n(\infty)) + d_L(\barx^n(\infty), x^*) < 3\ep.
\end{split}
\ees
where the second inequality follows from Theorem \ref{thFWLLN}, our contradictory assumption and Theorem \ref{thWLLNstat},
and the above holds for any fluid trajectory, regardless of the initial condition.
Hence, $x^*_0$ is globally asymptotically stable, in contradiction to the assumption that $x(0) \in \sO$.

The fact that $X^n$ may converge extremely slowly to stationarity for large $n$ is not entirely straightforward,
because $X^n$ is an exponentially ergodic CTMC, for each $n \ge 1$, and therefore considered to converge ``fast''.

\begin{theorem} \label{thExpErg}
Fix $n \ge 1$. Then for any initial condition $k \in \ZZ_+^6$, there exist positive constants $M_k$ and $\af$ (where $M_k$ depends on the initial state $k$
and $\af$ does not), such that
\bequ \label{expErg}
\|X^n(t) - X^n(\infty)\|_{tv} \le M_k e^{-\af t}
\eeq
\end{theorem}

\begin{proof}
Consider the queue process $Q^n_{bd} := \{Q^n_{bd}(t) : t \ge 0\}$ in an $M/M/\infty$
system that has arrival rate $2\lm^n$ and service rate $\theta$.
Then $Q^n_{bd}$ is distributed the same as the sum of the two queues in the X system in which the service process
is ``shut off'' so that all the output from the two queues is due to abandonment.
Specifically, we construct the $X$ model and the $M/M/\infty$ system on the same probability space by giving both
the same initial condition and the same Poisson arrival processes (exploiting the fact that a superposition of two independent Poisson processes is a Poisson process
with the sum of the rates). If $Q^n_{\Sigma}(t) = Q^n_{bd}(t)$ and there is an abandonment from $Q^n_{\Sigma}$, then we can generate an abandonment
from $Q^n_{bd}$; see, e.g., \cite{W81}.
Therefore, $Q^n_{bd}$ is never below $Q^n_{\Sigma}$.

It is well-known that the Markovian infinite-server queue is exponentially ergodic, see, e.g., Proposition 7.2 in \cite{Robert-Book}.
However, we need to show that this implies that the same holds for $X^n$. We thus use the exponential drift condition on the generator
of $X^n$ whose state space is
$$\Xi \equiv \ZZ_+^2 \times \{0,1,\dots m^n\}^4.$$
For $x \in \Xi$, let $V(x) := (1+\gamma)^{x_1 + x_2}$, for some $\gamma > 0$ which is characterized below.
For $Q^n_{bd}$ we consider the corresponding function $U(q) = (1+\gamma)^q$, $q = x_1 + x_2$.
Then $V: \RR_6 \ra [1,\infty)$ is a {\em norm-like} function, namely $V(x) \ra \infty$ as $\|x\| \tinf$ (we use the standard norm on $\RR_6$).
Similarly, $U: \RR \ra [1,\infty)$ is a norm-like Lyapunov function for the generator of $Q^n_{bd}$.

Due to the sample-path stochastic order relation between $Q^n_\Sigma$ and $Q^n_{bd}$, we have $\mathcal{Q} V \le \mathcal{Q}_{bd} U$,
where $\mathcal{Q}$ denotes the generator matrix of $X^n$ and $\mathcal{Q}_{bd}$ denotes the generator matrix of $Q^n_{bd}$.
Now, if we show that, for some compact set $C \subset \Xi$, the following exponential drift condition holds
\bes
\mathcal{Q}_{bd} U \le -c V + d \1_C,
\ees
for strictly positive constants $c$ and $d$ and $\gamma$, then the statement of the theorem will follow from Theorem 2.5 in \cite{KM03},
because $\mathcal{Q} V \le \mathcal{Q}_{bd} U.$

To that end, we recall that the off-diagonal components of $\mathcal{Q}_{bd}$ are given by
\bes
q_{i,i+1} = 2\lm^n, \quad q_{i,i-1} = k\theta, \qandq q_{i,j} = 0 \qforq |i-j| > 1, \qquad i \ge 1.
\ees
Then for $k \ge 1$
\bes
\bsplit
(\mathcal{Q}_{bd} U)(k) & = -\theta k \gamma [(1+\gamma)^{k-1} - (1+\gamma)^k] + 2\lm^n [(1+\gamma)^{k+1} - (1+\gamma)^k] \\
& = -\gamma (1+\gamma)^{k-1} (\theta k - 2\lm^n(1+\gamma)).
\end{split}
\ees
The RHS in the above display is negative for all states $k$ satisfying $\theta k - 2\lm^n(1+\gamma) > 0$, or equivalently,
\bequ \label{stateForDrift}
k > \frac{2\lm^n}{\theta} (1+\gamma).
\eeq
If $2\lm^n/\theta \notin \ZZ_+$, then we can always choose $\gamma > 0$ small enough such that \eqref{stateForDrift} holds for all
$k \notin C \equiv \{0,1,\dots, \lceil 2\lm^n/\theta \rceil\}$.
Otherwise, if $2\lm^n/\theta$ is an integer, we can simply make $C$ larger, e.g., take $C \equiv \{0,1,\dots, 2\lm^n/\theta + 1\}$,
so that \eqref{stateForDrift} holds for any state $k \notin C$ if $\gamma < \theta/2\lm^n$.
\end{proof}

\begin{remark}{\em
In general, the exponential drift condition in the above proof should hold for a ``small set'' $C$; see, e.g., \cite{KM03}.
In a discrete state space, as is the case here, any compact set is small.
}\end{remark}

\begin{remark}{\em
Instead of working with $X^n$ we can prove Theorem \ref{thExpErg} for all $n$ simultaneously by bounding
the fluid-scaled sequence $\{\barx^n : n \ge 1\}$ by a single $M/M/\infty$
queue having arrival rate $a := 2\lm + \xi$, for some $\xi > 0$ such that $2\lm^n/n < a$ for all $n \ge 1$.
With that proof, we show that $\barx^n$, and therefore $X^n$, are all exponentially ergodic.
One would then hope that, due to the uniform bound on all CTMC's, $\{\barx^n : n \ge 1\}$ are also uniformly ergodic in $n$,
i.e., that there exist constants $M$ and $\af$, such that \eqref{expErg} holds with those constants for all $n$.
However, $\{\barx^n : n \ge 1\}$ is clearly not uniformly ergodic due to \eqref{tvDist}.
The uniform ergodiciy fails to hold because the small set $C$ in the proof is increasing with $n$ and is therefore not {\em uniformly small} as in
Definition 8.1 in \cite{PW13}.
}\end{remark}

\end{appendix}

\end{document}